\documentclass[11pt,a4paper]{article}

\usepackage{cmap}
\usepackage[T1]{fontenc}
\input{glyphtounicode}
\usepackage[utf8]{inputenc}
\usepackage[margin=2.6cm]{geometry}
\usepackage{amsmath,amssymb,amsthm,mathtools}
\usepackage{booktabs,array,enumitem}
\usepackage{aliascnt}
\usepackage{xcolor}
\definecolor{linkblue}{RGB}{0,60,170}
\usepackage[colorlinks=true,linkcolor=linkblue,citecolor=linkblue,urlcolor=linkblue]{hyperref}
\hypersetup{
  pdftitle={Type-II Error Bounds for Test Supermartingales
            from Lower-Tail Hypotheses},
  pdfauthor={Patrick Forre},
  pdfsubject={Anytime-valid hypothesis testing},
  pdfkeywords={e-value, e-variable, test supermartingale, safe testing,
               anytime-valid inference, type-II error, error exponent,
               concentration inequality, information projection}}
\usepackage[capitalize,noabbrev]{cleveref}

\allowdisplaybreaks
\newcommand{\doi}[1]{\href{https://doi.org/#1}{\texttt{doi:}\nolinkurl{#1}}}
\newcommand{\E}{\mathbb{E}}
\newcommand{\Prob}{\mathbb{P}}
\newcommand{\R}{\mathbb{R}}
\newcommand{\N}{\mathbb{N}}
\newcommand{\Xc}{\mathcal{X}}
\newcommand{\Bs}[1]{\mathcal{B}_{#1}}
\newcommand{\Fc}{\mathcal{F}}
\newcommand{\Hc}{\mathcal{H}}
\newcommand{\KL}{\mathrm{KL}}

\newcommand{\Var}{\operatorname{Var}}
\newcommand{\kl}{\mathrm{kl}}
\newcommand{\evar}{\textsf{e}-variable}
\newcommand{\evars}{\textsf{e}-variables}
\newcommand{\epower}{\textsf{e}-power}
\newcommand{\etm}{test supermartingale}

\newcommand{\lp}{\left(}
\newcommand{\rp}{\right)}
\newcommand{\lB}{\left[}
\newcommand{\rB}{\right]}
\newcommand{\lC}{\left\{}
\newcommand{\st}{\;\middle|\;}
\newcommand{\rC}{\right\}}

\newtheorem{theorem}{Theorem}[section]
\newcommand{\aliasthm}[3]{%
  \newaliascnt{#1}{theorem}%
  \theoremstyle{#3}\newtheorem{#1}[#1]{#2}%
  \aliascntresetthe{#1}%
  \crefname{#1}{#2}{#2s}\Crefname{#1}{#2}{#2s}}

\theoremstyle{plain}
\aliasthm{proposition}{Proposition}{plain}
\aliasthm{lemma}{Lemma}{plain}
\aliasthm{corollary}{Corollary}{plain}
\aliasthm{definition}{Definition}{definition}
\aliasthm{assumption}{Assumption}{definition}
\aliasthm{hypothesis}{Hypothesis}{definition}
\aliasthm{example}{Example}{definition}
\aliasthm{remark}{Remark}{remark}
\crefname{theorem}{Theorem}{Theorems}
\Crefname{theorem}{Theorem}{Theorems}

\title{\bfseries Type-II Error Bounds for Test Supermartingales\\[3pt]
from Lower-Tail Hypotheses}
\author{Patrick Forr\'e\\[+10pt]
\small{AI4Science Lab}\\[0pt]
\small{Korteweg-de Vries Institute for Mathematics}\\[0pt]
\small{University of Amsterdam}\\[0pt]
\small{\texttt{p.d.forre@uva.nl}}}
\date{}

\begin{document}
\maketitle

\begin{abstract}
\noindent
In safe hypothesis testing with \etm{}s, Ville's inequality provides
anytime-valid type-I error guarantees for every significance level
$\alpha\in(0,1]$, if one rejects the null hypothesis whenever the wealth process
first exceeds $\frac1\alpha$.

Due to an inherent asymmetry, the type-II error does not have such guarantees: a
heavy concentration of the probability on the lower tail of the log-increments
can lead to one catastrophic bet that undoes any amount of accumulated evidence.
This paper studies how different hypotheses on those lower-tail probabilities
lead to different bounds on the type-II error of the sequential test. They all
reduce to one \emph{master inequality}, which bounds the type-II error at level
$\alpha$, at a fixed horizon and sequentially, in terms of a one-sided Legendre
transform of the (inverse-)moment generating function of the \evars{}, evaluated
at one number: the amount by which the lower bound of the accumulated \epower s
exceeds $\log\frac1\alpha$. And, the step is lossless, in the sense, that it
extracts exactly a constrained information projection.

Every bound presented here is a corollary, obtained by a certain majorant of the
above function. The hypotheses are: a finite negative moment; an exponentially
small crash probability with a moment on the winning side; a wealth floor with a
conditional variance, and its Bernstein variant, which interpolates between a
Gaussian regime set by the variance and an exponential one set by the scale; a
sub-Gaussian or bounded-tilt lower tail; bounded log-increments; and i.i.d.\
increments, where the majorant is the truth. We also provide an
empirical-Bernstein variant.

Each hypothesis may either be read as a condition on the \evars{} one has, or as
the price of betting with an approximation to the likelihood ratio rather than
the ratio itself, which satisfies the weakest condition for free.


(secondary).
\end{abstract}

\tableofcontents

\section{Introduction}
\label{sec:intro}

A \etm{} is a betting scheme against a null hypothesis. Observations arrive on
a filtered space $(\Omega,\Fc,(\Fc_t)_{t\in\N_0})$, indexed by
$\N_0=\{0,1,2,\dots\}$, carrying two probability
measures, the null $\Prob_0$ and the alternative $\Prob_1$; at step $n$ one
wagers on the next observation with a \emph{conditional \evar{}} $E_n$, an
$\Fc_n$-measurable nonnegative random variable with
$\E_{\Prob_0}[E_n\mid\Fc_{n-1}]\le1$, and the accumulated wealth
$W_t=\prod_{n\le t}E_n$ is the evidence. Writing
\[
   \tau_\alpha\ :=\ \inf\lC t\in\N\st W_t>\tfrac1\alpha\rC
\]
for the \emph{rejection time} --- the first moment at which the wealth passes
the threshold $\frac1\alpha$ --- rejecting at $\tau_\alpha$ is valid at level
$\alpha$ at every stopping time simultaneously, by Ville's inequality, and that
guarantee costs nothing beyond the defining property of a conditional
\evar{}. This asymmetry --- type-I control for free, type-II control not --- is
the starting point of both this paper and its companion \cite{CEIL}.

\subsection*{\texorpdfstring{\epower{}}{e-power} is not the answer}

The scalar the safe-testing literature attaches to a design is its \epower{}
$\E_{P_1}[\log E]$, the growth rate of the wealth under the alternative
\cite{Sha21,GHK24,RGVS23}; here and below $P_1$ denotes the law of one
observation when $\Prob_1$ is a product measure, and $E$ the \evar{} that
generates the bets. It is the right quantity for the \emph{mean}
rejection time, by Wald's identity. It is not a type-II guarantee. The
companion paper constructs, for every level $\alpha$, every horizon $t$, every
$c>0$ and every $\eta\in(0,1)$, a pair of hypotheses and a betting sequence
whose conditional \epower{} equals $c$ at every step and whose sequential
type-II error at horizon $t$ is at least $1-\eta$ \cite[Thm.~3.7]{CEIL}. No
inequality of the form ``\epower{} $\ge c$ implies
$\bar\gamma_t(\alpha)\le\Phi(c,\alpha,t)$'' --- where
$\bar\gamma_t(\alpha)=\Prob_1(\tau_\alpha>t)$ is the probability of not having
rejected by time $t$ --- can hold with $\Phi<1$, uniformly over testing
problems.

Something must therefore be assumed about how badly a single bet can crash. The
subject of this paper is the exchange rate: which hypothesis on the lower tail
of $\log E_n$ buys which exponent, and how much each one loses.

\subsection*{One reduction, applied nine times}

The whole paper rests on \Cref{thm:master}, which is three lines long. Suppose
one can certify a drift,
\[
   \E_{\Prob_1}\lB\log E_n\mid\Fc_{n-1}\rB\ \ge\ c_n ,
\]
and a majorant for the anchored conditional cumulant generating function,
\[
   \log\E_{\Prob_1}\lB e^{-s(\log E_n-c_n)}\mid\Fc_{n-1}\rB\ \le\ \psi_n(s) ,
\]
with $c_n$ and $\psi_n$ \emph{predictable} --- known at step $n-1$, and
otherwise arbitrary. Aggregate them into deterministic $C_t$ and $\Psi^{(t)}$
with $C_t\le\sum_{n\le t}c_n$ and $\Psi^{(t)}\ge\sum_{n\le t}\psi_n$
$\Prob_1$-a.s., and write
$\ell=\log\frac1\alpha$ for the log-threshold and
$\Delta_t=(C_t-\ell)_+$ for the \emph{deviation}, the amount by which the
certified drift has overshot the threshold by time $t$. Then, for the
sequential error $\bar\gamma_t(\alpha)=\Prob_1(\tau_\alpha>t)$ --- the
probability of not having rejected by $t$ --- and the fixed-horizon error
$\gamma_t(\alpha)=\Prob_1(W_t\le\frac1\alpha)$,
\[
   \bar\gamma_t(\alpha)\ \le\ \gamma_t(\alpha)\ \le\
   \exp\lp-\lp\Psi^{(t)}\rp^{\star}\!\lp\Delta_t\rp\rp ,
\]
where
\[
   \lp\Psi^{(t)}\rp^{\star}(x)\ :=\ \sup_{s\ge0}
      \lC sx-\Psi^{(t)}(s)\rC
\]
is the one-sided Legendre transform of the aggregated majorant: one-sided
because the tilt $s$ ranges over $[0,\infty)$ only, the Markov step requiring
$x\mapsto x^{-s}$ to be nonincreasing, and negative tilts being useless in any
case (\Cref{lem:positive}).
A conditional Chernoff bound, the tower property, Markov's inequality. In
particular no product structure and no independence are used, the filtration is
arbitrary, and nothing per step has to be a number: only the aggregate does,
because only the aggregate appears in the bound.

Three things make this more than bookkeeping. First, it is \emph{lossless}:
\Cref{prp:duality} shows that the exponent it extracts from a constraint on a
conditional mean equals the value of a constrained information projection,
\[
   \sup_{s\ge0}\lC-sc-\log\E_{\Prob_1}\lB E^{-s}\rB\rC
   \ =\ \inf\lC\KL(Q\|\Prob_1)\st\E_Q[\log E]\le c\rC ,
\]
where the infimum is over all probability measures $Q\ll\Prob_1$ and
$\KL(Q\|\Prob_1)=\E_Q[\log\frac{dQ}{d\Prob_1}]$ is the relative entropy: the
right-hand side is the cheapest way, measured in relative entropy, of moving
from $\Prob_1$ to a law under which the certified drift $c$ fails. The step from the hypothesis to the exponent is
therefore an identity, not an estimate, and every loss in what follows is
attributable to the majorant and to nothing else. (A matching \emph{lower}
bound on the error itself is a separate matter, available only at rung~6, by
Cram\'er's theorem.) Second, it comes
with a time-uniform form (\Cref{thm:master}(d)--(e)) that every rung inherits.
Third, it makes the effect of the level exact. In the homogeneous case, where
each step certifies the same drift $c$ and admits the same majorant $\psi$,
whatever the hypothesis,
\[
   \gamma_t(\alpha)\ \le\ \alpha^{-s^{\ast}(c)}\,e^{-t\,\psi^{\star}(c)}
\]
at every horizon past $\ell/c$ (\Cref{prp:price}); here
$\psi^{\star}(c)=\sup_{s\ge0}\{sc-\psi(s)\}$ is the per-observation rate
function at the full drift and $s^{\ast}(c)$ is a tilt attaining that supremum.
So the asymptotic rate is the rate function at the \emph{full} drift, and the
level enters only through the optimising tilt --- the price per nat of
threshold.

\subsection*{The ladder}

\Cref{sec:ladder} then substitutes majorants. Each rung states a hypothesis,
reads off a rate function and an optimising tilt, and is a corollary of
\Cref{thm:master}:
\begin{center}
\small
\begin{tabular}{@{}l>{\raggedright\arraybackslash}p{10.4cm}@{}}
\toprule
\Cref{sec:rung1} & one negative moment,
   $\E_{\Prob_1}[E_n^{-s_0}\mid\Fc_{n-1}]\le\rho_n$ at one fixed tilt $s_0>0$\\
\Cref{sec:rung2} & a polynomially small crash probability, together with a
   one-sided $p$-th moment ($2$), an $L\log L$ moment ($2'$) or a two-sided
   second moment ($2''$) on the winning side\\
\Cref{sec:rung3} & a wealth floor and a conditional variance ($3$), or
   Bernstein moments ($3'$)\\
\Cref{sec:rung4} & a sub-Gaussian lower tail ($4$), or a bounded tilted
   variance ($4'$)\\
\Cref{sec:rung5} & bounded log-increments, in an Azuma--Hoeffding ($5a$) and an
   exact binary-entropy ($5b$) form\\
\Cref{sec:rung6} & i.i.d.\ increments, where the majorant is the truth and the
   exponent is exact\\
\bottomrule
\end{tabular}
\end{center}
\Cref{tab:ladder} collects them. Two structural facts are worth stating in
advance. The rungs do \emph{not} form a chain: rungs~2 and~4 are incomparable,
because a sub-Gaussian lower tail is silent about the winning side, and that is
not a technicality --- it is why the moment conditions of rung~2 are needed at
all (\Cref{rem:notachain}). And, among the fixed-horizon bounds, rung~1 is the only
one with a critical horizon: its majorant is linear at the origin, so its bound
is empty, not merely loose, until $t$ exceeds $s_0\ell/(s_0c-\kappa)$, where
$\kappa=\log\rho+s_0c\ge0$ measures how far the assumed moment bound $\rho$ is
from the value conditional Jensen already forces; whereas
rungs~2--6 become informative as soon as the certified drift passes the
threshold (\Cref{rem:vacuity}). In the time-uniform form the tilt is fixed in
advance and rung~$2'$ acquires a critical horizon too (\Cref{rem:rung2prime}).

\subsection*{What the hypotheses are about}

It would be a misreading to take \Cref{sec:ladder} as a list of regularity
conditions on the testing problem. Every rung constrains the
\emph{approximation} to $R=\frac{dP_1}{dP_0}$ that one chooses to bet with, and
the exact likelihood ratio satisfies the weakest of them for nothing, since
\[
   \E_{P_1}\lB R^{-1}\rB=P_0(R>0)\ \le\ 1 .
\]
What the ladder measures is the cost of \emph{not} betting the likelihood ratio
--- of using a learned score, a plug-in estimate, a misspecified model --- and
the cost is paid in the crash direction only. The hypotheses can moreover be
enforced rather than assumed: clipping an arbitrary \evar{} from below buys
rungs~1 and~2 (\Cref{ex:clipping}), clipping from both sides with the same
parameter buys rung~5 for \emph{any} \evar{} whatsoever (\Cref{ex:twosided}),
and a bounded stake lands on rung~5 by construction (\Cref{ex:betting}). The
design question is not whether a rung holds but which rung to buy, and at what
price in \epower{}.

\subsection*{Relation to the companion paper, and scope}

The two papers answer complementary questions. This one asks what exponent can
be \emph{certified} from a hypothesis; \cite{CEIL} asks how large the exponent
can \emph{be}, and shows that over all \evars{} it is bounded by
$\KL(P_0\|P_1)$ --- the Chernoff--Stein exponent, in the opposite direction to
the divergence $\KL(P_1\|P_0)$ that caps the \epower{} --- with the bound tight
and not attained. The two meet at rung~6 and nowhere else, and neither is used
in proving the other; the present paper is self-contained. Notation is shared
between the two throughout.

\Cref{sec:observable} gives the one rung whose constant can be observed rather
than asserted; \Cref{sec:composite} and \Cref{sec:secondorder} record, briefly,
what survives for a composite alternative and what the level costs.
\Cref{sec:discussion} closes with what the reduction does not do. All proofs
are in \Cref{app:proofs}; the classical inequalities used are collected in
\Cref{app:concentration}, restated in the notation used here.

\section{Setup and notation}
\label{sec:setup}

Throughout, $\log$ is the natural logarithm, $\log0:=-\infty$, $e^{-\infty}:=0$,
$0\cdot\infty:=0$, and for $x\in[0,\infty]$ we set $x^{0}:=1$ and, for $s>0$,
$0^{-s}:=+\infty$ and $(+\infty)^{-s}:=0$.

\subsection{The filtered setting}
\label{sec:filtered}

The companion paper \cite{CEIL} works with i.i.d.\ observations, because its
results concern the exact exponent and that is an i.i.d.\ notion. Nothing in the
present paper needs a product structure: the master principle uses only the
tower property. We therefore work on an abstract filtered space and recover the
i.i.d.\ case as \Cref{ex:iidsetting}.

\begin{definition}[The setting]\label{def:setting}
Let $(\Omega,\Fc)$ be a measurable space with a filtration
$(\Fc_t)_{t\in\N_0}$, $\Fc_0=\{\emptyset,\Omega\}$, and let $\Prob_0,\Prob_1$ be
probability measures on $(\Omega,\Fc)$ with
\begin{equation}\label{eq:ac}
   \Prob_1\big|_{\Fc_t}\ \ll\ \Prob_0\big|_{\Fc_t}
   \qquad\text{for every }t\in\N_0 .
\end{equation}
A \emph{betting sequence} is a family $(E_n)_{n\in\N}$ of maps
$E_n:\Omega\to[0,\infty]$ such that $E_n$ is $\Fc_n$-measurable
(\emph{adapted}) and
\begin{equation}\label{eq:cevar}
   \E_{\Prob_0}\lB E_n\mid\Fc_{n-1}\rB\ \le\ 1
   \qquad\Prob_0\text{-a.s.}
\end{equation}
Its \emph{wealth process} is the adapted process
\[
   W_0:=1,\qquad W_t:=\prod_{n\le t}E_n\quad(t\in\N) ,
\]
a nonnegative $\Prob_0$-supermartingale \cite{SSVV11,GHK24,RGVS23}, and we write
\[
   L_n:=\log E_n\quad(\Fc_n\text{-measurable}),
   \qquad
   \log W_t=\sum_{n\le t}L_n .
\]
\end{definition}

\begin{example}[The i.i.d.\ case]\label{ex:iidsetting}
\Cref{def:setting} covers the setting of \cite[\S2]{CEIL}: take
$\Omega=\Xc^{\N}$ for a measurable space $(\Xc,\Bs{\Xc})$, let $(X_n)_{n\in\N}$
be the coordinate process, $\Fc_t=\sigma(X_1,\dots,X_t)$, and
$\Prob_i=P_i^{\otimes\N}$ for one-observation laws $P_0,P_1$ on
$(\Xc,\Bs{\Xc})$ with $P_1\ll P_0$; then \eqref{eq:ac} holds. In that case we
write $R:=\frac{dP_1}{dP_0}$. Only \Cref{sec:rung6} and the results quoted from
\cite{CEIL} use this specialisation; everything else is stated in the
generality of \Cref{def:setting}. Note that $(X_n)$ generates the filtration
there, but \Cref{def:setting} does not require the filtration to be generated by
anything: a design may use external randomisation or side information, provided
the bets remain adapted.
\end{example}

\begin{definition}[Level, threshold, rejection time, type-II errors]
\label{def:typeII}
Fix $\alpha\in(0,1)$ and put $\ell:=\log\frac1\alpha\in(0,\infty)$, a constant.
The \emph{rejection time} is the stopping time
$\tau_\alpha:=\inf\lC t\in\N\st W_t>\frac1\alpha\rC$, and the two \emph{type-II
errors} at horizon $t\in\N_0$ are
\[
   \gamma_t(\alpha):=\Prob_1\lp W_t\le\tfrac1\alpha\rp ,
   \qquad
   \bar\gamma_t(\alpha):=\Prob_1\lp\tau_\alpha>t\rp ,
\]
the fixed-horizon and the sequential error.
\end{definition}

Two facts come for free, and we record them here because everything below is
stated against them.

\begin{theorem}[Anytime-valid type-I guarantee; Ville \cite{Vil39}]
\label{thm:typeI}
Under \Cref{def:setting}, $\Prob_0(\tau_\alpha<\infty)\le\alpha$ for every
$\alpha\in(0,1)$.
\end{theorem}

\noindent\emph{Proof.}
$(W_t)_{t\in\N_0}$ is a nonnegative $\Prob_0$-supermartingale with $W_0=1$, so
Ville's inequality (\Cref{thm:app-ville}) gives
$\Prob_0(\sup_{t}W_t\ge\frac1\alpha)\le\alpha$, and
$\{\tau_\alpha<\infty\}\subseteq\{\sup_tW_t\ge\frac1\alpha\}$. \qed

\begin{lemma}[The sequential error dominates the fixed-horizon one]
\label{lem:ordering}
Under \Cref{def:setting}, for every $\alpha\in(0,1)$ and every $t\in\N_0$,
\begin{equation}\label{eq:ordering}
   \bar\gamma_t(\alpha)\ \le\ \gamma_t(\alpha) .
\end{equation}
\end{lemma}

\noindent\emph{Proof.}
On $\{\tau_\alpha>t\}$ the threshold has not been crossed at any time up to $t$,
in particular not at $t$, so $W_t\le\frac1\alpha$; that is
$\{\tau_\alpha>t\}\subseteq\{W_t\le\frac1\alpha\}$. \qed

\noindent
Every bound below is proved for $\gamma_t$ and inherited by $\bar\gamma_t$
through \eqref{eq:ordering}. The ordering supplies no \emph{reverse} inequality,
so a matching lower bound for $\bar\gamma_t$ always needs a separate argument;
\Cref{rem:seqexp} is the one place where that is discussed.

\begin{remark}[Where the null enters]\label{rem:nullenters}
After this subsection the null appears only through the definition of
$\tau_\alpha$ and through \Cref{rem:notevar}. Every bound below is a statement
about the law of $\log W_t=\sum_{n\le t}L_n$ under $\Prob_1$.
\end{remark}

\subsection{The two ingredients}
\label{sec:ingredients}

The whole paper rests on a pair: a certified drift and a certified majorant.
Neither is required to come from an \epower{} bound, and \Cref{rem:notdelta}
below explains why it matters that they need not. Neither is required to be
deterministic either; what must be is their aggregate, and \Cref{def:main}
separates the two roles. We are explicit throughout about which objects are
constants, which are adapted and which are predictable.

\begin{definition}[Drift certificate and majorant]\label{def:main}
Let $(E_n)_{n\in\N}$ be a betting sequence and fix $s_{\max}\in(0,\infty]$.
\begin{enumerate}[label=(\alph*),leftmargin=2.4em]
\item A \emph{drift certificate} is a \emph{predictable} sequence
      $(c_n)_{n\in\N}$ of real random variables --- $c_n$ is
      $\Fc_{n-1}$-measurable --- such that
      \begin{equation}\label{eq:anchor}
         \E_{\Prob_1}\lB L_n\mid\Fc_{n-1}\rB\ \ge\ c_n
         \qquad\Prob_1\text{-a.s., for every }n\in\N .
      \end{equation}
\item Write
      \begin{equation}\label{eq:cgf}
         \varphi_n(s)\ :=\
         \log\E_{\Prob_1}\lB e^{-s(L_n-c_n)}\mid\Fc_{n-1}\rB
         \qquad(s\ge0)
      \end{equation}
      for the \emph{anchored conditional cumulant generating function} of step
      $n$: for each fixed $s$ it is an $\Fc_{n-1}$-measurable random variable,
      hence predictable, and for each fixed $\omega$ it is a convex function of
      $s$ vanishing at $s=0$. A \emph{majorant} is a family
      $(\psi_n)_{n\in\N}$ of maps
      $\psi_n:[0,s_{\max})\times\Omega\to[0,\infty]$, \emph{predictable} in the
      sense that $\psi_n(s)$ is $\Fc_{n-1}$-measurable for each fixed $s$, with
      $\psi_n(0)=0$ and
      \begin{equation}\label{eq:majorant}
         \varphi_n(s)\ \le\ \psi_n(s)
         \qquad\Prob_1\text{-a.s.,}
      \end{equation}
      for every $n\in\N$ and $s\in[0,s_{\max})$.
\item An \emph{aggregate} is a pair of \emph{deterministic} objects --- numbers
      $C_t\in\R$ and maps $\Psi^{(t)}:[0,s_{\max})\to[0,\infty]$ --- such that,
      for every $t\in\N$ and every $s\in[0,s_{\max})$,
      \begin{equation}\label{eq:aggregate}
         \sum_{n\le t}c_n\ \ge\ C_t
         \qquad\text{and}\qquad
         \sum_{n\le t}\psi_n(s)\ \le\ \Psi^{(t)}(s)
         \qquad\Prob_1\text{-a.s.}
      \end{equation}
\end{enumerate}
We write $(C_t,\Psi^{(t)})_{t\in\N}$ for the aggregate, with the conventions
$C_0:=0$ and $\Psi^{(0)}:\equiv0$.
\end{definition}

\begin{remark}[What is random and what is not]\label{rem:whatsrandom}
Nothing \emph{per step} is required to be a number. The left-hand sides of
\eqref{eq:anchor} and \eqref{eq:majorant} are $\Fc_{n-1}$-measurable, and so are
the right-hand sides: $c_n$ and $\psi_n$ may depend on the past in any
measurable way, so each requirement compares two predictable quantities. The
process $(E_n)_{n\in\N}$, and with it $(L_n)_{n\in\N}$ and $(W_t)_{t\in\N_0}$,
is adapted but not predictable.

What must be deterministic is the \emph{aggregate}. The reason is not technical
but definitional: $(C_t,\Psi^{(t)})$ is what appears in the bound, and a bound
has to be a number. The tilt $s$ and the deviation $\Delta_t$ of
\Cref{def:rate} are then constants as well.
\end{remark}

\begin{remark}[Only the combination matters]\label{rem:onlyaggregate}
The proof of \Cref{thm:master} uses the two halves of \eqref{eq:aggregate} only
through the single inequality
\begin{equation}\label{eq:combined}
   s\sum_{n\le t}c_n-\sum_{n\le t}\psi_n(s)
   \ \ge\ s\,C_t-\Psi^{(t)}(s)
   \qquad\Prob_1\text{-a.s.},
\end{equation}
for each $s\in[0,s_{\max})$, so a deficit in the certified drift may be paid for
by a surplus in the majorant and conversely; \eqref{eq:aggregate} is the
convenient sufficient form. In two of the rungs the two halves genuinely cannot be
separated --- rungs~1 and~5(b), the only ones whose majorant involves $c_n$ ---
and there it is \eqref{eq:combined} that must be verified;
\Cref{rem:rung1nodrift} and \Cref{rem:predictable} do so.
\end{remark}

\begin{definition}[Deviation and rate function]\label{def:rate}
Given an aggregate, define the \emph{deviation} and the \emph{rate function}
\[
   \Delta_t:=\lp C_t-\ell\rp_+\ \in[0,\infty) ,
   \qquad
   \lp\Psi^{(t)}\rp^{\star}(x):=\sup_{s\in[0,s_{\max})}
      \lC sx-\Psi^{(t)}(s)\rC
   \quad(x\ge0) ,
\]
the Legendre transform of $\Psi^{(t)}$ restricted to nonnegative arguments.
Both are deterministic.
\end{definition}

\begin{remark}[Reading the two ingredients]\label{rem:reading}
Condition \eqref{eq:anchor} says the log-wealth drifts upward under the
alternative at a certified rate; \eqref{eq:majorant} controls how far below that
drift a single step can fall. The first makes the test eventually reject, the
second says how fast. The normalisation $\psi_n(0)=0$ is automatic from
\eqref{eq:majorant} at $s=0$ and is recorded because it makes
$(\Psi^{(t)})^{\star}\ge0$.
\end{remark}

\section{The master principle}
\label{sec:master}

The argument has three layers. \Cref{lem:onestep} treats one step;
\Cref{thm:master} multiplies the steps; \Cref{cor:iid} specialises to the
i.i.d.\ case, where the bound collapses to $t\,\psi^{\star}(\Delta_t/t)$.
\Cref{prp:duality} then shows that what \Cref{thm:master} produces is not
merely a Chernoff estimate but exactly the value of an information projection. All proofs are in
\Cref{app:proofs}.

\subsection{One step}

\begin{lemma}[One conditional \evar{}]\label{lem:onestep}
Under \Cref{def:main}, for every $n\in\N$ and every $s\in[0,s_{\max})$,
\[
   \E_{\Prob_1}\lB E_n^{-s}\mid\Fc_{n-1}\rB
   \ \le\ \exp\lp-s\,c_n+\psi_n(s)\rp
   \qquad\Prob_1\text{-a.s.}
\]
\end{lemma}

\subsection{The aggregate}

\begin{theorem}[Master principle]\label{thm:master}
Let $(E_n)_{n\in\N}$ be a betting sequence with an aggregate
$(C_t,\Psi^{(t)})$ in the sense of \Cref{def:main}. Then for every $t\in\N$:
\begin{enumerate}[label=(\alph*),leftmargin=2.4em]
\item \emph{(Negative moments of the wealth.)} For every $s\in[0,s_{\max})$,
      \begin{equation}\label{eq:negmom}
         \E_{\Prob_1}\lB W_t^{-s}\rB
         \ \le\ \exp\lp-s\,C_t+\Psi^{(t)}(s)\rp .
      \end{equation}
\item \emph{(Type-II error, sequential and fixed-horizon.)} For every
      $\alpha\in(0,1)$,
      \begin{equation}\label{eq:master}
         \bar\gamma_t(\alpha)\ \le\ \gamma_t(\alpha)
         \ \le\ \exp\lp-\lp\Psi^{(t)}\rp^{\star}\!\lp\Delta_t\rp\rp .
      \end{equation}
\item \emph{(The optimising tilt.)} Suppose $\Psi^{(t)}$ is finite,
      differentiable and strictly convex on $(0,s_{\max})$ with
      $\Psi^{(t)}(0)=0$. If $\Delta_t>\lim_{s\downarrow0}(\Psi^{(t)})'(s)$ then
      the supremum in \eqref{eq:master} is attained at the unique
      $s^{\ast}_t\in(0,s_{\max})$ solving
      \begin{equation}\label{eq:sstar}
         \lp\Psi^{(t)}\rp'(s^{\ast}_t)=\Delta_t ,
      \end{equation}
      provided such a point exists; otherwise it is approached as
      $s\uparrow s_{\max}$. If $\Delta_t\le\lim_{s\downarrow0}(\Psi^{(t)})'(s)$
      the supremum is $0$, attained at $s^{\ast}_t=0$, and \eqref{eq:master} is
      trivial.
\item \emph{(Time-uniform form.)} Fix $s\in(0,s_{\max})$ and
      $\delta\in(0,1)$. With $\Prob_1$-probability at least $1-\delta$,
      simultaneously for all $t\in\N_0$,
      \begin{equation}\label{eq:unif}
         \log W_t\ >\ C_t-\frac{\Psi^{(t)}(s)+\log\frac1\delta}{s} .
      \end{equation}
\item \emph{(Rejection time.)} Consequently, for every $s\in(0,s_{\max})$ and
      $\delta\in(0,1)$,
      \begin{equation}\label{eq:stop}
         \Prob_1\lp\tau_\alpha\le T_{s,\delta}\rp\ \ge\ 1-\delta ,
         \qquad
         T_{s,\delta}:=\inf\lC t\in\N:\
            s\,C_t-\Psi^{(t)}(s)\ \ge\ s\ell+\log\tfrac1\delta\rC .
      \end{equation}
\end{enumerate}
\end{theorem}

\begin{remark}[Fixed tilt versus optimised tilt]\label{rem:fixedtilt}
Parts (a)--(c) optimise over $s$ \emph{inside} a deterministic exponent and lose
nothing by it. Parts (d)--(e) cannot: they rest on a supermartingale built at
one tilt, which must be chosen before the data. Each rung of \Cref{sec:ladder}
therefore comes in two forms, and only the second requires a design choice; we
state \eqref{eq:unif} explicitly only for \Cref{sec:rung4}, where it is
prettiest, and note here that every rung inherits it by substituting its
$\Psi^{(t)}$.
\end{remark}

\begin{remark}[No \evar{} property is used]\label{rem:notevar}
Parts (a) and (b) use only that $(E_n)$ is adapted, nonnegative, and finite
$\Prob_1$-a.s. The conditional \evar{} property
$\E_{\Prob_0}[E_n\mid\Fc_{n-1}]\le1$ is what makes $\tau_\alpha$ a level-$\alpha$
rejection time (\Cref{thm:typeI}), and it delivers the finiteness --- it
forces $E_n<\infty$ $\Prob_0$-a.s., hence $\Prob_1$-a.s.\ --- but it plays no
role in the estimate itself. The master principle is a statement about a
sequence of certified drifts and certified tails; that the sequence happens to
be a betting sequence is what makes the conclusion a statement about a test.
\end{remark}

\begin{remark}[Why $\Delta_t$ is a readout, not a primitive]\label{rem:notdelta}
The pair $(C_t,\Psi^{(t)})$ is what \Cref{thm:master} consumes; $\Delta_t$
enters only when \eqref{eq:negmom} is converted into \eqref{eq:master}. Three
consequences.

First, the drift need not be certified by an \epower{} bound. Any argument
producing a deterministic $C_t$ with $C_t\le\sum_{n\le t}c_n$ $\Prob_1$-a.s.\
will do, and \eqref{eq:anchor} is one convenient sufficient condition among
several.

Second, the bound is trivial exactly when $C_t\le\ell$, and the positive part in
$\Delta_t$ is what truncates a useless negative rate to the trivial bound $1$.
That vacuity is a property of the readout, not of the method, and it is not an
artefact: \cite[Lem.~5.10]{CEIL} exhibits a matching lower bound: for an i.i.d.\ product
with $v:=\Var_{P_1}(\log E)<\infty$ and $c\ell>2v$, a design of \epower{} $c$
has $\bar\gamma_t(\alpha)\ge1-\frac{2v}{c\ell}$ at every horizon
$t\le\frac{\ell}{2c}$, whatever the majorant.

Third, the individual $c_n$ are already allowed to be random
(\Cref{rem:whatsrandom}); what \eqref{eq:aggregate} insists on is a
\emph{deterministic} lower bound for their sum. If even that is unavailable ---
if the best one has is an $\Fc_t$-measurable $\widehat C_t$ --- then bounding
$\widehat C_t$ from below on an event of $\Prob_1$-probability $1-\delta$
yields a version of \eqref{eq:master} with $\delta$ added to the right-hand
side. \Cref{sec:observable} does \emph{not} need this: there the drift is still
asserted, and it is the \emph{variance} that becomes observable.
\end{remark}

\subsection{The i.i.d.\ case}

\begin{corollary}[Product form]\label{cor:iid}
In the situation of \Cref{ex:iidsetting}, let $(E_n)_{n\in\N}$ be i.i.d.\ under
$\Prob_1$ with common law that of a single \evar{} $E$, let
$c:=\E_{P_1}[\log E]$ and let $\psi$ be a majorant for one step, so that
$C_t=tc$ and $\Psi^{(t)}=t\psi$ are an aggregate. Then
\begin{equation}\label{eq:iid}
   \gamma_t(\alpha)\ \le\
   \exp\lp-t\,\psi^{\star}\!\lp\frac{\Delta_t}{t}\rp\rp ,
   \qquad \Delta_t=(tc-\ell)_+ ,
\end{equation}
and when $\psi$ is differentiable and strictly convex the optimising tilt
solves $\psi'(s^{\ast})=\Delta_t/t$.
\end{corollary}

\begin{remark}[The per-observation exponent]\label{rem:periid}
In \eqref{eq:iid} the exponent per observation is $\psi^{\star}(\Delta_t/t)$ and
$\Delta_t/t\uparrow c$, so it tends to $\psi^{\star}(c)$. By Cram\'er's theorem
(\Cref{thm:app-cramer}) the exact exponent of an i.i.d.\ product is the
Chernoff--Stein exponent $\Lambda(E)=\sup_{s\ge0}-\log\E_{P_1}[E^{-s}]$, as
\Cref{cor:rung6} records,
and $\psi^{\star}(c)\le\Lambda(E)$
with equality precisely when $\psi$ may be taken to be the true conditional
cumulant generating function $s\mapsto\log\E_{P_1}[e^{-s(\log E-c)}]$. That is
the content of \Cref{sec:rung6}: the ladder stops losing at its top rung, and
nowhere below it.
\end{remark}

\subsection{Duality: the bound is an information projection}

\Cref{thm:master} would be of limited interest if it were only a Chernoff
estimate, since one could then always ask whether a different argument does
better from the same hypothesis. It is not. The next proposition identifies the
exponent as \emph{equal} to the value of a constrained information projection,
so that the passage from the hypothesis to the exponent loses nothing. The proof
is a two-line application of the Donsker--Varadhan variational formula
\cite{DV75}, quoted as \Cref{thm:app-dv}, together with its equality case.

\begin{proposition}[Duality]\label{prp:duality}
Let $E$ be a random variable on $(\Omega,\Fc,\Prob_1)$ with
$\Prob_1(0<E<\infty)=1$ and $\E_{\Prob_1}[\log E]$ well defined in
$(-\infty,+\infty]$. For $c\in\R$ put
\[
   \mathcal A_c:=\lC Q\ll\Prob_1\st\E_Q[\log E]\le c\rC ,
   \qquad
   \mathcal D(c):=\inf_{Q\in\mathcal A_c}\KL(Q\|\Prob_1)
   \quad(\inf\emptyset:=+\infty) ,
\]
and, with $\Lambda(s):=\log\E_{\Prob_1}[E^{-s}]\in(-\infty,+\infty]$ for
$s\ge0$,
\[
   \mathcal J(c):=\sup_{s\ge0}\lC-sc-\Lambda(s)\rC .
\]
Then
\begin{equation}\label{eq:duality}
   \mathcal J(c)\ =\ \mathcal D(c)
   \qquad\text{for every }c\in\R .
\end{equation}
Write $s_{\max}:=\sup\lC s\ge0\st\Lambda(s)<\infty\rC$ and, for $s$ with
$\Lambda(s)<\infty$, let $Q_s$ be the \emph{tilted law}
$dQ_s:=E^{-s}e^{-\Lambda(s)}\,d\Prob_1$ and
$m(s):=\E_{Q_s}[\log E]$. Then moreover:
Write $\mu:=\E_{\Prob_1}[\log E]$ and
$x_0:=\Prob_1\text{-}\operatorname{ess\,inf}\log E\in[-\infty,\infty)$. Then
moreover:
\begin{enumerate}[label=(\alph*),leftmargin=2.4em]
\item if $c\ge\mu$ then $\mathcal J(c)=\mathcal D(c)=0$, attained at
      $Q=\Prob_1$;
\item if $c<x_0$ then $\mathcal A_c=\emptyset$ and
      $\mathcal J(c)=\mathcal D(c)=+\infty$;
\item if $c=x_0>-\infty$ then
      $\mathcal J(c)=\mathcal D(c)=\log\frac1{\Prob_1(\log E=x_0)}$, attained
      exactly when $\Prob_1(\log E=x_0)>0$, with unique minimiser
      $\Prob_1(\,\cdot\mid\log E=x_0)$;
\item if $x_0<c<\mu$ and $s_{\max}=0$ --- that is, if Cram\'er's condition
      \Cref{hyp:cramer} fails at every tilt --- then
      $\mathcal J(c)=\mathcal D(c)=0$ and the infimum is not attained;
\item if $x_0<c<\mu$ and $s_{\max}>0$ then the common value lies in
      $(0,+\infty]$, and the infimum defining $\mathcal D(c)$ is attained if and
      only if $m(s^{\ast})=c$ for some $s^{\ast}$ with
      $\Lambda(s^{\ast})<\infty$, in which case $Q_{s^{\ast}}$ is the unique
      minimiser and
      \begin{equation}\label{eq:duality-attained}
         \mathcal D(c)=\KL(Q_{s^{\ast}}\|\Prob_1)=-s^{\ast}c-\Lambda(s^{\ast}) .
      \end{equation}
\end{enumerate}
\end{proposition}

\begin{remark}[When the projection is attained, and when it is not]
\label{rem:steep}
Part (e) of \Cref{prp:duality} is the case the rest of this paper lives in, and
its tilt equation $m(s^{\ast})=c$ is solvable for every
$c\in(x_0,\mu)$ as soon as either $s_{\max}=\infty$, or $s_{\max}<\infty$ and
$\Lambda$ is \emph{steep} there, meaning $\Lambda(s)\to\infty$ as
$s\uparrow s_{\max}$; this is the condition \cite[\S2.3]{DZ10} calls essential
smoothness. Exactly one configuration escapes it:
\begin{equation}\label{eq:nonsteep}
   0<s_{\max}<\infty,
   \qquad
   \E_{\Prob_1}\lB\lvert\log E\rvert\,E^{-s_{\max}}\rB<\infty ,
\end{equation}
which says that the extreme tilt $Q_{s_{\max}}$ exists \emph{and} has a finite
mean, so that $m(s_{\max})>-\infty$; for $c<m(s_{\max})$ the infimum is then not
attained. Note that \eqref{eq:nonsteep} is strictly stronger than
$\Lambda(s_{\max})<\infty$ alone: for
$\Prob_1(\log E=-k)\propto e^{-k}k^{-2}$ one has $s_{\max}=1$ and
$\Lambda(1)<\infty$, yet $m(1)=-\infty$, so the tilt equation is solvable at
every $c<\mu$ after all.

When \eqref{eq:nonsteep} does hold and $c<m(s_{\max})$, \eqref{eq:duality} still
holds, but a minimising sequence must mix $Q_{s_{\max}}$ with a vanishing amount
of mass placed arbitrarily far out in the lower tail, and no single $Q$ achieves
the value. This is the exact analogue, at the level of measures, of the
non-attainment in \cite[Prop.~4.24]{CEIL}.
\end{remark}

\begin{remark}[What duality buys]\label{rem:duality}
Read with $c$ the certified drift of one step, \eqref{eq:duality} says the
exponent \Cref{thm:master} extracts from \eqref{eq:anchor} alone is exactly the
cost, in relative entropy, of the cheapest alternative law under which the drift
fails. So the losses in \Cref{sec:ladder} are attributable entirely to the
majorant: they measure how
much a particular tail hypothesis fails to pin down the conditional law, not any
slack in the reduction.
\end{remark}

\section{The ladder}
\label{sec:ladder}

Each subsection below fixes one hypothesis on the lower tail of the increments,
records the majorant it supplies, and reads off the bound. The pattern is
always the same:
\[
   \textbf{hypothesis}\ \Longrightarrow\ \textbf{majorant }\psi_n\
   \Longrightarrow\ \textbf{rate function }(\Psi^{(t)})^{\star}\
   \Longrightarrow\ \textbf{optimising tilt }s^{\ast}_t .
\]
Only the first arrow requires an argument; the other two are calculus. The
arguments are in \Cref{app:proofs} and the classical inequalities behind them in
\Cref{app:concentration}, so that each subsection here can be read on its own.
Throughout, $(c_n)_{n\in\N}$ is a drift certificate and $(C_t,\Psi^{(t)})$ an
aggregate in the sense of \Cref{def:main}, and $\Delta_t=(C_t-\ell)_+$ is the
deviation of \Cref{def:rate}. Each hypothesis below introduces per-step
quantities of its own --- $\rho_n$, $K_n$, $M_{p,n}$, $b_n$, $w_n$,
$\sigma_n^{2}$, $r_n$. They are written as constants for readability, but by
\Cref{rem:whatsrandom} each may be a predictable random variable instead,
provided the deterministic aggregate dominates the corresponding sums
$\Prob_1$-a.s. Two rungs need more care than that, because their majorants
involve $c_n$ itself: \Cref{rem:rung1nodrift} and \Cref{rem:predictable} treat
them.

\begin{remark}[Two readings of a rung]\label{rem:tworeadings}
A rung can be read as a guarantee or as a design rule. As a guarantee: if the
increments satisfy this, the type-II error is at most that. As a design rule:
the hypotheses are conditions on the \emph{approximation} to the likelihood
ratio one chooses to bet with, and every one of them can be enforced by
construction --- by clipping (\Cref{ex:clipping}), by bounding a stake
(\Cref{ex:betting}), or by using a likelihood ratio in the first place
(\Cref{rem:lrfree}). The ladder is therefore a menu, and the question it answers
is which price one is willing to pay in \epower{} for which exponent.
\end{remark}

\subsection{Rung 1: a single negative moment}
\label{sec:rung1}

The weakest hypothesis that yields anything at all is Cram\'er's condition: one
finite negative moment of each \evar{} under the alternative, at a tilt $s_0$
fixed once and for all.

\begin{hypothesis}[One negative moment]\label{hyp:cramer}
There are a constant $s_0\in(0,\infty)$ and constants $\rho_n\in(0,\infty)$
such that
\begin{equation}\label{eq:h1}
   \E_{\Prob_1}\lB E_n^{-s_0}\mid\Fc_{n-1}\rB\ \le\ \rho_n
   \qquad\Prob_1\text{-a.s., for every }n\in\N .
\end{equation}
\end{hypothesis}

\begin{corollary}[Rung 1]\label{cor:rung1}
Assume \Cref{hyp:cramer} and put
\[
   \kappa_n\ :=\ \log\rho_n+s_0c_n\ \in[0,\infty) ,
\]
which is nonnegative by conditional Jensen; let $K_t\in[0,\infty)$ be constants
with $K_t\ge\sum_{n\le t}\kappa_n$. Then the \emph{chord}
\begin{equation}\label{eq:maj-rung1}
   \psi_n(s):=\frac{s}{s_0}\,\kappa_n\quad(0\le s\le s_0),
   \qquad \psi_n(s):=+\infty\quad(s>s_0),
\end{equation}
is a majorant, its aggregate is $\Psi^{(t)}(s)=\frac{s}{s_0}K_t$ on $[0,s_0]$
and $+\infty$ beyond, the rate function is the linear-with-a-kink
\begin{equation}\label{eq:rate-rung1}
   \lp\Psi^{(t)}\rp^{\star}(x)\ =\ \lp s_0x-K_t\rp_+ ,
\end{equation}
the optimising tilt is $s^{\ast}_t=s_0$ when $s_0\Delta_t>K_t$ and
$s^{\ast}_t=0$ otherwise, and
\begin{equation}\label{eq:bound-rung1}
   \bar\gamma_t(\alpha)\ \le\ \gamma_t(\alpha)\ \le\
   \min\lC 1,\ \exp\lp K_t-s_0\Delta_t\rp\rC .
\end{equation}
If the $\rho_n$ are constants, so that $K_t=\sum_{n\le t}\kappa_n$ and
$C_t=\sum_{n\le t}c_n$ are admissible together, the bound reads in its most
transparent form
\begin{equation}\label{eq:bound-rung1b}
   \gamma_t(\alpha)\ \le\ \min\lC 1,\ \alpha^{-s_0}\prod_{n\le t}\rho_n\rC .
\end{equation}
\end{corollary}

\begin{remark}[Rung 1 needs no drift certificate]\label{rem:rung1nodrift}
This is one of the two rungs whose majorant involves $c_n$, so it is one of the
two where the halves of \eqref{eq:aggregate} may \emph{not} be chosen
independently: $\kappa_n$ increases with $c_n$, while \eqref{eq:aggregate} wants
$\sum c_n$ bounded below. What happens is that the drift cancels identically in
the combination \eqref{eq:combined}: for $s\in[0,s_0]$,
\[
   s\sum_{n\le t}c_n-\sum_{n\le t}\psi_n(s)
   \ =\ \frac{s}{s_0}\sum_{n\le t}\log\frac1{\rho_n} ,
\]
so a pair $(C_t,K_t)$ is admissible precisely when
$s_0C_t-K_t\le\sum_{n\le t}\log\frac1{\rho_n}$ $\Prob_1$-a.s.

Read directly, this says that rung~1 consumes \Cref{hyp:cramer} and nothing
else. Indeed $N_t:=W_t^{-s_0}\prod_{n\le t}\rho_n^{-1}$ is a nonnegative
$\Prob_1$-supermartingale with $N_0=1$ by \eqref{eq:h1} alone, since $\rho_t$ is
predictable; so for \emph{any} predictable $(\rho_n)_{n\in\N}$ and any
deterministic $\bar\rho_t$ with $\prod_{n\le t}\rho_n\le\bar\rho_t$
$\Prob_1$-a.s.,
\begin{equation}\label{eq:bound-rung1c}
   \gamma_t(\alpha)\ \le\ \min\lC1,\ \alpha^{-s_0}\bar\rho_t\rC ,
\end{equation}
with no drift certificate at all. The certificate re-enters only to interpret
the bound --- through $\kappa_n\ge0$ and the critical horizon \eqref{eq:crit}.
\end{remark}

\begin{remark}[The rung with a critical horizon]\label{rem:rung1-horizon}
The rate \eqref{eq:rate-rung1} vanishes identically for $x\le K_t/s_0$: below
that gap the bound is not merely loose but empty. In the homogeneous case
$\rho_n\equiv\rho<1$, $c_n\equiv c$ and $\kappa=\log\rho+s_0c$, the per-step
rate is $\log\frac1\rho=s_0c-\kappa>0$ and \eqref{eq:bound-rung1b} is
\[
   \gamma_t(\alpha)\ \le\
   \min\lC1,\ e^{-\lp(s_0c-\kappa)t-s_0\ell\rp}\rC ,
\]
geometric, but trivial before the \emph{critical horizon}
\begin{equation}\label{eq:crit}
   t_0\ :=\ \frac{s_0\ell}{s_0c-\kappa}\ \ge\ \frac{\ell}{c} ,
\end{equation}
the inequality because $\kappa\ge0$. Every later rung
supplies a majorant with $\psi_n(s)=o(s)$ as $s\downarrow0$, whose transform is
strictly positive at every $x>0$; that is the single structural difference
between rung~1 and the rest, and \Cref{rem:vacuity} makes it precise.
\end{remark}

\begin{remark}[A likelihood ratio satisfies this for free]\label{rem:lrfree}
In the i.i.d.\ setting of \Cref{ex:iidsetting} with $E=R=\frac{dP_1}{dP_0}$,
\[
   \E_{P_1}\lB R^{-1}\rB\ =\ \int_{\{p_1>0\}}\frac{p_0}{p_1}\,p_1\,d\mu
   \ =\ P_0(p_1>0)\ \le\ 1 ,
\]
so \Cref{hyp:cramer} holds with $s_0=1$ and $\rho=P_0(R>0)$, with no assumption
whatsoever. The hypotheses of this paper are therefore not assumptions about the
testing problem; they are assumptions about how far the \evar{} one actually
bets with departs from the likelihood ratio. \Cref{sec:discussion} returns to this.
\end{remark}

\begin{remark}[Uniformity in $n$ cannot be dropped]\label{rem:uniformity}
Suppose one knows only $\rho_n<1$ for each $n$. Then $\sum_{n\le t}\log\rho_n$
need not grow linearly: for $\rho_n=e^{-1/n}$ it is $-H_t\sim-\log t$, and
\eqref{eq:bound-rung1b} delivers only $\gamma_t(\alpha)\le\alpha^{-s_0}/t$,
a polynomial rate, although every single \evar{} contributes strictly. A
geometric rate is a statement about the Ces\`aro means of
$\log\frac1{\rho_n}$ staying away from zero, not about each factor being useful.
\end{remark}

\subsection{Rung 2: a small crash probability}
\label{sec:rung2}

Cram\'er's condition has an equivalent form that is directly interpretable:
$\E_{\Prob_1}[E_n^{-s_0}\mid\Fc_{n-1}]<\infty$ for some $s_0>0$ if and only if
the probability that the \evar{} pays less than $\varepsilon$ is polynomially
small in $\varepsilon$ (\Cref{prp:cramer-equiv}). We call the exponent the
\emph{crash exponent}. Rung~1 spends this hypothesis alone and pays for it with
a critical horizon; rung~2 buys the critical horizon back with a single further
moment --- and the moment is needed on the \emph{winning} side.

\begin{hypothesis}[Crash tail and a one-sided $p$-th moment]\label{hyp:crash}
There are constants $a\in(0,\infty)$, $p\in(1,2]$ and $K_n,M_{p,n}\in[0,\infty)$
such that, $\Prob_1$-a.s.\ for every $n\in\N$,
\begin{equation}\label{eq:h2}
   \Prob_1\lp L_n\le c_n-x\mid\Fc_{n-1}\rp\ \le\ K_ne^{-ax}
   \qquad\text{for all }x\ge0 ,
\end{equation}
equivalently $\Prob_1(E_n\le\varepsilon e^{c_n}\mid\Fc_{n-1})\le
K_n\varepsilon^{a}$ for $\varepsilon\in(0,1]$, and
\begin{equation}\label{eq:h2b}
   \E_{\Prob_1}\lB\lp(L_n-c_n)_+\rp^{p}\mid\Fc_{n-1}\rB\ \le\ M_{p,n} .
\end{equation}
\end{hypothesis}

\begin{corollary}[Rung 2]\label{cor:rung2}
Assume \Cref{hyp:crash} and let $K\ge\sum_{n\le t}K_n$ and
$M_p\ge\sum_{n\le t}M_{p,n}$ be constants. Then, with $s_{\max}=a$,
\begin{equation}\label{eq:maj-rung2}
   \psi_n(s)\ :=\ c_p\,M_{p,n}\,s^{p}+\frac{K_n\,s^{2}}{a(a-s)},
   \qquad
   c_p:=\sup_{v>0}\frac{e^{-v}-1+v}{v^{p}}\ \le\ 2^{1-p},
   \quad c_2=\tfrac12,
\end{equation}
is a majorant --- the first term absorbs the winning side, the second the
crash side through the layer-cake formula, as \Cref{lem:maj-crash} shows ---
with aggregate $\Psi^{(t)}(s)=c_pM_ps^{p}+\frac{Ks^{2}}{a(a-s)}$, and
\begin{equation}\label{eq:bound-rung2}
   \bar\gamma_t(\alpha)\ \le\ \gamma_t(\alpha)\ \le\
   \exp\lp-\sup_{0\le s<a}
      \lC s\Delta_t-c_pM_ps^{p}-\frac{Ks^{2}}{a(a-s)}\rC\rp
   \ \le\ \exp\lp-\tfrac12\,s_{\ast}\Delta_t\rp ,
\end{equation}
where the explicit admissible tilt is
\begin{equation}\label{eq:sstar-rung2}
   s_{\ast}\ :=\ \min\lC
      \lp\frac{\Delta_t}{4c_pM_p}\rp^{\frac1{p-1}},\
      \frac{a^{2}\Delta_t}{8K},\
      \frac a2\rC .
\end{equation}
If $K>0$ and $\Delta_t>0$, the exact optimising tilt is the unique
$s^{\ast}_t\in(0,a)$ solving
\begin{equation}\label{eq:sstar-rung2-exact}
   p\,c_pM_p\,s^{p-1}+\frac{K\,s(2a-s)}{a(a-s)^{2}}\ =\ \Delta_t ,
\end{equation}
which has no closed form. The point of \eqref{eq:sstar-rung2} is that it is
explicit; what it costs is a bounded factor in the exponent, namely
\[
   \frac{\sup_{0\le s<a}\lC s\Delta_t-\Psi^{(t)}(s)\rC}
        {\tfrac12s_{\ast}\Delta_t}
   \ \le\ \max\lC4,\ \frac{2(p-1)}{p}\lp\frac4p\rp^{\frac1{p-1}}\rC ,
\]
which is $2$ in the moment-limited regime at $p=2$ and at most $4$ when the
crash term dominates, but which degrades as $p\downarrow1$: it is $4.7$ at
$p=\frac32$ and $137$ at $p=\frac65$. For small $p$ one should optimise
\eqref{eq:sstar-rung2-exact} numerically rather than use $s_{\ast}$.
\end{corollary}

\begin{remark}[The moment-limited regime]\label{rem:rung2-regimes}
Which of the three entries of \eqref{eq:sstar-rung2} is smallest decides the
shape of the bound. If it is the first --- the \emph{moment-limited} case, in
which the crash tail is light enough that
$(\frac{\Delta_t}{4c_pM_p})^{1/(p-1)}\le\min\{\frac{a^{2}\Delta_t}{8K},
\frac a2\}$ --- then \eqref{eq:bound-rung2} reads
\[
   \gamma_t(\alpha)\ \le\
   \exp\lp-\frac{\Delta_t^{\,q}}{2\,(4c_pM_p)^{q-1}}\rp,
   \qquad q:=\frac{p}{p-1}\ \ge\ 2 ,
\]
and for $p=2$ this is $\exp(-\Delta_t^{2}/(4M_2))$: the diffusive shape of the
higher rungs, with a worse constant. If instead the crash term dominates, the
exponent is linear in $\Delta_t$ and the $\Delta_t^{q}$ expression is
\emph{not} a valid bound. The proviso is not cosmetic.
\end{remark}

\begin{remark}[Why a moment on the winning side]\label{rem:winning}
It is the lower tail of $L_n$ that hurts the wealth, and \eqref{eq:h2} already
controls it, so \eqref{eq:h2b} looks gratuitous. It is not. The anchor $c_n$ is
a conditional \emph{mean}: an \evar{} whose \epower{} is generated by rare
enormous payoffs --- a lottery ticket --- has a large $c_n$ that says nothing
about typical behaviour, and its type-II error can be arbitrarily close to one.
Condition \eqref{eq:h2b} is exactly a quantitative statement that the \epower{}
is not concentrated in the far upper tail. \Cref{prp:p1-fails} shows that with
$p=1$ the hypothesis is genuinely insufficient, and that the exponent
$q=\frac{p}{p-1}$ of \Cref{rem:rung2-regimes} is the right one.
\end{remark}

\begin{hypothesis}[Crash tail and a two-sided second moment]\label{hyp:crash2}
There are constants $a\in(0,\infty)$ and $K_n,w_{0,n}\in[0,\infty)$ such that
\eqref{eq:h2} holds and, $\Prob_1$-a.s.\ for every $n\in\N$,
\begin{equation}\label{eq:h2c}
   \E_{\Prob_1}\lB (L_n-c_n)^{2}\mid\Fc_{n-1}\rB\ \le\ w_{0,n} .
\end{equation}
\end{hypothesis}

\begin{corollary}[Rung $2''$]\label{cor:rung2pp}
Assume \Cref{hyp:crash2} and let $K\ge\sum_{n\le t}K_n$ and
$w_0\ge\sum_{n\le t}w_{0,n}$ be constants. Then, with $s_{\max}=a$,
\begin{equation}\label{eq:maj-rung2pp}
   \psi_n(s)\ :=\ \frac{w_{0,n}s^{2}}{2}
      +K_n\sum_{k\ge3}\lp\frac sa\rp^{k}
   \ =\ \frac{w_{0,n}s^{2}}{2}+\frac{K_n\,(s/a)^{3}}{1-s/a}
\end{equation}
is a majorant, with aggregate
$\Psi^{(t)}(s)=\frac{w_0s^{2}}{2}+\frac{K(s/a)^{3}}{1-s/a}$, and
\begin{equation}\label{eq:bound-rung2pp}
   \bar\gamma_t(\alpha)\ \le\ \gamma_t(\alpha)\ \le\
   \exp\lp-\sup_{0\le s<a}\lC s\Delta_t-\frac{w_0s^{2}}{2}
      -\frac{K\,(s/a)^{3}}{1-s/a}\rC\rp .
\end{equation}
If $\Delta_t>0$ and $K>0$ the optimising tilt is $s^{\ast}_t=a\,u^{\ast}$,
where $u^{\ast}\in(0,1)$ is the unique root of
\begin{equation}\label{eq:sstar-rung2pp}
   a\,w_0\,u+\frac Ka\cdot\frac{u^{2}(3-2u)}{(1-u)^{2}}
   \ =\ \Delta_t ,
\end{equation}
and the exponent obeys the two explicit bounds
\begin{align}
   \label{eq:rung2pp-small}
   \text{(small gaps, }\Delta_t\le\tfrac{aw_0}{2}\text{)}\qquad
   &\ \ge\ \frac{\Delta_t^{2}}{2w_0}
      \lp1-\frac{4K\Delta_t}{a^{3}w_0^{2}}\rp ,\\
   \label{eq:rung2pp-large}
   \text{(large gaps, }a\Delta_t>K\text{)}\qquad
   &\ \ge\ a\Delta_t-2\sqrt{K\,a\,\Delta_t}-\frac{a^{2}w_0}{2} .
\end{align}
\end{corollary}

\begin{remark}[What the second moment buys, and the simplified form]
\label{rem:rung2pp}
Compare \eqref{eq:rung2pp-small} with \Cref{rem:rung2-regimes} at $p=2$, which
gives $\Delta_t^{2}/(4M_2)$. Strengthening the one-sided moment \eqref{eq:h2b}
to the two-sided \eqref{eq:h2c} recovers the constant $\frac1{2w_0}$ --- the
same small-gap constant as rungs~3 and~4, and the best any majorant with
$\psi''(0)=w_0$ can give. The mechanism is visible in
\eqref{eq:maj-rung2pp}: the tail hypothesis is fed only into the orders
$k\ge3$, leaving $\psi_n''(0)=w_{0,n}$ untouched, whereas
\eqref{eq:maj-rung2} pays for the winning side already at order $p\le2$.
By \eqref{eq:rung2pp-large} the linear rate $a\Delta_t$ at large gaps is
preserved, and it is the best available, since \Cref{hyp:crash2} guarantees
only $s_{\max}\ge a$.

A closed form is often enough. Put $w:=\max\{w_0,\frac{2K}{a^{2}}\}$. Then, on
$[0,a)$ and with $u=s/a$,
\begin{gather*}
   \Psi^{(t)}(s)\ \le\ \frac{w\,s^{2}}{2\lp1-\frac sa\rp} ,
   \qquad\text{because}\\
   \frac{a^{2}u^{2}}{2}\lp\frac{w}{1-u}-w_0\rp
   \ \ge\ \frac{a^{2}u^{2}}{2}\cdot\frac{wu}{1-u}
   \ \ge\ \frac{a^{2}u^{2}}{2}\cdot\frac{2Ku}{a^{2}(1-u)}
   =\frac{Ku^{3}}{1-u} ,
\end{gather*}
so the right-hand side is the sub-gamma majorant of \Cref{rem:rung3prime} with
$\tau=\frac1a$, and rung~$3'$ applies verbatim:
\begin{gather*}
   \gamma_t(\alpha)\ \le\
   \exp\lp-w\,a^{2}\,h_1\lp\frac{\Delta_t}{w\,a}\rp\rp
   \ \le\ \exp\lp-\frac{\Delta_t^{2}}{2\lp w+\frac{\Delta_t}{a}\rp}\rp ,\\
   s^{\ast}_t=a\lp1-\lp1+\frac{2\Delta_t}{w\,a}\rp^{-1/2}\rp .
\end{gather*}
(The domination is applied to the aggregate, not step by step: in the
inhomogeneous case $\sum_n\max\{w_{0,n},\frac{2K_n}{a^{2}}\}$ can exceed
$\max\{w_0,\frac{2K}{a^{2}}\}$.) Being a pointwise domination of
$\Psi^{(t)}$, this is never sharper than \eqref{eq:bound-rung2pp}; what happens
when $\frac{2K}{a^{2}}>w_0$ is that the small-gap constant degrades from
$\frac1{2w_0}$ to $\frac{a^{2}}{4K}$, and when $\frac{2K}{a^{2}}\le w_0$ the two
agree to second order at the origin.
\end{remark}

\begin{remark}[Rung $2'$: an $L\log L$ moment]\label{rem:rung2prime}
Replacing \eqref{eq:h2b} by the pair
$\E_{\Prob_1}[(L_n-c_n)_+\mid\Fc_{n-1}]\le m_{1,n}$ and
$\E_{\Prob_1}[\theta((L_n-c_n)_+)\mid\Fc_{n-1}]\le M_n$, where
$\theta(y):=y\log^{+}y$, gives for every truncation level $T>1$ the majorant
\[
   \psi_n(s;T)\ =\ \frac{m_{1,n}T}{2}\,s^{2}+\frac{M_n}{\log T}\,s
      +\frac{K_ns^{2}}{a(a-s)},
   \qquad s\in[0,a) ,
\]
and hence, with $m_1,M,K$ dominating the respective sums and any
$\vartheta\in(0,1)$,
\[
   \gamma_t(\alpha)\ \le\ \exp\lp-\tfrac{1-\vartheta}{2}s_\vartheta\Delta_t\rp,
   \qquad
   s_\vartheta:=\min\lC\frac{(1-\vartheta)\Delta_t}{2m_1}
      e^{-M/(\vartheta\Delta_t)},\
      \frac{(1-\vartheta)a^{2}\Delta_t}{8K},\ \frac a2\rC ;
\]
with $\vartheta=\frac12$ and the first entry smallest this is
$\exp(-\frac{\Delta_t^{2}}{16m_1}e^{-2M/\Delta_t})$. The exponent is positive at
every positive gap but exponentially small in $1/\Delta_t$: this is the
borderline of the method, the weakest hypothesis under which a critical horizon
still disappears. Note that $T$ may be tuned to the horizon in the fixed-horizon
bound but not in the time-uniform form \eqref{eq:unif}, which fixes one
supermartingale once; at a fixed $T$ the majorant is linear at the origin and
the critical horizon returns.
\end{remark}

\subsection{Rung 3: a wealth floor and a conditional variance}
\label{sec:rung3}

From here on the majorants vanish \emph{quadratically} at the origin, which is
what produces a diffusive exponent and removes the critical horizon for good.
The three quadratic rungs differ in which second-order quantity one is willing
to name. This one names two: a floor below which a single bet cannot fall, and a
conditional second moment about the anchor. It is the regime of Bennett's,
Bernstein's and --- in the martingale form needed here --- Freedman's
inequalities \cite{Ben62,Ber46,Fre75}.

\begin{hypothesis}[Wealth floor and conditional variance]\label{hyp:freedman}
There are constants $b_n\in(0,\infty)$ and $w_n\in[0,\infty)$ such that,
$\Prob_1$-a.s.\ for every $n\in\N$,
\begin{equation}\label{eq:h3}
   E_n\ \ge\ e^{\,c_n-b_n}
   \quad\text{(equivalently } L_n\ge c_n-b_n),
   \qquad
   \E_{\Prob_1}\lB (L_n-c_n)^{2}\mid\Fc_{n-1}\rB\ \le\ w_n .
\end{equation}
\end{hypothesis}

Both requirements are \emph{anchored}: they involve the certified $c_n$, not the
unknown conditional mean, so they are checkable. If one knows instead a
conditional variance $\Var_{\Prob_1}(L_n\mid\Fc_{n-1})\le v_n$ and an upper
bound $\bar c_n\ge\E_{\Prob_1}[L_n\mid\Fc_{n-1}]$, then
$w_n=v_n+(\bar c_n-c_n)^{2}$ is admissible.

\begin{corollary}[Rung 3]\label{cor:rung3}
Assume \Cref{hyp:freedman} and let $w\ge\sum_{n\le t}w_n$ and
$b\ge\max_{n\le t}b_n$ be constants. Then, with $s_{\max}=\infty$,
\begin{equation}\label{eq:maj-rung3}
   \psi_n(s)\ :=\ \frac{w_n}{b_n^{2}}\lp e^{sb_n}-sb_n-1\rp
\end{equation}
is a majorant --- this is Bennett's bound \cite{Ben62}, in the conditional form
that makes the resulting inequality Freedman's \cite{Fre75}; see
\Cref{app:rung3}, whose aggregation step rests on \Cref{lem:elem} --- with
aggregate $\Psi^{(t)}(s)=\frac{w}{b^{2}}(e^{sb}-sb-1)$, the rate function is
\begin{equation}\label{eq:rate-rung3}
   \lp\Psi^{(t)}\rp^{\star}(x)\ =\ \frac{w}{b^{2}}\,h\lp\frac{bx}{w}\rp
   \ \ge\ \frac{x^{2}}{2\lp w+\frac{bx}{3}\rp},
   \qquad h(z):=(1+z)\log(1+z)-z ,
\end{equation}
the optimising tilt is
\begin{equation}\label{eq:sstar-rung3}
   s^{\ast}_t\ =\ \frac1b\log\lp1+\frac{b\Delta_t}{w}\rp ,
\end{equation}
and
\begin{equation}\label{eq:bound-rung3}
   \bar\gamma_t(\alpha)\ \le\ \gamma_t(\alpha)\ \le\
   \exp\lp-\frac{w}{b^{2}}\,h\lp\frac{b\Delta_t}{w}\rp\rp
   \ \le\ \exp\lp-\frac{\Delta_t^{2}}{2\lp w+\frac{b\Delta_t}{3}\rp}\rp .
\end{equation}
\end{corollary}

\begin{remark}[Rung $3'$: Bernstein moments instead of a floor]
\label{rem:rung3prime}
The floor may be traded for a moment growth condition: if there is a constant
$\tau>0$ with $\E_{\Prob_1}[(L_n-c_n)^{2}\mid\Fc_{n-1}]\le w_n$ and
\[
   \E_{\Prob_1}\lB\lp(c_n-L_n)_+\rp^{k}\mid\Fc_{n-1}\rB
   \ \le\ \frac{k!}{2}\,w_n\,\tau^{\,k-2}
   \qquad\text{for every integer }k\ge3 ,
\]
then $\psi_n(s)=\frac{w_ns^{2}}{2(1-\tau s)}$ is a majorant on
$s_{\max}=1/\tau$ --- the sub-gamma bound behind Bernstein's inequality
\cite{Ber46}, \cite[\S2.4]{BLM13} --- with rate function and tilt
\[
   \lp\Psi^{(t)}\rp^{\star}(x)=\frac{w}{\tau^{2}}\,h_1\lp\frac{\tau x}{w}\rp
   \ \ge\ \frac{x^{2}}{2(w+\tau x)},
   \qquad
   s^{\ast}_t=\frac1\tau\lp1-\lp1+\frac{2\tau\Delta_t}{w}\rp^{-1/2}\rp ,
\]
where $h_1(z):=1+z-\sqrt{1+2z}$, in the notation of
\cite[\S2.4]{BLM13}. \Cref{lem:floor-bernstein} shows that
\Cref{hyp:freedman} implies this with $\tau=b/3$, which is where the $b/3$ in
\eqref{eq:rate-rung3} comes from.
\end{remark}

\begin{remark}[The dichotomy in the horizon]\label{rem:dichotomy}
In the homogeneous case $c_n\equiv c$, $w_n\equiv\bar w$, $b_n\equiv b$, the
Bernstein form of \eqref{eq:bound-rung3} is
\[
   \gamma_t(\alpha)\ \le\
   \exp\lp-\frac{t\,\delta_t^{2}}{2(\bar w+\varsigma\delta_t)}\rp,
   \qquad
   \delta_t:=\lp c-\tfrac{\ell}{t}\rp_+,\quad \varsigma=b/3 ,
\]
a single expression interpolating between a diffusive regime
$\delta_t\ll\bar w/\varsigma$, where the exponent is
$t\delta_t^{2}/(2\bar w)$ and only the variance matters, and a linear regime
$\delta_t\gg\bar w/\varsigma$, where it is $t\delta_t/(2\varsigma)$ and only the
scale does. Since $\delta_t\uparrow c$, an experiment run long enough ends in
whichever regime $c$ dictates.
\end{remark}

\begin{remark}[A strictly sharper form]\label{rem:fgl}
Fan, Grama and Liu \cite{FGL12} prove a Hoeffding-type inequality for
\emph{supermartingales} that dominates \eqref{eq:bound-rung3} under identical
hypotheses. With $x:=\Delta_t/b$ and $v^{2}:=w/b^{2}$, their bound is
\[
   \gamma_t(\alpha)\ \le\ H_t(x,v):=
   \lB\lp\frac{v^{2}}{x+v^{2}}\rp^{x+v^{2}}
     \lp\frac{t}{t-x}\rp^{t-x}\rB^{\frac{t}{t+v^{2}}}
   \qquad(x\le t),
\]
with the convention $(+\infty)^{0}:=1$ at $x=t$, and
$H_t(x,v)\le\exp(-\frac{x^{2}}{2(v^{2}+x/3)})$. The extra dependence on the
horizon is precisely the Hoeffding factor that a Bennett exponent discards. That
their result applies verbatim is an instance of the point made in
\Cref{rem:notevar}: what \Cref{def:main} produces is a supermartingale
difference sequence, and the supermartingale --- not the martingale --- forms of
the classical inequalities are the ones this paper needs.
\end{remark}

\begin{example}[Clipping puts you on this rung]\label{ex:clipping}
Let $E$ be any \evar{} and $\nu\in(0,1)$. Its \emph{$\nu$-clipped}
version
\[
   E^{(\nu)}\ :=\ (1-\nu)+\nu E
\]
is again an \evar{}, since $\E_{\Prob_0}[E^{(\nu)}\mid\Fc_{n-1}]\le1$, and
it satisfies the deterministic floor $E^{(\nu)}\ge1-\nu$, hence the
floor half of \Cref{hyp:freedman} with $b=c+\log\frac1{1-\nu}$ for any
anchor $c\le\E_{\Prob_1}[\log E^{(\nu)}\mid\Fc_{n-1}]$. Clipping therefore
restores, for an arbitrary \evar{} and unconditionally, Cram\'er's condition at
every $s_0>0$ and the crash tail \eqref{eq:h2} with any exponent $a$ and
$K=e^{ab}$: it puts the construction on rungs~1 and~2 outright. It does not by
itself reach rung~3, because \Cref{hyp:freedman} also demands a conditional
second moment and clipping leaves the \emph{winning} side untouched,
$\log((1-\nu)+\nu E)\sim\log E$ as $E\to\infty$ --- exactly the
asymmetry of \Cref{rem:winning}. It reaches rung~3 as soon as
$\E_{\Prob_1}[(\log E)^{2}\mid\Fc_{n-1}]<\infty$, and rung~5 as soon as
$E$ is bounded above. \Cref{ex:twosided} shows how to buy the missing ceiling
for nothing, by clipping from both sides with the same parameter.
\end{example}

\subsection{Rung 4: a sub-Gaussian lower tail}
\label{sec:rung4}

\begin{hypothesis}[Conditionally sub-Gaussian lower tail]\label{hyp:subgauss}
There are constants $\sigma_n^{2}\in[0,\infty)$ with
\begin{equation}\label{eq:h4}
   \varphi_n(s)\ \le\ \frac{\sigma_n^{2}s^{2}}{2}
   \qquad\Prob_1\text{-a.s., for every }n\in\N\text{ and }s\ge0 .
\end{equation}
\end{hypothesis}

This is a one-sided condition: only the crash direction is constrained, which is
all a lower deviation of the wealth requires. For the two-sided notion and its
equivalent formulations in terms of Orlicz norms and tail bounds see
\cite[\S2.3]{BLM13} and \cite[\S2.1]{Wai19}. It is the hypothesis that the
mixture and stitching devices of the literature consume
\cite{Rob70,RS74,Lai76,How21,KK21}, since those are built for boundaries of the
shape $\sqrt{V_th(V_t)}$.

\begin{corollary}[Rung 4]\label{cor:rung4}
Assume \Cref{hyp:subgauss} and let $V_t\ge\sum_{n\le t}\sigma_n^{2}$ be
constants. Then $\psi_n(s)=\sigma_n^{2}s^{2}/2$ is a majorant with aggregate
$\Psi^{(t)}(s)=V_ts^{2}/2$ on $s_{\max}=\infty$, the rate function is
$(\Psi^{(t)})^{\star}(x)=\frac{x^{2}}{2V_t}$, the optimising tilt is
$s^{\ast}_t=\Delta_t/V_t$, and
\begin{equation}\label{eq:bound-rung4}
   \bar\gamma_t(\alpha)\ \le\ \gamma_t(\alpha)\ \le\
   \exp\lp-\frac{\Delta_t^{2}}{2V_t}\rp .
\end{equation}
Moreover \eqref{eq:unif} reads: for every constant $s>0$ and $\delta\in(0,1)$,
with $\Prob_1$-probability at least $1-\delta$, simultaneously for all
$t\in\N_0$,
\begin{equation}\label{eq:unif-rung4}
   \log W_t\ >\ C_t-\frac{sV_t}{2}-\frac{\log\frac1\delta}{s} ,
\end{equation}
and in the homogeneous case $c_n\equiv c>0$, $\sigma_n^{2}\equiv\sigma^{2}$,
choosing $s=c/\sigma^{2}$ in \eqref{eq:stop},
\begin{equation}\label{eq:stop-rung4}
   \Prob_1\lp\tau_\alpha\le
      \Big\lceil\tfrac{2\ell}{c}
      +\tfrac{2\sigma^{2}}{c^{2}}\log\tfrac1\delta\Big\rceil\rp\ \ge\ 1-\delta .
\end{equation}
\end{corollary}

\begin{hypothesis}[Bounded tilted variance]\label{hyp:tiltvar}
There are constants $s_{\max}\in(0,\infty]$ and $v_n\in[0,\infty)$ such that
$\varphi_n$ is finite on $[0,s_{\max})$ $\Prob_1$-a.s.\ and, $\Prob_1$-a.s.\
for every $n\in\N$ and every $s\in[0,s_{\max})$,
\begin{equation}\label{eq:h4p}
   \Var_{Q_{n,s}}\lp L_n\rp\ \le\ v_n ,
   \qquad
   \frac{dQ_{n,s}}{d\Prob_1(\,\cdot\mid\Fc_{n-1})}
   \ :=\ \frac{e^{-s(L_n-c_n)}}{\E_{\Prob_1}[e^{-s(L_n-c_n)}\mid\Fc_{n-1}]} ,
\end{equation}
where $\Prob_1(\,\cdot\mid\Fc_{n-1})$ is a regular conditional distribution of
$L_n$ given $\Fc_{n-1}$ and $Q_{n,s}$ its exponential tilt, so that
$\Var_{Q_{n,s}}(L_n)$ is an $\Fc_{n-1}$-measurable random variable.
\end{hypothesis}

\begin{corollary}[Rung $4'$]\label{cor:rung4p}
Assume \Cref{hyp:tiltvar} and let $V_t\ge\sum_{n\le t}v_n$ be constants. Then
$\psi_n(s)=v_ns^{2}/2$ on $[0,s_{\max})$ is a majorant, the rate function is
$\frac{x^{2}}{2V_t}$ for $x\le s_{\max}V_t$ and
$s_{\max}x-\frac{V_ts_{\max}^{2}}{2}$ beyond, the optimising tilt is
\begin{equation}\label{eq:sstar-rung4p}
   s^{\ast}_t=\min\lC\frac{\Delta_t}{V_t},\ s_{\max}\rC
   \quad\text{(approached, not attained, in the second case)} ,
\end{equation}
and in both cases
\begin{equation}\label{eq:bound-rung4p}
   \bar\gamma_t(\alpha)\ \le\ \gamma_t(\alpha)\ \le\
   \exp\lp-\frac{\Delta_t}{2}
      \min\lC\frac{\Delta_t}{V_t},\ s_{\max}\rC\rp .
\end{equation}
\end{corollary}

\begin{remark}[Why the tilted variance is the natural hypothesis]
\label{rem:tiltvar}
The second derivative of the anchored cumulant generating function \emph{is} the
tilted variance, $\varphi_n''(s)=\Var_{Q_{n,s}}(L_n)$, so \Cref{hyp:tiltvar} is
\Cref{hyp:subgauss} read through Taylor's formula --- in one direction only.
It is strictly stronger. Sub-Gaussianity constrains the \emph{values} of
$\varphi_n$, and the anchor may be slack, so $\varphi_n''$ need not be bounded
at all: if $L_n-c_n$ takes the values $0$ and $2$ with equal conditional
probability then $\varphi_n(s)=-s+\log\cosh s\le0$, so \Cref{hyp:subgauss} holds
with $\sigma_n^{2}=0$ while $\Var_{Q_{n,0}}(L_n)=1$; and in the first example of
\Cref{rem:notachain} one has $\varphi_n\le0$ with infinite conditional variance,
so \Cref{hyp:tiltvar} fails for every finite $v_n$.

What rung~$4'$ adds is the finite tilt range. When $\varphi_n$ is finite only up
to a horizon, a quadratic majorant is available strictly inside it and not up to
it. An exponential lower tail is the model case: if $c_n-L_n=X-\frac1a$ with $X$
conditionally $\mathrm{Exp}(a)$, then $\varphi_n''(s)=\frac1{(a-s)^{2}}$ on
$[0,a)$, unbounded at the horizon, so one takes $s_{\max}$ strictly inside, say
$s_{\max}=\frac a2$ and $v_n=\frac4{a^{2}}$, and accepts the linear regime
beyond $\Delta_t=s_{\max}V_t$. Reaching the horizon $a$ itself needs rung~$2''$
or rung~$3'$, whose majorants blow up there too. The bound
\eqref{eq:bound-rung4p} degrades gracefully across the changeover: quadratic in
$\Delta_t$ while $\Delta_t\le s_{\max}V_t$ and linear afterwards, the same
dichotomy as \Cref{rem:dichotomy} with $\frac1{s_{\max}}$ in place of
$\frac b3$.
\end{remark}

\begin{remark}[The factor $2$ in \eqref{eq:stop-rung4} is an artefact]
\label{rem:factor2}
The rejection-time bound $2\ell/c$ is twice the ideal horizon $\ell/c$ dictated
by the \epower{}, and it is tempting to blame the fixed tilt of
\Cref{rem:fixedtilt}. That is not where it comes from. The quantity to minimise
is $\Theta(s)=\frac{s\ell+\log\frac1\delta}{sc-\sigma^{2}s^{2}/2}$, and the
proof evaluates it at the convenient $s=c/\sigma^{2}$; minimising over the still
deterministic tilt gives, with $A=\ell/c$ and
$B=\frac{\sigma^{2}}{c^{2}}\log\frac1\delta$,
\begin{equation}\label{eq:thetamin}
   \min_{s>0}\Theta(s)=\frac{A+B/u^{\ast}}{1-u^{\ast}/2},
   \qquad u^{\ast}=\frac{-B+\sqrt{B^{2}+2AB}}{A},
   \qquad s^{\ast}=\frac{c\,u^{\ast}}{\sigma^{2}} ,
\end{equation}
which is $\frac{\ell}{c}+\sqrt{\frac{2\ell\sigma^{2}}{c^{3}}\log\frac1\delta}
+O(1)=\frac{\ell}{c}(1+o(1))$ as $\ell\to\infty$. A single pre-chosen tilt
already attains the ideal constant asymptotically. What mixtures and stitching
buy is therefore not the constant but \emph{adaptivity}: \eqref{eq:thetamin}
needs $c$ and $\sigma^{2}$ in advance.
\end{remark}

\subsection{Rung 5: bounded log-increments}
\label{sec:rung5}

\begin{hypothesis}[Bounded log-increments]\label{hyp:bounded}
There are constants $\underline L_n<\overline L_n$ with
\begin{equation}\label{eq:h5}
   e^{\underline L_n}\ \le\ E_n\ \le\ e^{\overline L_n}
   \qquad\text{equivalently}\qquad
   \underline L_n\ \le\ L_n\ \le\ \overline L_n
   \qquad\Prob_1\text{-a.s.}
\end{equation}
We write
\begin{gather*}
   r_n:=\overline L_n-\underline L_n\quad\text{(the \emph{range})},
   \qquad
   d_n:=c_n-\underline L_n\ \ge0\quad\text{(the \emph{downward gap})},\\
   y_n:=\frac{d_n}{r_n}\in[0,1] ,
\end{gather*}
so that $y_n$ is the position of the certified drift inside the range.
\end{hypothesis}

This is the rung most constructions can be forced onto: the floor comes from
clipping (\Cref{ex:clipping}) or a bounded stake (\Cref{ex:betting}), and the
ceiling is often automatic.

\begin{corollary}[Rung 5]\label{cor:rung5}
Assume \Cref{hyp:bounded}. Then, with $s_{\max}=\infty$, both
\begin{equation}\label{eq:maj-rung5}
   \psi_n^{\mathrm{quad}}(s):=\frac{r_n^{2}s^{2}}{8}
   \qquad\text{and}\qquad
   \psi_n^{\mathrm{exact}}(s):=s\,d_n
      +\log\lp y_ne^{-sr_n}+1-y_n\rp
\end{equation}
are majorants: the first is Hoeffding's lemma \cite{Hoe63} in its conditional
form (\Cref{thm:app-hoeffding}), the second the chord bound for $y\mapsto
e^{sy}$ on the range, which is the anchored cumulant generating function of the
two-point law on the endpoints. In particular
$\psi_n^{\mathrm{exact}}\le\psi_n^{\mathrm{quad}}$; both are computed in
\Cref{app:rung5}.
\begin{enumerate}[label=(\alph*),leftmargin=2.4em]
\item \emph{(Azuma--Hoeffding form \cite{Hoe63,Azu67}.)} With constants
      $R_t^{2}\ge\sum_{n\le t}r_n^{2}$, the rate function is $2x^{2}/R_t^{2}$, the optimising tilt is
      $s^{\ast}_t=4\Delta_t/R_t^{2}$, and
      \begin{equation}\label{eq:bound-rung5a}
         \bar\gamma_t(\alpha)\ \le\ \gamma_t(\alpha)\ \le\
         \exp\lp-\frac{2\Delta_t^{2}}{R_t^{2}}\rp .
      \end{equation}
\item \emph{(Exact binary-entropy form; $\underline L_n$ and $d_n$
      deterministic.)} Let $r\ge\max_{n\le t}r_n$ and put
      $\bar y_t:=\frac{1}{tr}\sum_{n\le t}d_n\in[0,1]$ and
      $x_t:=\bar y_t-\frac{\Delta_t}{tr}$. Unlike (a), this part needs
      $\bar y_t$ to be a number; see \Cref{rem:predictable} for predictable
      $\underline L_n,d_n$. Then $\Psi^{(t)}(s)=
      t\,[\,s\,r\bar y_t+\log(\bar y_te^{-sr}+1-\bar y_t)\,]$ is an aggregate,
      the rate function is $t\,\kl(\bar y_t-\frac{x}{tr}\,\|\,\bar y_t)$, the
      optimising tilt is
      \begin{equation}\label{eq:sstar-rung5b}
         s^{\ast}_t\ =\ \frac1r\,
         \log\frac{\bar y_t\,(1-x_t)}{x_t\,(1-\bar y_t)} ,
      \end{equation}
      and
      \begin{equation}\label{eq:bound-rung5b}
         \bar\gamma_t(\alpha)\ \le\ \gamma_t(\alpha)\ \le\
         \exp\lp-t\;\kl\lp\bar y_t-\frac{\Delta_t}{t\,r}\ \Big\|\
            \bar y_t\rp\rp ,
      \end{equation}
      where $\kl(x\|y):=x\log\frac xy+(1-x)\log\frac{1-x}{1-y}$, with the
      conventions $0\log0:=0$ and $\kl(x\|y):=+\infty$ for $x\notin[0,1]$.
\end{enumerate}
Pinsker's inequality $\kl(x\|y)\ge2(x-y)^{2}$ (\Cref{thm:app-pinsker})
recovers $\exp(-2\Delta_t^{2}/(tr^{2}))$ from \eqref{eq:bound-rung5b}.
\end{corollary}

\begin{remark}[The two forms are not ordered]\label{rem:rung5-order}
Part (a) keeps the exact sum of squared ranges $R_t^{2}\le tr^{2}$ and is
sharper when the ranges are inhomogeneous; part (b) keeps the binary relative
entropy, which is far larger than its quadratic minorant whenever $\bar y_t$ is
near $0$ or $1$ --- the heavily clipped regime, where the floor sits far below
the anchor. Taking the maximum of the two exponents costs nothing.
\Cref{prp:bdd-sharp} shows that the exponent of part~(b) is exactly right: no
bound depending only on $(\underline L,\overline L,c)$ can have a larger one.
\end{remark}

\begin{remark}[Aggregating a predictable rung 5(b)]\label{rem:predictable}
By \Cref{rem:whatsrandom} the per-step quantities of every rung may be
predictable; what has to be checked in each case is that a deterministic
aggregate exists. For rungs~2--4 and~5(a) the check is termwise and needs only
a bound in the obvious direction --- below for $c_n$, above for
$w_n,\sigma_n^{2},r_n,b_n,K_n,M_{p,n}$ --- because there $\psi_n$ does not
involve $c_n$ and is nondecreasing in each of its own constants. The two
exceptions are the rungs whose majorant does involve $c_n$. Rung~1 is settled
by \Cref{rem:rung1nodrift}: the drift cancels, and \eqref{eq:bound-rung1c}
holds for arbitrary predictable $\rho_n$. Rung~5(b) is harder, because
$d\mapsto\psi_n^{\mathrm{exact}}$ is \emph{not} monotone: it vanishes at $d=0$
and at $d=r$ and peaks strictly inside, so no bound on $d_n$ alone controls
$\psi_n^{\mathrm{exact}}$.

What is monotone is the combination \eqref{eq:combined}. Writing
$c_n=\underline L_n+d_n$,
\[
   s\sum_{n\le t}c_n-\sum_{n\le t}\psi_n^{\mathrm{exact}}(s)
   \ =\ s\sum_{n\le t}\underline L_n
      -\sum_{n\le t}\log\lp1-\frac{d_n}{r}\lp1-e^{-sr}\rp\rp ,
\]
and $d\mapsto-\log(1-\frac dr(1-e^{-sr}))$ is nondecreasing, so the right-hand
side is nondecreasing in $\underline L_n$ and in $d_n$ separately. Predictable
$\underline L_n,d_n$ are therefore admissible provided one has deterministic
a.s.\ \emph{lower} bounds for both and a deterministic $r\ge\max_nr_n$: one
then applies \Cref{cor:rung5}(b) with those. An upper bound on $d_n$ is of no
use and must not be substituted --- the exponent would be over-estimated and
the resulting bound invalid. This is exactly the generality a predictable
clipping schedule $(\nu_n)_{n\in\N}$ requires.
\end{remark}

\begin{example}[Two-sided clipping reaches this rung, for any \evar{}]
\label{ex:twosided}
\Cref{ex:clipping} clips from below and leaves the winning side untouched, so
it does not by itself reach \Cref{hyp:bounded}. Clipping from both sides with
the \emph{same} parameter does. For $\nu\in(0,1)$ put
\begin{equation}\label{eq:twosided}
   E^{[\nu]}\ :=\ \nu+(1-\nu)\min\lC E,\ \tfrac1\nu\rC .
\end{equation}
Then $E^{[\nu]}$ is again an \evar{}, since
$\min\{E,\frac1\nu\}\le E$ gives
$\E_{\Prob_0}[E^{[\nu]}\mid\Fc_{n-1}]\le\nu+(1-\nu)=1$, and it satisfies
\Cref{hyp:bounded} with
\[
   \underline L=\log\nu ,
   \qquad
   \overline L=\log\frac{1-\nu+\nu^{2}}{\nu} ,
   \qquad
   r(\nu)=\log\frac{1-\nu+\nu^{2}}{\nu^{2}} ,
\]
\emph{whatever $E$ is}: no boundedness, no moment, no tail hypothesis is
required. This is the sense in which rung~5 is universally reachable, and it is
the cleanest answer to the question of what one should bet with when nothing is
known about $E$.

Nothing is lost in the limit. For fixed $e\in[0,\infty]$ the map
$\nu\mapsto\nu+(1-\nu)\min\{e,\frac1\nu\}$ is nonincreasing on $(0,1)$ when
$e>1$ and nondecreasing when $e<1$, and tends to $e$ as $\nu\downarrow0$; so
monotone convergence, applied separately on $\{E>1\}$, where the integrands
increase and are bounded below by $\log E^{[\nu_0]}\ge0$, and on $\{E\le1\}$,
where they decrease and are bounded above by $0$, gives
\[
   \E_{\Prob_1}\lB\log E^{[\nu]}\mid\Fc_{n-1}\rB
   \ \xrightarrow[\ \nu\downarrow0\ ]{}\
   \E_{\Prob_1}\lB\log E\mid\Fc_{n-1}\rB\ \in[-\infty,\infty] .
\]
What is lost is the range. As $\nu\downarrow0$,
\[
   r(\nu)=2\log\tfrac1\nu+O(\nu)\ \longrightarrow\ \infty ,
   \qquad
   y^{[\nu]}=\frac{c^{[\nu]}-\log\nu}{r(\nu)}
   =\frac12+O\lp\tfrac{1}{\log(1/\nu)}\rp ,
\]
writing $c^{[\nu]}$ for the \epower{} of $E^{[\nu]}$. So the certified drift
sits asymptotically at the midpoint of the range, where \Cref{cor:rung5}(b)
gains nothing over \Cref{cor:rung5}(a), and the Azuma--Hoeffding exponent per
observation behaves like
\[
   \frac{2\,(c^{[\nu]})^{2}}{r(\nu)^{2}}
   \ \sim\ \frac{c^{2}}{2\log^{2}\frac1\nu}\ \longrightarrow\ 0 .
\]
The optimal $\nu$ is therefore interior and strictly positive: aggressive
clipping ($\nu\downarrow0$) restores the raw \evar{} but destroys the rate,
and timid clipping ($\nu\uparrow1$) yields a tiny range but no \epower{},
since $c^{[\nu]}\to0$. A predictable schedule $(\nu_n)_{n\in\N}$ tuned on the
past is admissible, by \Cref{rem:predictable}.
\end{example}

\begin{example}[Betting \evars]\label{ex:betting}
Let $g$ be measurable with values in $[0,1]$, let
$\E_{\Prob_0}[g(X_n)\mid\Fc_{n-1}]=m\in(0,1)$ be known, and let the
\emph{stake} $\nu_n\in(0,\frac1m)$ be predictable. Then
$E_n:=1+\nu_n(g(X_n)-m)$ is a betting sequence in the sense of
\Cref{def:setting}, strictly positive, with
\[
   0<1-\nu_nm\ \le\ E_n\ \le\ 1+\nu_n(1-m) ,
\]
so \Cref{hyp:bounded} holds with $\underline L_n=\log(1-\nu_nm)$ and
$\overline L_n=\log(1+\nu_n(1-m))$, and \Cref{hyp:freedman} holds with the
floor $b_n=c_n-\underline L_n$ and any predictable
$w_n\ge\E_{\Prob_1}[(L_n-c_n)^{2}\mid\Fc_{n-1}]$, for which the crude choice
$w_n=r_n^{2}$ always works. This is the betting construction of Waudby-Smith and
Ramdas \cite{WSR24} and covers Shafer's testing-by-betting formulation
\cite{Sha21}; the stake governs the variance--floor trade-off of
\Cref{rem:dichotomy}, and \Cref{rem:predictable} is exactly the generality a
predictable stake schedule requires.
\end{example}

\subsection{Rung 6: i.i.d.\ increments and the exact rate}
\label{sec:rung6}

The top rung replaces every majorant by the truth. It is the terminus in two
senses: the exponent it delivers is exact, so no further hypothesis can improve
it, and it is the rung at which the companion paper can say what that exponent
is bounded by, for \emph{every} \evar{}. The price is the only genuinely
restrictive hypothesis in the paper.

\begin{hypothesis}[I.i.d.\ increments]\label{hyp:iid}
The situation is that of \Cref{ex:iidsetting}: $(E_n)_{n\in\N}$ is i.i.d.\
under $\Prob_1$, generated by a single \evar{} $E$ with
$\E_{P_1}|\log E|<\infty$ and $c:=\E_{P_1}[\log E]>0$.
\end{hypothesis}

\begin{corollary}[Rung 6]\label{cor:rung6}
Assume \Cref{hyp:iid} and write
\[
   \Lambda_s(E):=-\log\E_{P_1}\lB E^{-s}\rB,
   \qquad
   \Lambda(E):=\sup_{s\ge0}\Lambda_s(E)\ \in[0,\infty] ,
\]
the \emph{tilted exponent} at tilt $s$ and the \emph{Chernoff--Stein exponent}
of $E$ --- the exponent it achieves at fixed level as $t\to\infty$, whose
supremum over all \evars{} is the classical Chernoff--Stein exponent
\cite[Thm.~4.25]{CEIL}; at $E=R$ it is the Chernoff information of the pair
\cite[Cor.~4.18]{CEIL}.
Then $\psi(s):=\varphi(s)=sc-\Lambda_s(E)$ is a majorant --- with equality in
\eqref{eq:majorant}, so no majorant is smaller --- the rate function is
$\psi^{\star}(x)=\sup_{s\ge0}\{\Lambda_s(E)-s(c-x)\}$, the optimising tilt
solves
\begin{equation}\label{eq:sstar-rung6}
   \E_{Q_s}[\log E]=\frac{\ell}{t},
   \qquad \frac{dQ_s}{dP_1}:=\frac{E^{-s}}{\E_{P_1}[E^{-s}]} ,
\end{equation}
and
\begin{equation}\label{eq:bound-rung6}
   \gamma_t(\alpha)\ \le\
   \exp\lp-\sup_{s\ge0}\lC t\,\Lambda_s(E)-s\ell\rC\rp .
\end{equation}
If moreover $0$ is interior to $\lC x\in\R\st I(x)<\infty\rC$, $I$ being the Cram\'er rate
function of $\log E$ under $P_1$, then \eqref{eq:bound-rung6} is exact on the
exponential scale, by Cram\'er's theorem (\Cref{thm:app-cramer}):
\begin{equation}\label{eq:cramerlimit}
   \lim_{t\to\infty}\ \frac1t\log\gamma_t(\alpha)\ =\ -\Lambda(E)
   \qquad\text{for every fixed }\alpha\in(0,1) .
\end{equation}
For the sequential error \eqref{eq:ordering} and \eqref{eq:cramerlimit} give
one half,
\begin{equation}\label{eq:cramerlimit-seq}
   \liminf_{t\to\infty}\ -\frac1t\log\bar\gamma_t(\alpha)\ \ge\ \Lambda(E) .
\end{equation}
\end{corollary}

\begin{remark}[The sequential error has the same exponent]\label{rem:seqexp}
Equality is expected in \eqref{eq:cramerlimit-seq} but is not proved here, and
it does not follow from \eqref{eq:ordering}, which supplies no reverse
inequality. The argument is a change of measure. Suppose the supremum defining
$\Lambda(E)$ is attained at an $s^{\ast}$ interior to
$\lC s\ge0\st\E_{P_1}[E^{-s}]<\infty\rC$ and let $Q:=Q_{s^{\ast}}$ be the tilt
of \eqref{eq:sstar-rung6}, under which the increments have mean zero and finite
variance. Writing $S_t=\log W_t$ and using
$\frac{d\Prob_1}{dQ}\big|_{\Fc_t}=e^{s^{\ast}S_t-t\Lambda(E)}$,
\begin{align*}
   \bar\gamma_t(\alpha)
   &\ \ge\ \E_Q\lB\mathbf 1\lC\max_{n\le t}S_n\le\ell,\ S_t\ge-K_t\rC
      e^{s^{\ast}S_t-t\Lambda(E)}\rB\\
   &\ \ge\ e^{-s^{\ast}K_t-t\Lambda(E)}\,
      Q\lp\max_{n\le t}S_n\le\ell,\ S_t\ge-K_t\rp ,
\end{align*}
so it suffices that the $Q$-probability decay sub-exponentially with
$K_t=o(t)$. It does: under $Q$ the walk is centred with finite variance, so
staying below a fixed level up to time $t$ has probability of order
$t^{-1/2}$ by fluctuation theory. Making this precise is a digression from the
subject of this paper, and we do not pursue it; nothing below uses it.
\end{remark}

\begin{lemma}[When an exponent exists at all]\label{lem:positive}
Let $E\ge0$ with $c:=\E_{P_1}[\log E]$ well defined in $[-\infty,\infty)$, and
let $\Lambda_s(E)$ and $\Lambda(E)$ be as in \Cref{cor:rung6}. The following
are equivalent:
\begin{enumerate}[label=(\roman*),leftmargin=2.4em]
\item\label{it:p1} $\Lambda(E)>0$, that is, $\E_{P_1}[E^{-s}]<1$ for some
      $s>0$;
\item\label{it:p2} $c>0$ \emph{and} \Cref{hyp:cramer} holds, that is,
      $\E_{P_1}[E^{-s_0}]<\infty$ for some $s_0>0$.
\end{enumerate}
Moreover, when $c>0$ nothing is lost by restricting the tilt to $s\ge0$:
$\sup_{s\in\R}\{-\log\E_{P_1}[E^{-s}]\}=\Lambda(E)$.
\end{lemma}

\begin{remark}[Reading \Cref{lem:positive}]\label{rem:positive}
This is the reason the ladder begins where it does. Positive \epower{} alone
gives no exponent, and a negative moment alone gives no exponent; it is exactly
their conjunction that does. Rungs~2--5 all imply \Cref{hyp:cramer} --- rungs~2,
$2''$, 3, 4 and~5 by \Cref{lem:implications}, and rungs~$2'$, $3'$ and~$4'$
because their majorants are finite on $[0,s_{\max})$ with $s_{\max}>0$ --- so on
that part of the ladder the only remaining question is whether $c>0$, which is
what $\Delta_t>0$ encodes, and how large the exponent is, which is what the
majorant decides.

Rung~6 is the exception, and this is exactly why the lemma is stated here.
\Cref{hyp:iid} asks only for $\E_{P_1}|\log E|<\infty$ and $c>0$, which permits
a lower tail with no exponential moment at all: for
$\Prob_1(\log E=-k)\propto k^{-3}$ together with one large positive atom, the
\epower{} is positive while $\E_{P_1}[E^{-s}]=\infty$ for every $s>0$, so
$\Lambda(E)=0$ and \eqref{eq:bound-rung6} is the trivial bound $1$ at every
horizon --- although $\Delta_t>0$. At the top of the ladder, positive \epower{}
is not enough.

The lemma also explains why the tilt is searched over $s\ge0$: by Jensen,
$\E_{P_1}[E^{-s}]\ge e^{-sc}>1$ for every $s<0$ when $c>0$, so negative tilts
could not help even if they were admissible --- and in \Cref{thm:master} they
are not, the Markov step requiring $x\mapsto x^{-s}$ to be nonincreasing.
\end{remark}

\begin{remark}[Where the ladder stops losing, and what stops it]
\label{rem:ceiling}
Rungs~1--5 are computable lower bounds for the Chernoff--Stein exponent $\Lambda(E)$
in increasing order of sharpness and of what they assume; rung~6 is
$\Lambda(E)$ itself. The natural
next question --- how large $\Lambda(E)$ can be, over all \evars{} --- is
answered in the companion paper: $\Lambda(E)<\KL(P_0\|P_1)$ for every \evar{},
the bound is tight, and it is not attained \cite[Thm.~4.25]{CEIL}. Note the
direction of the divergence. The \epower{} $c$ is capped by $\KL(P_1\|P_0)$ and
the exponent by $\KL(P_0\|P_1)$; the two design criteria genuinely pull apart,
and \cite[Thm.~4.29]{CEIL} makes the separation exact.
\end{remark}

\subsection{The ladder at a glance}
\label{sec:table}

\Cref{tab:ladder} collects the majorants and their transforms. It is stated
per observation --- one step, one hypothesis, drift $c$ --- so that the entries
are the functions $\psi$ and $\psi^{\star}$ themselves; the horizon-$t$ bound is
$\exp(-t\,\psi^{\star}(\Delta_t/t))$ by \Cref{cor:iid} in the homogeneous case,
and $\exp(-(\Psi^{(t)})^{\star}(\Delta_t))$ in general.

\begin{table}[htbp]
\centering
\footnotesize
\setlength{\tabcolsep}{3pt}
\renewcommand{\arraystretch}{1.35}
\begin{tabular}{@{}>{\centering\arraybackslash}p{0.75cm}
   >{\raggedright\arraybackslash}p{3.5cm}
   >{\raggedright\arraybackslash}p{4.2cm}
   >{\raggedright\arraybackslash}p{3.2cm}
   >{\raggedright\arraybackslash}p{3.2cm}@{}}
\toprule
& Hypothesis & Majorant $\psi(s)$ & Rate $\psi^{\star}(x)$
& Tilt $s^{\ast}$\\
\midrule
1 & one negative moment
  & $\frac{s}{s_0}\kappa$ on $[0,s_0]$
  & $(s_0x-\kappa)_+$
  & $s_0$\\
2 & crash tail, one-sided $p$-th moment
  & $c_pM_ps^{p}+\frac{Ks^{2}}{a(a-s)}$
  & ---
  & \eqref{eq:sstar-rung2-exact}\\
$2'$ & crash tail, $L\log L$ moment
  & $\frac{m_1T}{2}s^{2}+\frac{M}{\log T}s+\frac{Ks^{2}}{a(a-s)}$
  & ---
  & $s_\vartheta$\\
$2''$ & crash tail, two-sided second moment
  & $\frac{w_0s^{2}}{2}+\frac{K(s/a)^{3}}{1-s/a}$
  & ---
  & \eqref{eq:sstar-rung2pp}\\
3 & floor $b$, variance $w$
  & $\frac{w}{b^{2}}(e^{sb}-sb-1)$
  & $\frac{w}{b^{2}}h\!\lp\frac{bx}{w}\rp$
  & $\frac1b\log\!\lp1+\frac{bx}{w}\rp$\\
$3'$ & Bernstein moments, scale $\tau$
  & $\frac{ws^{2}}{2(1-\tau s)}$
  & $\frac{w}{\tau^{2}}h_1\!\lp\frac{\tau x}{w}\rp$
  & $\frac1\tau\lp1-\lp1+\frac{2\tau x}{w}\rp^{-\frac12}\rp$\\
4 & sub-Gaussian, proxy $\sigma^{2}$
  & $\frac{\sigma^{2}s^{2}}{2}$
  & $\frac{x^{2}}{2\sigma^{2}}$
  & $\frac{x}{\sigma^{2}}$\\
$4'$ & tilted variance $\le v$ on $[0,s_{\max})$
  & $\frac{vs^{2}}{2}$ on $[0,s_{\max})$
  & $\ \ge\frac x2\min\!\lC\frac xv,s_{\max}\rC$
  & $\min\!\lC\frac xv,s_{\max}\rC$\\
5a & range $r$
  & $\frac{r^{2}s^{2}}{8}$
  & $\frac{2x^{2}}{r^{2}}$
  & $\frac{4x}{r^{2}}$\\
5b & range $r$, position $y$
  & $syr+\log\!\lp ye^{-sr}+1-y\rp$
  & $\kl\!\lp\bar x\,\|\,y\rp$
  & $\frac1r\log\frac{y(1-\bar x)}{\bar x(1-y)}$\\
6 & i.i.d.\ increments
  & $sc-\Lambda_s(E)$
  & $\sup_{s\ge0}\{\Lambda_s(E)-s(c-x)\}$
  & \eqref{eq:sstar-rung6}\\
\midrule
--- & none
  & $+\infty$ for $s>0$
  & $0$, and unimprovable
  & ---\\
\bottomrule
\end{tabular}
\caption{The ladder, per observation. Here
$h(z)=(1+z)\log(1+z)-z$, $h_1(z)=1+z-\sqrt{1+2z}$,
$\kl(x\|y)=x\log\frac xy+(1-x)\log\frac{1-x}{1-y}$,
$\kappa=\log\rho+s_0c$, $\bar x=y-\frac xr$ and $q=\frac{p}{p-1}$. Rungs~2,
$2'$ and $2''$ share the crash-tail hypothesis \eqref{eq:h2} and differ in what
they assume in addition --- a one-sided $p$-th moment, an $L\log L$ moment, and
a two-sided second moment respectively; their transforms have no closed form,
but \eqref{eq:sstar-rung2} gives an explicit tilt for rung~2, and
\eqref{eq:rung2pp-small} and \eqref{eq:rung2pp-large} give explicit exponents
for rung~$2''$, valid for $x\le\frac{aw_0}{2}$ and for $ax>K$ respectively. In the moment-limited regime rung~2 gives
$\frac{x^{q}}{2(4c_pM_p)^{q-1}}$. Rung~5b requires $0\le x\le yr$.
The last row is not a rung: it is the assertion that no hypothesis at all buys
nothing at all, which is the no-free-lunch theorem
\cite[Thm.~3.7]{CEIL} of the companion paper, and it is what makes the rest of
the table necessary.}
\label{tab:ladder}
\end{table}

\begin{remark}[Which parameters may be predictable]\label{rem:promote}
\Cref{rem:whatsrandom} allows $c_n$ and $\psi_n$ to be predictable, and
\Cref{rem:predictable} says how each rung's per-step quantities are then
aggregated. It is worth separating three roles.

\emph{Free.} The quantities that enter a majorant additively or monotonically
may be predictable, provided a deterministic bound in the direction that
\emph{increases} the majorant is available $\Prob_1$-a.s.: above for
$\rho_n,K_n,M_{p,n},m_{1,n},M_n,w_{0,n},w_n,b_n,\sigma_n^{2},v_n,r_n$, below for
$c_n,\underline L_n,d_n$.

To this list belongs the common range $r$ of rung~5(b), with $r\ge r_n$: the
chord may always be taken on a larger interval, which is what
\Cref{rem:predictable} does.

\emph{Free, but in the direction that shrinks the tilt domain.} The parameters
that fix $s_{\max}$ may also be predictable, bounded so that the common domain
survives: $a\le a_n$ for rungs~2 and~$2''$, since $\psi_n$ decreases in $a$;
$\tau\ge\tau_n$ for rung~$3'$, since it increases in $\tau$; and $s_{\max}\le
s_{\max,n}$ for rung~$4'$.

\emph{Not free.} Three things must be numbers. The aggregate
$(C_t,\Psi^{(t)})$, hence $\Delta_t$, because it is the bound. The tilt $s$ in
\Cref{thm:master}(d)--(e), because those rest on one supermartingale built at
one tilt --- and with it $s_0$ in rung~1, which plays the same role. And the
moment order $p$ of rung~2, which indexes the hypothesis rather than scaling it.
The truncation level $T$ of rung~$2'$ is different again: $\psi_n(\cdot\,;T)$
is \emph{not} monotone in $T$ --- the term $\frac{M_n}{\log T}s$ decreases while
$\frac{m_{1,n}T}{2}s^{2}$ increases --- so a one-sided bound on a predictable
$T_n$ is not admissible. Since $T$ is a free design parameter rather than part
of the hypothesis, the remedy is simply to fix one deterministic $T>1$; a
predictable $T_n$ needs two-sided bounds $\underline T\le T_n\le\bar T$ and the
aggregate $\frac{m_1\bar T}{2}s^{2}+\frac{M}{\log\underline T}s
+\frac{Ks^{2}}{a(a-s)}$.

In \Cref{thm:empbern} predictable versions of $\lambda$ and $\bar r$ are
admissible but move inside the sums: \eqref{eq:eb-pred} exhibits this for
$\lambda$, and the substitution $\xi_n=(c_n-L_n)/\bar r_n$ does the same for
$\bar r$.
\end{remark}

\begin{remark}[The rungs do not form a chain]\label{rem:notachain}
It is tempting to read \Cref{tab:ladder} as a totally ordered list of
hypotheses. It is not one. What does hold, by \Cref{lem:implications}, is
\[
   5\ \Longrightarrow\ 3\ \Longrightarrow\ 4\ \Longrightarrow\ 1 ,
   \qquad
   3\ \Longrightarrow\ 2\ \Longrightarrow\ 1 ,
\]
with $5\Rightarrow4$ also available directly and with a far better constant
proxy. Everything else fails, and it fails for reasons worth naming.

Rung~4 does not imply rung~3, because \Cref{hyp:subgauss} is one-sided. Let
$L_n-c_n\ge0$ $\Prob_1$-a.s.\ with a heavy upper tail, say
$\Prob_1(L_n-c_n>x\mid\Fc_{n-1})=(1+x)^{-2}$. Then $e^{-s(L_n-c_n)}\le1$, so
$\varphi_n\le0$ and \Cref{hyp:subgauss} holds with $\sigma_n^{2}=0$; but
$\E_{\Prob_1}[(L_n-c_n)^{2}\mid\Fc_{n-1}]=\infty$, so no $w_n$ exists.

Rung~4 does not imply rung~2 either. Here one must be slightly more careful,
because the exponent $p$ in \Cref{hyp:crash} is existentially quantified, so a
tail with a finite $p'$-th moment for some $p'\in(1,p)$ would still qualify.
Take therefore
\[
   \Prob_1\lp L_n-c_n>x\mid\Fc_{n-1}\rp
   =\min\lC1,\ \frac{1}{x\log^{2}x}\rC\qquad(x>0) ,
\]
again with $L_n\ge c_n$ $\Prob_1$-a.s. The mean is finite, so the anchor holds
and $\varphi_n\le0$ gives \Cref{hyp:subgauss} with $\sigma_n^{2}=0$; but
$\E_{\Prob_1}[((L_n-c_n)_+)^{p}\mid\Fc_{n-1}]=\infty$ for \emph{every} $p>1$,
since $\int^{\infty}px^{p-2}\log^{-2}x\,dx$ diverges, so \eqref{eq:h2b} fails at
every admissible exponent. Sub-Gaussianity of the \emph{lower} tail is silent
about the winning side, which is exactly the point of \Cref{rem:winning}.

Rung~2 does not imply rung~4, because an exponential crash tail is not a
Gaussian one. Let $c_n-L_n=X-\frac1a-\epsilon$ with $X$ conditionally
$\mathrm{Exp}(a)$ and $\epsilon>0$. The anchor holds, \eqref{eq:h2} holds with
$K_n=e^{-1-a\epsilon}\le1$, and \eqref{eq:h2b} holds because the winning side is
bounded by $\frac1a+\epsilon$; but
$\varphi_n(s)=-s(\frac1a+\epsilon)+\log\frac{a}{a-s}\to\infty$ as $s\uparrow a$,
which no quadratic majorises. The same example has an unbounded crash side and
so does not satisfy rung~3 either. Rungs~2 and~4 are therefore
\emph{incomparable}. Rung~6 is comparable to none of them: it is a structural
hypothesis about the sequence, not a tail hypothesis about one step.

The picture is thus a partial order --- rung~5 at the top of the tail
hypotheses, rung~3 below it, rungs~2 and~4 incomparable beneath rung~3, rung~1
below everything --- and, importantly, a hypothesis being logically weaker does
not make its bound worse at a given horizon, since the constants a direct
argument supplies are much better than those the implications transport. That is
why all six are worth stating. An earlier draft of this material asserted a
total order; the assertion was false.
\end{remark}

\begin{remark}[When each rung becomes informative]\label{rem:vacuity}
Every bound above is trivial while $\Delta_t=0$, that is, before the horizon
$\ell/c$ at which the certified drift has accumulated past the threshold. That
much is unavoidable, and is not an artefact of the method: by
\cite[Lem.~5.10]{CEIL} a design of \epower{} $c$ has essentially no power
before $\ell/(2c)$ observations, whatever the majorant. The rungs differ in what
happens \emph{just after} $\ell/c$. Rungs~2--6 have $\psi(s)=o(s)$ as
$s\downarrow0$, hence $\psi^{\star}(x)>0$ for every $x>0$, and the bound becomes
informative immediately. Rung~1 has $\psi(s)=\frac{s}{s_0}\kappa$, linear at the
origin, and its rate vanishes identically up to $x=\kappa/s_0$; by
\eqref{eq:crit} it becomes informative only at $t_0\ge\ell/c$, the overshoot
factor $\frac{s_0c}{s_0c-\kappa}\ge1$ being the price of assuming nothing but
one moment.
\end{remark}

\section{An observable rung}
\label{sec:observable}

Every rung of \Cref{sec:ladder} is stated in terms of \emph{population}
constants under $\Prob_1$: a variance $w$, a range $r$, a crash exponent $a$. In
practice these are asserted rather than known, and of them the variance is the
one an experimenter is least able to assert honestly. This section gives the
rung that can be \emph{observed} instead: an empirical-Bernstein floor whose
rejection-time certificate is a stopping time, evaluable from the data as they
arrive.

It is worth being precise about the scope, since it is easy to overstate. Only
the variance becomes observable. The drift certificate of \Cref{def:main}(a) may
depend on the past --- $c_n$ is predictable, not a number --- but it remains a
hypothesis about $\Prob_1$, and the rule below does not validate it: if the asserted drift is wrong the rule will still fire, and
the conclusion it licenses will be false. What this rung removes is the
variance, not the requirement to have asserted a growth rate.

\begin{hypothesis}[A ceiling on the deficit]\label{hyp:ceiling}
There is a constant $\bar r\in(0,\infty)$ with
\begin{equation}\label{eq:h6}
   L_n\ \le\ c_n+\bar r
   \qquad\Prob_1\text{-a.s., for every }n\in\N .
\end{equation}
\end{hypothesis}

This is a \emph{ceiling}, not a floor, so \Cref{hyp:ceiling} and
\Cref{hyp:freedman} are not comparable and this rung is not a weakening of
rung~3. A clipped likelihood ratio in a Gaussian location model satisfies
\Cref{hyp:freedman} but is unbounded above, so \eqref{eq:h6} holds for no finite
$r$; conversely \eqref{eq:h6} says nothing about how far a single bet may crash.
What the rung trades is a floor and a variance for a ceiling and an observable.
(\Cref{hyp:bounded} with $\bar r\ge\overline L_n-c_n$ is a special case.)

\begin{theorem}[Empirical-Bernstein floor and an observable rejection time]
\label{thm:empbern}
Assume \Cref{hyp:ceiling} and put
\[
   \widehat V_t\ :=\ \sum_{n\le t}\lp c_n-L_n\rp^{2}
   \ =\ \sum_{n\le t}\lp c_n-\log E_n\rp^{2} ,
   \qquad
   \psi_{\mathrm E}(\lambda):=-\log(1-\lambda)-\lambda ,
\]
$\psi_{\mathrm E}$ being a fixed function and not one of the majorants of
\Cref{def:main}, so that $\widehat V_t$ --- the \emph{realised} quadratic variation of the
deficit --- is $\Fc_t$-measurable and computable from the data. The engine is
Fan's one-sided exponential inequality \cite{FGL15}, quoted as
\Cref{thm:app-fan}. Fix constants
$\lambda\in(0,1)$ and $\delta\in(0,1)$. Then, with $\Prob_1$-probability at
least $1-\delta$, simultaneously for all $t\in\N_0$,
\begin{equation}\label{eq:eb-floor}
   \log W_t\ >\ C_t
     -\frac{\bar r}{\lambda}\log\frac1\delta
     -\frac{\psi_{\mathrm E}(\lambda)}{\lambda\,\bar r}\,\widehat V_t
   \ \ge\ C_t-\frac{\bar r}{\lambda}\log\frac1\delta
     -\frac{\lambda\,\widehat V_t}{2(1-\lambda)\,\bar r} .
\end{equation}
Consequently
\begin{equation}\label{eq:eb-stop}
   \Prob_1\lp\tau_\alpha\le\widehat T\rp\ \ge\ 1-\delta ,
   \qquad
   \widehat T:=\inf\lC t\in\N:\
      \frac{\lambda\lp C_t-\ell\rp}{\bar r}\ \ge\
      \log\frac1\delta+\frac{\psi_{\mathrm E}(\lambda)}{\bar r^{2}}\,\widehat V_t\rC ,
\end{equation}
and $\widehat T$ is a stopping time: the experimenter evaluates the certificate
at every $t$ without knowing any population variance.
\end{theorem}

\begin{remark}[What it buys, exactly]\label{rem:eb}
The deficit in the second form of \eqref{eq:eb-floor} is
$D(\lambda)=\frac{\bar r}{\lambda}\log\frac1\delta
+\frac{\lambda\widehat V_t}{2(1-\lambda)\bar r}$. Tuning $\lambda$ in advance to
a target variance $\bar V$ gives a closed form: writing
$\Lambda_\delta:=\log\frac1\delta$ for this remark only,
\begin{equation}\label{eq:eb-tuned}
   \lambda^{\ast}
   =\frac{\bar r\sqrt{2\Lambda_\delta}}
         {\sqrt{\bar V}+\bar r\sqrt{2\Lambda_\delta}} ,
   \qquad
   \min_{\lambda\in(0,1)}\lB
      \frac{\bar r\Lambda_\delta}{\lambda}
      +\frac{\lambda\bar V}{2(1-\lambda)\bar r}\rB
   \ =\ \sqrt{2\,\bar V\,\Lambda_\delta}\ +\ \bar r\,\Lambda_\delta ,
\end{equation}
an identity, not an asymptotic statement. So at a perfect match
$\widehat V_t=\bar V$ the floor is $C_t-\sqrt{2\bar V\log\frac1\delta}
-\bar r\log\frac1\delta$: the first term is what a known variance would cost,
matching the optimised form of \eqref{eq:unif-rung4}, and the additive
$\bar r\log\frac1\delta$ is the price of the ceiling. A mismatch costs more,
increasingly in $\widehat V_t/\bar V$. The gain is that the \emph{realised}
$\widehat V_t$ replaces a worst case: if \Cref{hyp:bounded} also holds, so that
$\bar r$ may be taken to be the range $r$, then $\widehat V_t\le tr^{2}$ and
\Cref{cor:rung5}(a) must pay that bound whatever happens, while
\eqref{eq:eb-floor} pays only what the data show. This is why the betting
constructions of \Cref{ex:betting} perform well in practice \cite{WSR24}: they
have a large range and a small realised deficit.
\end{remark}

\begin{remark}[What is not true]\label{rem:eb-lil}
One cannot optimise $\lambda$ at the realised $\widehat V_t$ and keep the
time-uniform guarantee. The minimiser is a function of the data, and the
resulting claim --- that
$\sum_{n\le t}(c_n-L_n)\le\sqrt{2\widehat V_t\log\frac1\delta}+\bar r
\log\frac1\delta$ for all $t$ --- is false, by the law of the iterated
logarithm. For i.i.d.\ Rademacher deficits, where $\widehat V_t=t$ and
$\bar r=1$, and at $\delta=0.05$, its failure probability is about $0.014$ by
$t=10^{2}$ and about $0.090$ by $t=10^{4}$, and tends to one; for the still
more naive form without the additive $\bar r\log\frac1\delta$ the same
probabilities are $0.063$ and $0.148$.

Two remedies are legitimate. The first is a \emph{predictable} plug-in
$\lambda_n\propto\sqrt{\log\frac1\delta\,/\,\widehat V_{n-1}}$
\cite[\S3.2]{WSR24}. \Cref{thm:app-fan} is a pointwise inequality, so it accepts
any $[0,1)$-valued $\lambda_n$; what predictability buys is the
conditional-expectation step that makes the product a supermartingale. The
conclusion is then the $\lambda$-weighted analogue of \eqref{eq:eb-floor}:
\emph{every} occurrence of $\lambda$ moves inside the sums, giving, with
$\Prob_1$-probability at least $1-\delta$ and simultaneously for all $t$,
\begin{equation}\label{eq:eb-pred}
   \sum_{n\le t}\lambda_nL_n
   \ >\ \sum_{n\le t}\lambda_nc_n-\bar r\log\frac1\delta
     -\frac1{\bar r}\sum_{n\le t}\psi_{\mathrm E}(\lambda_n)
        \lp c_n-L_n\rp^{2} .
\end{equation}
Note that this consumes the individual $c_n$ and not the aggregate $C_t$, and
that its left-hand side is no longer $\log W_t$; converting it into a statement
about $\tau_\alpha$ therefore needs $\lambda_n$ bounded below. The second
remedy is a mixture over $\lambda$, at the usual $\sqrt{\log\log}$ price.
\end{remark}

\begin{remark}[Relation to the literature]\label{rem:eb-prior}
The empirical-Bernstein line \cite{MP09,AMS09,How21,WSR24} is concerned with
confidence sequences and estimation; the point here is that the same
supermartingale yields a type-II statement, and that in this direction the
resulting bound is evaluable online. \Cref{thm:empbern} is closest to
Balsubramani and Ramdas \cite{BR16}, who prove a uniform empirical-Bernstein
inequality and, separately, bound the expected stopping time of a sequential
test --- though there the stopping-time bound comes from comparison with an
oracle batch test rather than from inverting the empirical-Bernstein boundary
itself.
\end{remark}

\section{Composite alternatives}
\label{sec:composite}

Nothing so far has used that the alternative is a single law. This section
records what survives when it is not, and what does not. Let $\Hc_1$ be a set of
probability measures on $(\Omega,\Fc)$, each satisfying \eqref{eq:ac}.

\begin{proposition}[Uniformity is free, once the constants are uniform]
\label{prp:composite}
Let $(E_n)_{n\in\N}$ be a betting sequence. Suppose the aggregate of any one
rung can be chosen uniformly over $\Hc_1$: that is, $(c_n)_{n\in\N}$ satisfies
\eqref{eq:anchor} and $C_t$ satisfies \eqref{eq:aggregate} under every
$\Prob\in\Hc_1$, and the majorant constants are
chosen at their worst case --- for \Cref{hyp:freedman}, both
$w\ge\sup_{\Prob\in\Hc_1}\sum_{n\le t}w_n(\Prob)$ and
$b\ge\sup_{\Prob\in\Hc_1}\max_{n\le t}b_n(\Prob)$, since the exponent
\eqref{eq:rate-rung3} is monotone in each; for \Cref{hyp:subgauss},
$V_t\ge\sup_{\Prob\in\Hc_1}\sum_{n\le t}\sigma_n^{2}(\Prob)$. Then the conclusion of that rung holds verbatim with
$\Prob_1$ replaced by $\sup_{\Prob\in\Hc_1}\Prob$:
\[
   \sup_{\Prob\in\Hc_1}\ \Prob\lp\tau_\alpha>t\rp
   \ \le\ \sup_{\Prob\in\Hc_1}\ \Prob\lp W_t\le\tfrac1\alpha\rp
   \ \le\ \exp\lp-\lp\Psi^{(t)}\rp^{\star}(\Delta_t)\rp .
\]
\end{proposition}

\noindent
The proof is one line --- \Cref{thm:master} is applied under each $\Prob$
separately and the bound does not depend on which --- and that is the whole
point: the ladder needs uniformity of the \emph{aggregate}, never of the
argument. What it costs is that the worst case over $\Hc_1$ is what enters, and
for a rich $\Hc_1$ the worst case can be trivial.

\begin{remark}[Separation is necessary]\label{rem:separation}
Suppose $\Hc_1$ accumulates at the null, in the sense that
$\inf_{\Prob\in\Hc_1}\E_{\Prob}[\log E_n\mid\Fc_{n-1}]$ can be made arbitrarily
close to $0$. Then no uniform drift certificate with $c_n>0$ exists, $\Delta_t=0$
for every $t$, and every bound in \Cref{tab:ladder} is trivial at every
horizon. This is not an artefact of the method. In the i.i.d.\ setting the
\epower{} of any \evar{} is at most $\KL(P_1\|P_0)$
\cite[Lem.~2.20]{CEIL}, which tends to $0$ as $P_1\to P_0$, so the certified
horizon $\ell/c$ diverges; and by \cite[Lem.~5.10]{CEIL} --- for an i.i.d.\ product with
$v:=\Var_{P_1}(\log E)<\infty$ and $c\ell>2v$ --- such a design has
$\bar\gamma_t(\alpha)\ge1-\frac{2v}{c\ell}$ at every $t\le\frac{\ell}{2c}$,
whatever the majorant. A composite alternative must be separated from the null for a
uniform type-II guarantee to exist at all --- exactly as in the fixed-horizon
theory.
\end{remark}

\begin{remark}[The composite ceiling, and which scalar to maximise]
\label{rem:composite-ceiling}
In the i.i.d.\ setting, applying the Chernoff--Stein ceiling
\cite[Prop.~4.15]{CEIL} under each $P_1\in\Hc_1$ with
$\KL(P_0\|P_1)<\infty$ gives
\[
   \liminf_{t\to\infty}\ \frac1t\log
   \sup_{P_1\in\Hc_1}P_1^{\otimes\N}\lp W_t\le\tfrac1\alpha\rp
   \ \ge\ -\inf_{P_1\in\Hc_1}\KL(P_0\|P_1)
\]
for every betting sequence: against the worst member of the alternative, no
design can achieve an exponent better than that of the member \emph{closest to
the null}, which is the hardest to detect. Whether that ceiling is approachable by a
design that does not know which $\Prob$ obtains is a different question, and it
is not settled here.

What can be said is which scalar \emph{not} to maximise. The standard composite
design criterion is growth-rate optimality, maximising
$\inf_{\Prob\in\Hc_1}\E_{\Prob}[\log E]$ over \evars{} \cite{GHK24,LRR25}. By
\cite[Thm.~4.29]{CEIL} that is not the criterion the Chernoff--Stein exponent responds
to, already for a simple alternative: the \epower{} is capped by
$\KL(P_1\|P_0)$ and attained uniquely by the likelihood ratio, while the
exponent is capped by $\KL(P_0\|P_1)$ and attained by nobody. The composite
problem inherits that separation, since it contains the simple one.
\end{remark}

\section{The price of the threshold}
\label{sec:secondorder}

The bounds of \Cref{sec:ladder} depend on the level only through
$\Delta_t=(C_t-\ell)_+$, so the whole effect of $\alpha$ on the exponent is a
shift of the gap. The following says exactly what that shift costs, for every
rung at once, and with no error term.

\begin{proposition}[The level costs a factor $\alpha^{-s^{\ast}}$]
\label{prp:price}
In the homogeneous case of \Cref{cor:iid} --- $C_t=tc$, $\Psi^{(t)}=t\psi$ ---
let $s^{\ast}(c)$ be \emph{any} maximiser in
$\psi^{\star}(c)=\sup_{s}\{sc-\psi(s)\}$, assumed to exist; equivalently, any
element of the superdifferential of $\psi^{\star}$ at $c$. When $\psi$ is
finite, differentiable and strictly convex on $(0,s_{\max})$ this is the unique
solution of $\psi'(s^{\ast})=c$, and when $\psi$ is the chord of
\Cref{cor:rung1} it is $s_0$. Then for every $t$ with $tc\ge\ell$,
\begin{equation}\label{eq:price}
   \gamma_t(\alpha)\ \le\ \alpha^{-s^{\ast}(c)}\,e^{-t\,\psi^{\star}(c)}
   \ =\ \exp\lp-t\,\psi^{\star}(c)+s^{\ast}(c)\,\ell\rp .
\end{equation}
\end{proposition}

\begin{remark}[Reading \eqref{eq:price}]\label{rem:price}
Three things are worth extracting.

First, the asymptotic rate is $\psi^{\star}(c)$ per observation --- the value of
the rate function at the \emph{full} drift, with no reference to the level ---
and the level enters only through the constant $\alpha^{-s^{\ast}}$. Halving
$\alpha$ therefore costs $s^{\ast}\log2$ in the exponent once and for all, not
a fraction of the rate: the optimising tilt is the price per nat of threshold.
This is the structural reason every entry of \Cref{tab:ladder} has a tilt
column.

Second, \eqref{eq:price} is not asymptotic. It follows from convexity of
$\psi^{\star}$ alone --- a convex function lies above its tangent --- so it
holds at every horizon past $\ell/c$, and the loss relative to
\Cref{cor:iid} is the curvature term $\frac{\ell^{2}}{2t}(\psi^{\star})''$,
which vanishes as $t$ grows.

Third, the two factors pull in opposite directions along a ladder. Rung~1 has
$s^{\ast}=s_0$ and $\psi^{\star}(c)=(s_0c-\kappa)_+$ --- for which
\eqref{eq:price} is an identity, not an inequality, since $\psi^{\star}$ is
affine there --- so both shrink together as $s_0$ does; rung~5a has
$s^{\ast}=4c/r^{2}$ and $\psi^{\star}(c)=2c^{2}/r^{2}$, both governed by the
range. A design chosen to maximise the rate is not in
general one that keeps the threshold cheap.
\end{remark}

\begin{remark}[What is left to the companion paper]\label{rem:tocompanion}
\Cref{prp:price} says what a \emph{given} majorant costs. It says nothing about
how large $\psi^{\star}(c)$ can be, which is a question about the \evar{} and
not about the hypothesis, and which the companion paper answers: over all
\evars{} in the i.i.d.\ setting, $\sup_E\Lambda(E)=\KL(P_0\|P_1)$, the supremum
is not attained, and the designs that approach it are the flattened likelihood
ratios $R_\beta=R^{\beta}/Z_\beta$ \cite[Thm.~4.25]{CEIL}. The price of
approaching it is quantified there too: at horizon $t$ the optimal flattening is
$\beta^{\ast}(t)=\sqrt{2\ell/(tV_0)}$ with $V_0=\Var_{P_0}(\log R)$, and the
resulting error obeys
\[
   \gamma_t(\alpha)\ \le\
   \exp\lp-t\,\KL(P_0\|P_1)+\sqrt{2\ell tV_0}+O(1)\rp
\]
\cite[Thm.~5.6]{CEIL}. The two statements compose: \Cref{prp:price} is the
price of the level at a fixed design, and \cite[Thm.~5.6]{CEIL} is the price of
the design at a fixed level. Neither is used in proving the other.
\end{remark}

\section{Discussion}
\label{sec:discussion}

\begin{remark}[The hypotheses are about the bet, not about the problem]
\label{rem:whatabout}
It is easy to read \Cref{sec:ladder} as a list of regularity conditions on the
testing problem, and to object that they are unverifiable. They are not
conditions on the problem. Every rung constrains the \emph{approximation} to
$R=\frac{dP_1}{dP_0}$ that one chooses to bet with, and the exact likelihood
ratio satisfies the weakest of them for nothing: $\E_{P_1}[R^{-1}]=P_0(R>0)\le1$
gives \Cref{hyp:cramer} at $s_0=1$ with $\rho\le1$ (\Cref{rem:lrfree}). What the
ladder measures is therefore the cost of \emph{not} betting the likelihood
ratio --- of using a learned score, a misspecified model, a plug-in estimate ---
and the cost is paid in the crash direction only, which is the asymmetry
\Cref{rem:winning} keeps returning to.

This also explains why the hypotheses can be enforced rather than assumed.
\Cref{ex:clipping} buys rungs~1 and~2 for an arbitrary \evar{};
\Cref{ex:twosided} buys rung~5, also for an arbitrary \evar{}, at a range
growing like $2\log\frac1\nu$; \Cref{ex:betting} lands a bounded stake on
rung~5 by construction. The design question is not whether a rung holds but
which rung to buy, and at what price in \epower{}.
\end{remark}

\begin{remark}[What the reduction does and does not do]\label{rem:reduction}
The technical content of this paper is small and deliberately so.
\Cref{thm:master} is three lines: a conditional Chernoff bound, the tower
property, and Markov's inequality. Everything else is the observation that this
reduction is \emph{lossless} in the sense of \Cref{prp:duality} --- what it
extracts from a constraint on the conditional mean is exactly the value of an
information projection --- so that every loss in \Cref{tab:ladder} is
attributable to the majorant and to nothing else. That is what makes the ladder
a meaningful comparison of hypotheses rather than a list of unrelated
estimates: the rungs differ only in how much a tail hypothesis fails to pin down
the conditional law.

Three limitations should be stated plainly. The bounds are Cram\'er--Chernoff
bounds and therefore discard a polynomial factor; at the top rung the discarded
factor is the Bahadur--Rao prefactor of order $t^{-1/2}$, quantified in
\cite[Rem.~5.9]{CEIL}. The time-uniform forms
\Cref{thm:master}(d)--(e) fix the tilt in advance and so are not adaptive; only
\Cref{sec:observable} removes an unknown, and only one. And the drift
certificate is an assumption throughout: no rung validates it, and a wrong
$c_n$ produces a confident and false conclusion.
\end{remark}

\begin{remark}[Relation to the companion paper]\label{rem:relation}
The two papers answer complementary questions about the same object. This one
asks what exponent can be \emph{certified} from a hypothesis on the lower tail
of $\log E_n$, and answers it with one Legendre transform per hypothesis, under
an arbitrary filtration and with no product structure. The companion
\cite{CEIL} asks how large the exponent can \emph{be}, over all \evars{}, and
answers that it is bounded by $\KL(P_0\|P_1)$ --- the Chernoff--Stein exponent,
in the opposite direction to the divergence that caps the \epower{} --- with the
bound tight and not attained. The two meet at rung~6 and nowhere else: there the
majorant may be taken to be the truth, the certified exponent equals
$\Lambda(E)$, and \cite{CEIL} is what bounds $\Lambda(E)$ in turn. Neither main
theorem is used in proving the other.
\end{remark}

\phantomsection
\addcontentsline{toc}{section}{Acknowledgements}
\section*{Acknowledgements}

This note was prepared with the assistance of Claude Opus 5, a large language
model developed by Anthropic. Based on the author's notes, ideas and input, the
model was used to re-draft and restructure the exposition, to prepare the
typescript, to search for and cross-check references, to check and complete
arguments, and to provide feedback, which the author used to refine further
inputs. All references, statements and proofs have been verified by the author,
who takes full responsibility for the content, including any remaining errors.

\phantomsection

\appendix

\section{Deferred proofs}
\label{app:proofs}

Proofs appear in the order of the statements. Each rung subsection verifies the
majorant, computes the Legendre transform and identifies $s^{\ast}$;
\Cref{app:implications} then compares the hypotheses with one another.

\subsection{The one-step lemma}

\noindent\emph{Proof of \Cref{lem:onestep}.}

Since $E_n^{-s}=e^{-sL_n}=e^{-sc_n}e^{-s(L_n-c_n)}$ and $e^{-sc_n}$ is
$\Fc_{n-1}$-measurable and nonnegative, it comes out of the conditional
expectation:
\[
   \E_{\Prob_1}\lB E_n^{-s}\mid\Fc_{n-1}\rB
   =e^{-sc_n}\,\E_{\Prob_1}\lB e^{-s(L_n-c_n)}\mid\Fc_{n-1}\rB
   \ \le\ e^{-sc_n}e^{\psi_n(s)}
\]
$\Prob_1$-a.s., by \eqref{eq:majorant}. \qed

\subsection{The master principle}

\noindent\emph{Proof of \Cref{thm:master}.}

\emph{Step 0: one supermartingale carries (a) and (d).} Fix
$s\in[0,s_{\max})$ and put
\begin{equation}\label{eq:tilted}
   N_0:=1,
   \qquad
   N_t:=W_t^{-s}\exp\lp s\sum_{n\le t}c_n-\sum_{n\le t}\psi_n(s)\rp
   \qquad(t\in\N) .
\end{equation}
We claim that $(N_t)_{t\in\N_0}$ is a nonnegative $\Prob_1$-supermartingale.
It is adapted and nonnegative. The conditional \evar{} property forces
$E_n<\infty$ $\Prob_0$-a.s., and $P_1\ll P_0$ gives $\Prob_1\ll\Prob_0$ on
$\Fc_n$, so $E_n<\infty$ $\Prob_1$-a.s.; hence
$W_t^{-s}=W_{t-1}^{-s}E_t^{-s}$ holds $\Prob_1$-a.s.\ (the identity can fail
only where one factor is $0$ and another $\infty$), and therefore
$N_t=N_{t-1}\,e^{sc_t-\psi_t(s)}E_t^{-s}$. The factor
$N_{t-1}e^{sc_t-\psi_t(s)}$ is nonnegative and $\Fc_{t-1}$-measurable ---
\emph{this is the only thing predictability of $c_t$ and $\psi_t$ is used
for} --- so \Cref{lem:onestep} gives
\[
   \E_{\Prob_1}\lB N_t\mid\Fc_{t-1}\rB
   =N_{t-1}\,e^{sc_t-\psi_t(s)}\,
     \E_{\Prob_1}\lB E_t^{-s}\mid\Fc_{t-1}\rB
   \ \le\ N_{t-1}
   \qquad\Prob_1\text{-a.s.},
\]
and taking expectations, $\E_{\Prob_1}[N_t]\le\E_{\Prob_1}[N_{t-1}]\le\dots\le
N_0=1$ by induction; in particular each $N_t$ is integrable.

\emph{(a)} By \eqref{eq:aggregate} and $s\ge0$ we have, $\Prob_1$-a.s.,
$s\sum_{n\le t}c_n-\sum_{n\le t}\psi_n(s)\ge sC_t-\Psi^{(t)}(s)$, which is
\eqref{eq:combined}; since $W_t^{-s}\ge0$ this gives
$N_t\ge W_t^{-s}\exp(sC_t-\Psi^{(t)}(s))$, and hence
\[
   \exp\lp sC_t-\Psi^{(t)}(s)\rp\,\E_{\Prob_1}\lB W_t^{-s}\rB
   \ \le\ \E_{\Prob_1}[N_t]\ \le\ 1 ,
\]
which is \eqref{eq:negmom}.

\emph{(b)} The first inequality is \Cref{lem:ordering}. For the second, fix
$s\in[0,s_{\max})$. Since $x\mapsto x^{-s}$ is nonincreasing on $[0,\infty]$,
$\{W_t\le\frac1\alpha\}\subseteq\{W_t^{-s}\ge\alpha^{s}\}$, so Markov's
inequality and \eqref{eq:negmom} give
\[
   \gamma_t(\alpha)\ \le\ \alpha^{-s}\,\E_{\Prob_1}\lB W_t^{-s}\rB
   \ \le\ \exp\lp s\ell-sC_t+\Psi^{(t)}(s)\rp
   =\exp\lp-\lC s(C_t-\ell)-\Psi^{(t)}(s)\rC\rp .
\]
Taking the infimum of the right-hand side over $s\in[0,s_{\max})$ gives
$\gamma_t(\alpha)\le\exp(-(\Psi^{(t)})^{\star}(C_t-\ell))$. If $C_t\ge\ell$ this
is \eqref{eq:master}. If $C_t<\ell$ then $\Delta_t=0$ and, since
$\Psi^{(t)}\ge0$ with $\Psi^{(t)}(0)=0$, we have
$(\Psi^{(t)})^{\star}(0)=\sup_{s}\{-\Psi^{(t)}(s)\}=0$, so the right-hand side of
\eqref{eq:master} equals $1$ and the bound holds trivially.

\emph{(c)} The map $s\mapsto s\Delta_t-\Psi^{(t)}(s)$ is strictly concave and
differentiable on $(0,s_{\max})$ with derivative
$\Delta_t-(\Psi^{(t)})'(s)$, which is strictly decreasing. If
$\Delta_t>\lim_{s\downarrow0}(\Psi^{(t)})'(s)$ the derivative is positive near
$0$, so the supremum is either attained at the unique interior zero
\eqref{eq:sstar} or approached as $s\uparrow s_{\max}$. Otherwise the derivative
is nonpositive throughout, the map is maximised at $s=0$, and its value there is
$-\Psi^{(t)}(0)=0$.

\emph{(d)} Fix $s\in(0,s_{\max})$ and let $(N_t)_{t\in\N_0}$ be the
supermartingale \eqref{eq:tilted} of Step~0.
Ville's inequality (\Cref{thm:app-ville}) gives
$\Prob_1(\sup_tN_t\ge\frac1\delta)\le\delta$. On the complementary event one has,
simultaneously for all $t$,
\[
   -s\log W_t+s\sum_{n\le t}c_n-\sum_{n\le t}\psi_n(s)\ <\ \log\tfrac1\delta ,
\]
that is $\log W_t>\sum_{n\le t}c_n-\frac1s(\sum_{n\le t}\psi_n(s)
+\log\frac1\delta)$; \eqref{eq:combined} then gives \eqref{eq:unif}, the
countably many almost-sure inequalities being intersected.

\emph{(e)} On the same event, $\log W_t>\ell$ --- that is, $W_t>\frac1\alpha$
and hence $\tau_\alpha\le t$ --- as soon as
$C_t-\frac1s(\Psi^{(t)}(s)+\log\frac1\delta)\ge\ell$, which rearranges to the
defining inequality of $T_{s,\delta}$ in \eqref{eq:stop}. \qed

\subsection{The product form}

\noindent\emph{Proof of \Cref{cor:iid}.}

Immediate from \Cref{thm:master} with $C_t=tc$ and $\Psi^{(t)}=t\psi$, together
with the scaling identity
\[
   (t\psi)^{\star}(x)=\sup_{s\ge0}\lC sx-t\psi(s)\rC
   =t\,\sup_{s\ge0}\lC s\tfrac{x}{t}-\psi(s)\rC=t\,\psi^{\star}\!\lp\tfrac xt\rp ,
\]
valid for every $t>0$. The characterisation of $s^{\ast}$ is
\Cref{thm:master}(c). \qed

\subsection{Duality}

\noindent\emph{Proof of \Cref{prp:duality}.}
Write $X:=\log E$, which is real-valued $\Prob_1$-a.s., $\mu:=\E_{\Prob_1}[X]
\in(-\infty,+\infty]$, $x_0:=\Prob_1\text{-}\operatorname{ess\,inf}X
\in[-\infty,\infty)$, and $g(s):=-sc-\Lambda(s)$, so that
$\mathcal J(c)=\sup_{s\ge0}g(s)$.

\emph{Step 1: weak duality, $\mathcal J(c)\le\mathcal D(c)$.}
Let $Q\in\mathcal A_c$ with $\KL(Q\|\Prob_1)<\infty$ and let $s\ge0$; we may
assume $\Lambda(s)<\infty$, the inequality below being vacuous otherwise. The
Donsker--Varadhan formula (\Cref{thm:app-dv}) applied to $\zeta:=-sX$ gives
\begin{equation}\label{eq:dv-step}
   -s\,\E_Q[X]\ \le\ \KL(Q\|\Prob_1)+\Lambda(s) ,
\end{equation}
and $\E_Q[X]\le c$ with $s\ge0$ gives $-sc\le-s\E_Q[X]$, so
$g(s)\le\KL(Q\|\Prob_1)$. Take the supremum over $s$ and then the infimum over
$Q$; if $\mathcal A_c$ contains no $Q$ of finite relative entropy the right-hand
side is $+\infty$ and there is nothing to prove.

\emph{Step 2: the tilted family.}
By H\"older's inequality $\Lambda$ is convex on $[0,\infty)$, with
$\Lambda(0)=0$; it is finite on $[0,s_{\max})$ and differentiable there, with
$\Lambda'(s)=-m(s)$. Hence $m$ is nonincreasing and continuous on
$(0,s_{\max})$ with $m(0^{+})=\mu$, and $g$ is concave with $g(0)=0$ and
$g'(s)=m(s)-c$. Whenever $\Lambda(s)<\infty$,
\begin{equation}\label{eq:tiltKL}
   \KL(Q_s\|\Prob_1)=\E_{Q_s}\lB-sX-\Lambda(s)\rB=-s\,m(s)-\Lambda(s) .
\end{equation}

\emph{Step 3: solvability of the tilt equation implies everything.}
Suppose $m(s^{\ast})=c$ for some $s^{\ast}$ with $\Lambda(s^{\ast})<\infty$.
Then $Q_{s^{\ast}}\in\mathcal A_c$ and, by \eqref{eq:tiltKL},
\[
   \KL(Q_{s^{\ast}}\|\Prob_1)=-s^{\ast}c-\Lambda(s^{\ast})=g(s^{\ast})
   \ \le\ \mathcal J(c)\ \le\ \mathcal D(c)
   \ \le\ \KL(Q_{s^{\ast}}\|\Prob_1) ,
\]
so all four coincide, which is \eqref{eq:duality} and
\eqref{eq:duality-attained}. The minimiser is unique because $\mathcal A_c$ is
convex and $Q\mapsto\KL(Q\|\Prob_1)=\int\varphi(\frac{dQ}{d\Prob_1})d\Prob_1$
is strictly convex on $\{\KL<\infty\}$, $\varphi(u)=u\log u$ being strictly
convex.

\emph{Step 4: the boundary cases (a) and (b).}
If $c\ge\mu$ then $\Prob_1\in\mathcal A_c$, so $\mathcal D(c)=0$; and $g$ is
concave with $g(0)=0$ and $g'(0^{+})=\mu-c\le0$, so $\mathcal J(c)=0$. If
$c<x_0$ then $X\ge x_0>c$ holds $Q$-a.s.\ for every $Q\ll\Prob_1$, so
$\mathcal A_c=\emptyset$ and $\mathcal D(c)=+\infty$; and
$\Lambda(s)\le-sx_0$ for every $s\ge0$ gives $g(s)\ge s(x_0-c)\to+\infty$.

\emph{Step 5: the case $c=x_0>-\infty$.}
Here $\E_{\Prob_1}[e^{-s(X-x_0)}]\downarrow\Prob_1(X=x_0)$ as $s\to\infty$, by
dominated convergence with dominating function $1$, so
$g(s)=-\log\E_{\Prob_1}[e^{-s(X-x_0)}]$ increases to
$\log\frac1{\Prob_1(X=x_0)}=\mathcal J(c)$. If $\Prob_1(X=x_0)>0$ then
$Q:=\Prob_1(\,\cdot\mid X=x_0)$ lies in $\mathcal A_c$ with
$\KL(Q\|\Prob_1)=\log\frac1{\Prob_1(X=x_0)}$, so
$\mathcal D(c)\le\mathcal J(c)\le\mathcal D(c)$ and the infimum is attained; if
$\Prob_1(X=x_0)=0$ then $\mathcal A_c=\emptyset$ and both sides are $+\infty$.

\emph{Step 6: the case $x_0<c<\mu$ with $s_{\max}=0$.}
Here $\Lambda(s)=+\infty$ for every $s>0$, so
$\mathcal J(c)=g(0)=0$ and it remains to show $\mathcal D(c)=0$; the infimum is
not attained, since $\KL(Q\|\Prob_1)=0$ forces $Q=\Prob_1\notin\mathcal A_c$.
As in Step~7 below, $\Lambda(s)=+\infty$ for every $s>0$ gives
$M_j\uparrow\infty$ with $\log\frac1{\Prob_1(X\le-M_j)}\le\delta M_j$, for each
fixed $\delta>0$. Fix $K$ large enough that
$m_K:=\E_{\Prob_1}[X\mid X\le K]$ satisfies $c<m_K<\infty$, which is possible
because $\E_{\Prob_1}[X^{-}]<\infty$ and $m_K\uparrow\mu>c$ as $K\uparrow
\infty$. With $P^{K}:=\Prob_1(\,\cdot\mid X\le K)$,
$R_j:=\Prob_1(\,\cdot\mid X\le-M_j)$ and
$\theta_j:=\frac{m_K-c}{m_K+M_j}$, the mixture
$Q_j:=(1-\theta_j)P^{K}+\theta_jR_j$ satisfies $\E_{Q_j}[X]\le
(1-\theta_j)m_K-\theta_jM_j=c$, hence $Q_j\in\mathcal A_c$, and convexity of
$\KL(\,\cdot\,\|\Prob_1)$ gives
\[
   \KL(Q_j\|\Prob_1)\ \le\ \log\frac1{\Prob_1(X\le K)}
      +\theta_j\,\delta M_j
   \ \xrightarrow[\ j\to\infty\ ]{}\
   \log\frac1{\Prob_1(X\le K)}+\delta\lp m_K-c\rp .
\]
Letting $\delta\downarrow0$ and then $K\uparrow\infty$ gives
$\mathcal D(c)=0$.

\emph{Step 7: the main case $x_0<c<\mu$ with $s_{\max}>0$.}
Here $g$ is finite on a neighbourhood of $0$ with $g'(0^{+})=\mu-c>0$, so
$\mathcal J(c)>0$. If $m(s)\le c$ for some
$s\in(0,s_{\max})$ then, $m$ being continuous with $m(0^{+})=\mu>c$, the
intermediate value theorem supplies $s^{\ast}$ with $m(s^{\ast})=c$ and Step~3
applies. So assume
\begin{equation}\label{eq:noroot}
   m(s)>c\qquad\text{for every }s\in(0,s_{\max}) ,
\end{equation}
whence $g$ is strictly increasing on $[0,s_{\max})$ and
$\mathcal J(c)=\lim_{s\uparrow s_{\max}}g(s)$. We show first that this forces
$s_{\max}<\infty$ and $\Lambda(s_{\max})<\infty$.

If $s_{\max}=\infty$, fix $c'\in(x_0,c)$, so that $\Prob_1(X\le c')>0$. Then
$\Lambda(s)\ge\log\lp e^{-sc'}\Prob_1(X\le c')\rp$, hence
$g(s)\le-s(c-c')-\log\Prob_1(X\le c')\to-\infty$, contradicting that $g$
increases from $g(0)=0$. If $s_{\max}<\infty$ and $\Lambda(s_{\max})=\infty$
then $\Lambda(s)\uparrow\infty$ as $s\uparrow s_{\max}$ by monotone convergence,
while $-sc$ stays bounded, so again $g(s)\to-\infty$; the same contradiction.
Note also that $s_{\max}<\infty$ forces $x_0=-\infty$, since $x_0>-\infty$ gives
$\Lambda(s)\le-sx_0<\infty$ for every $s$.

So $s_{\max}<\infty$ and $\Lambda(s_{\max})<\infty$. By \eqref{eq:noroot} and
monotonicity, $m(s_{\max}^{-})\ge c>-\infty$; finiteness of that limit means
$\E_{\Prob_1}[|X|e^{-s_{\max}X}]<\infty$, so dominated convergence gives
$m(s_{\max})=m(s_{\max}^{-})\ge c$ and, $\Lambda$ being continuous from the
left at $s_{\max}$ by monotone convergence,
\begin{equation}\label{eq:Jboundary}
   \mathcal J(c)=g(s_{\max})=-s_{\max}c-\Lambda(s_{\max}) .
\end{equation}

\emph{Step 8: the mixing construction.}
It remains to show $\mathcal D(c)\le\mathcal J(c)$ in the situation of Step~7.
Since $\Lambda(s)=+\infty$ for $s>s_{\max}$, for every $\delta>0$ there is no
$M_0$ with $\Prob_1(X\le-M)\le e^{-(s_{\max}+\delta)M}$ for all $M\ge M_0$ ---
such a bound would make $\E_{\Prob_1}[e^{-sX}]$ finite for
$s<s_{\max}+\delta$ by the layer-cake formula. Hence there are $M_j\uparrow
\infty$ with
\begin{equation}\label{eq:tailrate}
   \log\frac1{\Prob_1(X\le-M_j)}\ \le\ (s_{\max}+\delta)M_j ,
\end{equation}
and in particular $\Prob_1(X\le-M_j)>0$, which is consistent with
$x_0=-\infty$. Put $R_j:=\Prob_1(\,\cdot\mid X\le-M_j)$ and
$m_j:=\E_{R_j}[X]\in[-\infty,-M_j]$.

If $m_j=-\infty$ for some $j$, then $Q^{(\theta)}:=(1-\theta)Q_{s_{\max}}
+\theta R_j$ lies in $\mathcal A_c$ for every $\theta\in(0,1]$, and letting
$\theta\downarrow0$ gives, by convexity of $\KL(\,\cdot\,\|\Prob_1)$,
$\mathcal D(c)\le\KL(Q_{s_{\max}}\|\Prob_1)=-s_{\max}m(s_{\max})-\Lambda(s_{\max})
\le\mathcal J(c)$ by \eqref{eq:Jboundary} and $m(s_{\max})\ge c$.

Otherwise all $m_j$ are finite. Put
\[
   \theta_j:=\frac{m(s_{\max})-c}{m(s_{\max})-m_j}\in(0,1] ,
   \qquad
   Q_j:=(1-\theta_j)Q_{s_{\max}}+\theta_jR_j ,
\]
so that $\E_{Q_j}[X]=(1-\theta_j)m(s_{\max})+\theta_jm_j=c$ and
$Q_j\in\mathcal A_c$. Convexity of $\KL(\,\cdot\,\|\Prob_1)$,
\eqref{eq:tiltKL}, \eqref{eq:tailrate} and $m_j\le-M_j$ give
\begin{align*}
   \KL(Q_j\|\Prob_1)
   &\le\ \KL(Q_{s_{\max}}\|\Prob_1)+\theta_j\KL(R_j\|\Prob_1)\\
   &\le\ -s_{\max}m(s_{\max})-\Lambda(s_{\max})
      +\frac{\lp m(s_{\max})-c\rp(s_{\max}+\delta)M_j}{m(s_{\max})+M_j} .
\end{align*}
The last fraction tends to $(m(s_{\max})-c)(s_{\max}+\delta)$ as $j\to\infty$,
so
\[
   \mathcal D(c)\ \le\ -s_{\max}m(s_{\max})-\Lambda(s_{\max})
      +\lp m(s_{\max})-c\rp(s_{\max}+\delta)
   =\mathcal J(c)+\delta\lp m(s_{\max})-c\rp ,
\]
using \eqref{eq:Jboundary}. Letting $\delta\downarrow0$ gives
$\mathcal D(c)\le\mathcal J(c)$ and hence \eqref{eq:duality}.

\emph{Step 9: non-attainment in the situation of Step~7.}
Suppose $m(s_{\max})>c$ and some $Q\in\mathcal A_c$ had
$\KL(Q\|\Prob_1)=\mathcal J(c)$. Running \eqref{eq:dv-step} at $s=s_{\max}$,
\[
   \mathcal J(c)=-s_{\max}c-\Lambda(s_{\max})
   \ \le\ -s_{\max}\E_Q[X]-\Lambda(s_{\max})
   \ \le\ \KL(Q\|\Prob_1)=\mathcal J(c) ,
\]
so both inequalities are equalities. The second is the equality case
\eqref{eq:dv-gap} of \Cref{thm:app-dv}, which forces $Q=Q_{s_{\max}}$; the first forces
$\E_Q[X]=c$, that is $m(s_{\max})=c$, a contradiction --- note $s_{\max}>0$ is
what makes the first inequality informative. Together with Step~3 this proves
the attainment criterion in (e); Steps~4, 5, 6 and~8 prove (a), (b), (c), (d)
and the identity \eqref{eq:duality} in every case. \qed

\subsection{Two elementary monotonicities}
\label{app:elem}

\begin{lemma}[Monotonicity of the Bennett quotient]\label{lem:elem}
The map
\[
   g(z):=\frac{e^{z}-z-1}{z^{2}}\quad(z\ne0),
   \qquad g(0):=\tfrac12 ,
\]
is continuous, positive and strictly increasing on $\R$. Consequently, for
$s\ge0$ and $y\le b$, $\ \frac{e^{sy}-sy-1}{y^{2}}\le\frac{e^{sb}-sb-1}{b^{2}}$.
\end{lemma}

\noindent\emph{Proof.}
Continuity and positivity follow from the power series
$g(z)=\sum_{j\ge0}z^{j}/(j+2)!$. For $z\ne0$,
\[
   g'(z)=\frac{z^{2}e^{z}-2z\lp e^{z}-z-1\rp}{z^{4}}=\frac{v(z)}{z^{3}},
   \qquad v(z):=ze^{z}-2e^{z}+z+2 .
\]
Now $v(0)=0$, $v'(z)=(z-1)e^{z}+1$ and $v''(z)=ze^{z}$, so $v'$ is decreasing on
$(-\infty,0)$ and increasing on $(0,\infty)$, whence $v'\ge v'(0)=0$ with
equality only at $z=0$; thus $v$ is strictly increasing with $v(0)=0$, so $v(z)$
and $z^{3}$ have the same sign and $g'>0$ on $\R\setminus\{0\}$. \qed

\begin{lemma}[The constant $c_p$]\label{lem:cp}
For $p\in(1,2]$ put $c_p:=\sup_{v>0}\frac{e^{-v}-1+v}{v^{p}}$. Then
$c_p\le2^{1-p}<\infty$, and $c_2=\frac12$.
\end{lemma}

\noindent\emph{Proof.}
Since $e^{-v}\le1-v+\frac{v^{2}}{2}$ and $e^{-v}\le1$ for $v\ge0$, we have
$e^{-v}-1+v\le\min\{\frac{v^{2}}{2},v\}$. For $v\le2$ this gives
$\frac{e^{-v}-1+v}{v^{p}}\le\frac{v^{2-p}}{2}\le\frac{2^{2-p}}{2}=2^{1-p}$, and
for $v\ge2$ it gives $\frac{e^{-v}-1+v}{v^{p}}\le v^{1-p}\le2^{1-p}$, using
$1-p<0$. For $p=2$ the first bound is $\frac12$ for every $v$ and is approached
as $v\downarrow0$. \qed

\subsection{Rung 1}
\label{app:rung1}

\noindent\emph{Proof of \Cref{cor:rung1}.}

\emph{Nonnegativity of $\kappa_n$.} Conditional Jensen for the convex map
$x\mapsto e^{-s_0x}$ and \eqref{eq:anchor} give
\[
   \E_{\Prob_1}\lB E_n^{-s_0}\mid\Fc_{n-1}\rB
   \ \ge\ \exp\lp-s_0\,\E_{\Prob_1}\lB L_n\mid\Fc_{n-1}\rB\rp
   \ \ge\ e^{-s_0c_n}
   \qquad\Prob_1\text{-a.s.,}
\]
so $\rho_n\ge e^{-s_0c_n}$ and $\kappa_n=\log\rho_n+s_0c_n\ge0$.

\emph{The majorant.} Evaluating \eqref{eq:cgf} at $s_0$,
\[
   \varphi_n(s_0)
   =s_0c_n+\log\E_{\Prob_1}\lB E_n^{-s_0}\mid\Fc_{n-1}\rB
   \ \le\ s_0c_n+\log\rho_n=\kappa_n .
\]
For fixed $\omega$ the map $s\mapsto\varphi_n(s)$ is convex with
$\varphi_n(0)=0$, so it lies below the chord joining $(0,0)$ to
$(s_0,\varphi_n(s_0))$ on $[0,s_0]$:
\[
   \varphi_n(s)\ \le\ \frac{s}{s_0}\,\varphi_n(s_0)\ \le\ \frac{s}{s_0}\kappa_n
   \qquad(0\le s\le s_0) .
\]
Beyond $s_0$ the value $+\infty$ majorises trivially, so
\eqref{eq:maj-rung1} is a majorant with $s_{\max}=\infty$, and
$\Psi^{(t)}(s)=\frac{s}{s_0}K_t$ on $[0,s_0]$, $+\infty$ beyond, dominates
$\sum_{n\le t}\psi_n$.

\emph{The transform.} For $x\ge0$,
\[
   \lp\Psi^{(t)}\rp^{\star}(x)
   =\sup_{0\le s\le s_0}\lC sx-\frac{s}{s_0}K_t\rC
   =s_0\lp x-\frac{K_t}{s_0}\rp_+
   =\lp s_0x-K_t\rp_+ ,
\]
the supremum of a linear function of $s$, attained at $s=s_0$ if the slope
$x-K_t/s_0$ is positive and at $s=0$ otherwise. \Cref{thm:master}(b) now gives
\eqref{eq:bound-rung1}.

\emph{The transparent form.} Taking $C_t=\sum_{n\le t}c_n$ and
$K_t=\sum_{n\le t}\kappa_n$ in \eqref{eq:negmom} at $s=s_0$,
\[
   \E_{\Prob_1}\lB W_t^{-s_0}\rB
   \ \le\ \exp\lp-s_0C_t+K_t\rp
   =\exp\lp\sum_{n\le t}\log\rho_n\rp=\prod_{n\le t}\rho_n ,
\]
and Markov's inequality as in the proof of \Cref{thm:master}(b) gives
$\gamma_t(\alpha)\le\alpha^{-s_0}\prod_{n\le t}\rho_n$. \qed

\begin{remark}[The Jensen floor]\label{rem:jensenfloor}
The computation above says more than $\kappa_n\ge0$: it says
$\log\frac1{\rho_n}\le s_0c_n$, so no choice of $s_0$ extracts a per-step rate
exceeding $s_0$ times the certified \epower{}, and the rate \emph{per unit of
tilt} degrades as $s_0$ grows. That is the trade-off \eqref{eq:crit} records: a
large $s_0$ makes the threshold term $s_0\ell$ expensive, a small one makes the
rate $\log\frac1{\rho_n}$ small. It also shows that the condition $\rho_n<1$
means more than positive \epower{}: it says the \epower{} survives the
pessimistic reweighting by $E_n^{-s_0}$. Under $\Prob_0$ the analogous condition
can never hold, since $\E_{\Prob_0}[E_n\mid\Fc_{n-1}]\le1$ and Jensen give
$\E_{\Prob_0}[E_n^{-s}\mid\Fc_{n-1}]\ge1$; the asymmetry between the two
measures is the whole content of the hypothesis.
\end{remark}

\subsection{\texorpdfstring{Rungs 2 and $2''$}{Rungs 2 and 2''}}
\label{app:rung2}

\begin{proposition}[Equivalent forms of Cram\'er's condition]
\label{prp:cramer-equiv}
Let $E$ be a nonnegative random variable on $(\Omega,\Fc,\Prob_1)$. The
following are equivalent, and the optimal exponent in \ref{it:c1} equals
$s_{\max}:=\sup\lC s\ge0\st\E_{\Prob_1}[E^{-s}]<\infty\rC$:
\begin{enumerate}[label=(\roman*),leftmargin=2.4em]
\item\label{it:c1} there are $K'<\infty$ and $a>0$ with
      $\Prob_1(E\le\varepsilon)\le K'\varepsilon^{a}$ for all $\varepsilon>0$;
\item\label{it:c2} there are $K<\infty$, $a>0$ and $u\in\R$ with
      $\Prob_1(\log E\le u-x)\le Ke^{-ax}$ for all $x\ge0$;
\item\label{it:c3} $\E_{\Prob_1}[E^{-s_0}]<\infty$ for some $s_0>0$.
\end{enumerate}
\end{proposition}

\noindent\emph{Proof.}
\ref{it:c1}$\Leftrightarrow$\ref{it:c2} is the substitution
$\varepsilon=e^{u-x}$, with $K=K'e^{au}$.
\ref{it:c3}$\Rightarrow$\ref{it:c1}: Markov's inequality applied to $E^{-s_0}$
gives $\Prob_1(E\le\varepsilon)=\Prob_1(E^{-s_0}\ge\varepsilon^{-s_0})\le
\varepsilon^{s_0}\E_{\Prob_1}[E^{-s_0}]$, so \ref{it:c1} holds with $a=s_0$.
\ref{it:c1}$\Rightarrow$\ref{it:c3}: for $0<s<a$, by the layer-cake formula,
\[
   \E_{\Prob_1}\lB E^{-s}\rB
   =\int_0^{\infty}\Prob_1\lp E<u^{-1/s}\rp du
   \ \le\ 1+\int_1^{\infty}K'u^{-a/s}du
   =1+\frac{K's}{a-s}\ <\ \infty .
\]
The same two displays show that the set of admissible $a$ in \ref{it:c1} and the
set of $s$ with $\E_{\Prob_1}[E^{-s}]<\infty$ have the same supremum. \qed

\begin{lemma}[Majorant for rung 2]\label{lem:maj-crash}
Under \Cref{hyp:crash}, $\Prob_1$-a.s.\ for every $n\in\N$ and $s\in[0,a)$,
\[
   \varphi_n(s)\ \le\ c_p\,M_{p,n}\,s^{p}+\frac{K_n\,s^{2}}{a(a-s)} .
\]
\end{lemma}

\noindent\emph{Proof.}
Write $Y:=c_n-L_n$, so that $\varphi_n(s)=\log\E_{\Prob_1}[e^{sY}\mid\Fc_{n-1}]$;
the crash side is $Y_+=(c_n-L_n)_+$ and the winning side is
$Y_-=(L_n-c_n)_+$. By \eqref{eq:anchor},
$\E_{\Prob_1}[Y\mid\Fc_{n-1}]\le0$, so
\begin{align*}
   \E_{\Prob_1}\lB e^{sY}\mid\Fc_{n-1}\rB
   &=1+s\,\E_{\Prob_1}\lB Y\mid\Fc_{n-1}\rB
     +\E_{\Prob_1}\lB e^{sY}-1-sY\mid\Fc_{n-1}\rB\\
   &\le\ 1+\E_{\Prob_1}\lB e^{sY}-1-sY\mid\Fc_{n-1}\rB ,
\end{align*}
the integrand being nonnegative. Split according to the sign of $Y$.

On $\{Y\le0\}$ we have $e^{sY}-1-sY=e^{-sY_-}-1+sY_-\le c_p(sY_-)^{p}$ by the
definition of $c_p$ (\Cref{lem:cp}), so this part contributes at most
$c_ps^{p}\E_{\Prob_1}[Y_-^{p}\mid\Fc_{n-1}]\le c_ps^{p}M_{p,n}$ by
\eqref{eq:h2b}.

On $\{Y>0\}$ let $f(x):=e^{sx}-1-sx$, so $f(0)=0$ and $f'(x)=s(e^{sx}-1)\ge0$
for $x\ge0$. The layer-cake formula and \eqref{eq:h2} give, for $s<a$,
\begin{align*}
   \E_{\Prob_1}\lB f(Y)\mathbf 1_{\{Y>0\}}\mid\Fc_{n-1}\rB
   &=\int_0^{\infty}f'(x)\,\Prob_1\lp Y>x\mid\Fc_{n-1}\rp dx\\
   &\le\ K_n\int_0^{\infty}s\lp e^{sx}-1\rp e^{-ax}dx
   =\frac{K_ns^{2}}{a(a-s)} ,
\end{align*}
since $\int_0^{\infty}(e^{sx}-1)e^{-ax}dx=\frac1{a-s}-\frac1a=\frac{s}{a(a-s)}$.
Adding the two parts and using $\log(1+u)\le u$ finishes the proof. \qed

\noindent\emph{Proof of \Cref{cor:rung2}.}
\Cref{lem:maj-crash} gives the majorant; the aggregate is immediate since
$\psi_n$ depends on $n$ only through the additive constants $M_{p,n},K_n$, both
of which sum. \Cref{thm:master}(b) then gives the first inequality of
\eqref{eq:bound-rung2}, and \Cref{thm:master}(c) gives \eqref{eq:sstar-rung2-exact}: the derivative
$(\Psi^{(t)})'(s)=pc_pM_ps^{p-1}+\frac{Ks(2a-s)}{a(a-s)^{2}}$ is continuous and
strictly increasing on $(0,a)$, with limits $0$ at $s\downarrow0$ and, when
$K>0$, $+\infty$ at $s\uparrow a$, so it takes the value $\Delta_t>0$ exactly
once. (When $K=0$ and $p<2$ the derivative stays bounded and for large
$\Delta_t$ the supremum is approached as $s\uparrow a$ instead.)

For the surrogate, evaluate the supremand at $s=s_{\ast}$ of
\eqref{eq:sstar-rung2}. Since $s_{\ast}^{p-1}\le\frac{\Delta_t}{4c_pM_p}$,
\[
   c_pM_ps_{\ast}^{p}=c_pM_ps_{\ast}\cdot s_{\ast}^{p-1}
   \ \le\ \frac{s_{\ast}\Delta_t}{4} ,
\]
and since $s_{\ast}\le\frac a2$ we have $a-s_{\ast}\ge\frac a2$, so
$s_{\ast}\le\frac{a^{2}\Delta_t}{8K}$ gives
\[
   \frac{Ks_{\ast}^{2}}{a(a-s_{\ast})}\ \le\ \frac{2Ks_{\ast}^{2}}{a^{2}}
   \ \le\ \frac{2Ks_{\ast}}{a^{2}}\cdot\frac{a^{2}\Delta_t}{8K}
   =\frac{s_{\ast}\Delta_t}{4} .
\]
Hence the supremand is at least
$s_{\ast}\Delta_t-\frac{s_{\ast}\Delta_t}{4}-\frac{s_{\ast}\Delta_t}{4}
=\frac12s_{\ast}\Delta_t$. If the first entry of \eqref{eq:sstar-rung2} is the
smallest then $\frac12s_{\ast}\Delta_t
=\frac12\Delta_t(\frac{\Delta_t}{4c_pM_p})^{q-1}
=\frac{\Delta_t^{q}}{2(4c_pM_p)^{q-1}}$, using $\frac1{p-1}=q-1$ and
$1+\frac1{p-1}=q$, which is \Cref{rem:rung2-regimes}; for $p=2$ this is
$\Delta_t^{2}/(4M_2)$ by $c_2=\frac12$. \qed

\begin{proposition}[$p=1$ fails, and the exponent $q$ is right]
\label{prp:p1-fails}
Fix $a>0$, $\Delta>0$ and $c>\Delta$, and put $\ell:=c-\Delta>0$. For
$\epsilon\in(0,\frac12]$ let $\Omega=\{\omega_0,\omega_1,\omega_2\}$ and
\[
   E(\omega_0):=0,\qquad
   E(\omega_1):=e^{c-\Delta},\qquad
   E(\omega_2):=e^{c+(1-\epsilon)\Delta/\epsilon},
\]
with $\Prob_1(\omega_0)=0$, $\Prob_1(\omega_1)=1-\epsilon$,
$\Prob_1(\omega_2)=\epsilon$; choosing $\Prob_0$ with enough mass on $\omega_0$
and positive mass on $\omega_1,\omega_2$ makes $E$ an \evar{} with
$\Prob_1\ll\Prob_0$. Then $\E_{\Prob_1}[\log E]=c$, \eqref{eq:h2} holds with
$K=e^{a\Delta}$ uniformly in $\epsilon$, and, with $t=1$:
\begin{enumerate}[label=(\roman*),leftmargin=2.4em]
\item $\E_{\Prob_1}[(\log E-c)_+]=(1-\epsilon)\Delta\le\Delta$, while
      $\gamma_1(\alpha)=1-\epsilon\to1$ as $\epsilon\downarrow0$. Hence no bound
      depending only on $(a,K,M_1,\Delta)$ and positive for $\Delta>0$ can hold:
      with $p=1$ the hypothesis is insufficient.
\item For $p\in(1,2]$, $M_p=\epsilon^{1-p}(1-\epsilon)^{p}\Delta^{p}$, so that
      $\frac{\Delta^{q}}{M_p^{q-1}}=\epsilon(1-\epsilon)^{-q}$, while the exact
      exponent is $-\log\gamma_1(\alpha)=\log\frac1{1-\epsilon}$. The ratio of
      the two tends to $1$ as $\epsilon\downarrow0$: on this family the
      guaranteed exponent $\Delta^{q}/(2(4c_pM_p)^{q-1})$ has the exact order in
      $M_p$, so the exponent $q=\frac{p}{p-1}$ cannot be improved and only the
      constant is at issue.
\end{enumerate}
\end{proposition}

\noindent\emph{Proof.}
$\E_{\Prob_1}[\log E]=(1-\epsilon)(c-\Delta)+\epsilon(c+\frac{(1-\epsilon)
\Delta}{\epsilon})=c$. For the crash tail, $\Prob_1(\log E\le c-x)$ equals
$1-\epsilon$ for $x\le\Delta$ and $0$ for $x>\Delta$, and
$(1-\epsilon)\le e^{a\Delta}e^{-ax}$ on $[0,\Delta]$. With
$\ell=c-\Delta$ the threshold is $\frac1\alpha=e^{c-\Delta}$, so
$\gamma_1(\alpha)=\Prob_1(E\le e^{c-\Delta})=1-\epsilon$. Part (i) is then
immediate from $\E_{\Prob_1}[(\log E-c)_+]=\epsilon\cdot
\frac{(1-\epsilon)\Delta}{\epsilon}$. For (ii),
$M_p=\epsilon(\frac{(1-\epsilon)\Delta}{\epsilon})^{p}
=\epsilon^{1-p}(1-\epsilon)^{p}\Delta^{p}$, hence
$M_p^{q-1}=\epsilon^{(1-p)(q-1)}(1-\epsilon)^{p(q-1)}\Delta^{p(q-1)}
=\epsilon^{-1}(1-\epsilon)^{q}\Delta^{q}$ because $(p-1)(q-1)=1$ and
$p(q-1)=q$; the stated ratio follows, and
$\log\frac1{1-\epsilon}=\epsilon+O(\epsilon^{2})$. \qed

\noindent\emph{Proof of \Cref{cor:rung2pp}.}
Write $Y:=c_n-L_n$ as before. By the layer-cake formula and \eqref{eq:h2},
for every integer $k\ge2$,
\[
   \E_{\Prob_1}\lB (Y_+)^{k}\mid\Fc_{n-1}\rB
   =\int_0^{\infty}k\,x^{k-1}\,\Prob_1\lp Y>x\mid\Fc_{n-1}\rp dx
   \ \le\ K_n\int_0^{\infty}k\,x^{k-1}e^{-ax}dx=\frac{K_n\,k!}{a^{k}} .
\]
Next, $e^{z}\le1+z+\frac{z^{2}}{2}+\sum_{k\ge3}\frac{(z_+)^{k}}{k!}$ for every
$z\in\R$: for $z\ge0$ this is the exponential series, and for $z\le0$ it is
$e^{z}\le1+z+\frac{z^{2}}{2}$. Applying it to $z=sY$ with $0\le s<a$, taking
conditional expectations, dropping the linear term because
$\E_{\Prob_1}[Y\mid\Fc_{n-1}]\le0$, and bounding the quadratic term by
\eqref{eq:h2c} and the terms of order $k\ge3$ by the display above,
\[
   \E_{\Prob_1}\lB e^{sY}\mid\Fc_{n-1}\rB
   \ \le\ 1+\frac{w_{0,n}s^{2}}{2}
      +K_n\sum_{k\ge3}\lp\frac sa\rp^{k} ,
\]
and $\log(1+u)\le u$ gives \eqref{eq:maj-rung2pp}; the series sums to
$\frac{(s/a)^{3}}{1-s/a}$ for $s<a$. The aggregate is immediate, since the
constants enter additively. The function $\Psi^{(t)}$ is a power series in $s$
with nonnegative coefficients, hence convex with $\Psi^{(t)}(0)=0$, and
\[
   \lp\Psi^{(t)}\rp'(a\,u)=a\,w_0\,u
      +\frac Ka\sum_{k\ge3}k\,u^{k-1}
   =a\,w_0\,u+\frac Ka\cdot\frac{u^{2}(3-2u)}{(1-u)^{2}} ,
\]
writing $u:=s/a\in[0,1)$ and using
$\sum_{k\ge3}k\,u^{k-1}=\frac{d}{du}\frac{u^{3}}{1-u}
=\frac{u^{2}(3-2u)}{(1-u)^{2}}$. When $K>0$ this is continuous and
strictly increasing on $[0,1)$ from $0$ to $+\infty$, so
\eqref{eq:sstar-rung2pp} has a unique root and \Cref{thm:master}(c) applies.

For \eqref{eq:rung2pp-small} take $s:=\Delta_t/w_0$, admissible because
$\Delta_t\le\frac{aw_0}{2}$ gives $s\le\frac a2<a$. Then
$s\Delta_t-\frac{w_0s^{2}}{2}=\frac{\Delta_t^{2}}{2w_0}$, while
$(1-\frac sa)^{-1}\le2$ bounds the remaining term by
$2K(\frac sa)^{3}=\frac{2K\Delta_t^{3}}{(aw_0)^{3}}$; subtracting gives
\eqref{eq:rung2pp-small}. For \eqref{eq:rung2pp-large} take
$s:=a(1-\sqrt{K/(a\Delta_t)})\in(0,a)$, which is admissible because
$a\Delta_t>K$. Then $s\Delta_t=a\Delta_t-\sqrt{K a\Delta_t}$,
$\frac{w_0s^{2}}{2}\le\frac{a^{2}w_0}{2}$, and
$K(\frac sa)^{3}(1-\frac sa)^{-1}\le K(1-\frac sa)^{-1}=\sqrt{Ka\Delta_t}$.
\qed

\begin{remark}[An optional second majorant]\label{rem:optional}
Discarding \eqref{eq:h2c} altogether and using only the tail bound, the same
layer-cake computation gives, for $0\le s<a$,
\[
   \varphi_n(s)\le\log\E_{\Prob_1}\lB e^{sY_+}\mid\Fc_{n-1}\rB
   \le\log\lp1+\frac{K_ns}{a-s}\rp=:\tilde\psi_n(s) .
\]
\Cref{def:main} requires no convexity, so $\min\{\psi_n,\tilde\psi_n\}$ is
again a majorant and the two may be combined. Since
$\tilde\psi_n'(0^{+})=K_n/a>0$, the second majorant is linear at the origin and
carries a critical horizon of its own; it can only help at large gaps.
\end{remark}

\subsection{Rung 3}
\label{app:rung3}

\begin{lemma}[A wealth floor supplies Bernstein moments]
\label{lem:floor-bernstein}
\Cref{hyp:freedman} implies the moment condition of \Cref{rem:rung3prime} with
$\tau=b/3$, where $b:=\sup_nb_n$, provided $b<\infty$.
\end{lemma}

\noindent\emph{Proof.}
The floor gives $(c_n-L_n)_+\le b_n\le b$, so for every integer $k\ge3$,
\[
   \E_{\Prob_1}\lB\lp(c_n-L_n)_+\rp^{k}\mid\Fc_{n-1}\rB
   \ \le\ b^{\,k-2}\,\E_{\Prob_1}\lB(L_n-c_n)^{2}\mid\Fc_{n-1}\rB
   \ \le\ b^{\,k-2}w_n ,
\]
and $b^{k-2}\le\frac{k!}{2}(b/3)^{k-2}$ is exactly $3^{k-2}\le\frac{k!}{2}$,
which holds for every $k\ge3$ with equality at $k=3$ and strictly thereafter,
since the quotient $\frac{k!}{2\cdot3^{k-2}}$ equals $1$ at $k=3$ and is
multiplied by $\frac{k+1}{3}>1$ at each further step. \qed

\noindent\emph{Proof of \Cref{cor:rung3}.}

\emph{The majorant.} Write $Y:=c_n-L_n$, so $Y\le b_n$ $\Prob_1$-a.s.\ by
\eqref{eq:h3}, $\E_{\Prob_1}[Y\mid\Fc_{n-1}]\le0$ by \eqref{eq:anchor}, and
$\E_{\Prob_1}[Y^{2}\mid\Fc_{n-1}]\le w_n$. As in \Cref{lem:maj-crash},
\[
   \E_{\Prob_1}\lB e^{sY}\mid\Fc_{n-1}\rB
   \ \le\ 1+\E_{\Prob_1}\lB Y^{2}\,g_s(Y)\mid\Fc_{n-1}\rB ,
   \qquad g_s(y):=\frac{e^{sy}-sy-1}{y^{2}} ,
\]
and $g_s$ is nondecreasing by \Cref{lem:elem}, so $g_s(Y)\le g_s(b_n)$ and the
right-hand side is at most $1+w_ng_s(b_n)$. Taking logarithms and using
$\log(1+u)\le u$ gives \eqref{eq:maj-rung3}.

\emph{The aggregate.} With $b\ge\max_{n\le t}b_n$, \Cref{lem:elem} gives
$g_s(b_n)\le g_s(b)$, hence
$\sum_{n\le t}\psi_n(s)=s^{2}\sum_{n\le t}w_ng_s(b_n)\cdot\frac{1}{s^{2}}
\le\frac{w}{b^{2}}(e^{sb}-sb-1)$ for $w\ge\sum_{n\le t}w_n$.

\emph{The transform.} Put $\Psi(s)=\frac{w}{b^{2}}(e^{sb}-sb-1)$. Then
$\Psi'(s)=\frac{w}{b}(e^{sb}-1)$, so $\Psi'(s)=x$ at
$s^{\ast}=\frac1b\log(1+\frac{bx}{w})$, which is \eqref{eq:sstar-rung3}, and
$e^{s^{\ast}b}=1+\frac{bx}{w}$ gives
\begin{align*}
   \Psi^{\star}(x)&=s^{\ast}x-\frac{w}{b^{2}}\lp\frac{bx}{w}-s^{\ast}b\rp
   =s^{\ast}\lp x+\frac wb\rp-\frac xb\\
   &=\frac{w}{b^{2}}\lB\lp1+\frac{bx}{w}\rp\log\lp1+\frac{bx}{w}\rp
      -\frac{bx}{w}\rB
   =\frac{w}{b^{2}}h\lp\frac{bx}{w}\rp .
\end{align*}
The inequality $h(z)\ge\frac{z^{2}}{2(1+z/3)}$ is Bennett's classical
comparison, and substituting $z=bx/w$ turns it into
$\frac{w}{b^{2}}h(\frac{bx}{w})\ge\frac{x^{2}}{2(w+\frac{bx}{3})}$.
\Cref{thm:master}(b) then gives \eqref{eq:bound-rung3}.

\emph{Rung $3'$.} Under the moment condition of \Cref{rem:rung3prime}, for
$0\le s<1/\tau$,
\begin{align*}
   \E_{\Prob_1}\lB e^{sY}\mid\Fc_{n-1}\rB
   &\le\ 1+\frac{s^{2}}{2}\E_{\Prob_1}\lB Y^{2}\mid\Fc_{n-1}\rB
     +\sum_{k\ge3}\frac{s^{k}}{k!}
       \E_{\Prob_1}\lB(Y_+)^{k}\mid\Fc_{n-1}\rB\\
   &\le\ 1+\frac{w_ns^{2}}{2}\sum_{k\ge2}(\tau s)^{k-2}
   =1+\frac{w_ns^{2}}{2(1-\tau s)} ,
\end{align*}
and $\log(1+u)\le u$ gives the majorant. For its transform put
$\Psi(s)=\frac{ws^{2}}{2(1-\tau s)}$ and
$u:=\tau s$; then $\Psi'(s)=\frac{w}{\tau}\frac{u(2-u)}{2(1-u)^{2}}$, and with
$\varsigma:=\sqrt{1+\frac{2\tau x}{w}}$ the choice $u=1-\varsigma^{-1}$ gives
$1-u=\varsigma^{-1}$ and $u(2-u)=1-\varsigma^{-2}=\frac{2\tau x}{w}
\varsigma^{-2}$, hence $\Psi'(s)=x$; this is the stated $s^{\ast}$, and
substituting back gives
$\Psi^{\star}(x)=\frac{w}{\tau^{2}}\gamma(\frac{\tau x}{w})$ with
$h_1(z)=1+z-\sqrt{1+2z}$. Finally $h_1(z)\ge\frac{z^{2}}{2(1+z)}$ gives
the Bernstein form. \qed

\subsection{\texorpdfstring{Rungs 4 and $4'$}{Rungs 4 and 4'}}
\label{app:rung4}

\noindent\emph{Proof of \Cref{cor:rung4}.}
The majorant is \Cref{hyp:subgauss} itself, and
$\sum_{n\le t}\frac{\sigma_n^{2}s^{2}}{2}\le\frac{V_ts^{2}}{2}$. For $x\ge0$,
\[
   \lp\Psi^{(t)}\rp^{\star}(x)
   =\sup_{s\ge0}\lC sx-\frac{V_ts^{2}}{2}\rC=\frac{x^{2}}{2V_t} ,
\]
attained at $s^{\ast}=x/V_t\ge0$ (for $V_t=0$ read the supremum as $+\infty$
for $x>0$ and $0$ for $x=0$). \Cref{thm:master}(b) gives
\eqref{eq:bound-rung4} and \Cref{thm:master}(d) gives \eqref{eq:unif-rung4}.
For \eqref{eq:stop-rung4}, insert $C_t=tc$, $\Psi^{(t)}(s)=t\sigma^{2}s^{2}/2$
and $s=c/\sigma^{2}$ into \eqref{eq:stop}: the condition
$s\,tc-t\sigma^{2}s^{2}/2\ge s\ell+\log\frac1\delta$ becomes
$t\frac{c^{2}}{2\sigma^{2}}\ge\frac{c\ell}{\sigma^{2}}+\log\frac1\delta$, i.e.\
$t\ge\frac{2\ell}{c}+\frac{2\sigma^{2}}{c^{2}}\log\frac1\delta$. \qed

\noindent\emph{Proof of \Cref{cor:rung4p}.}
On $[0,s_{\max})$ the map $\varphi_n$ is finite and twice differentiable with
$\varphi_n''(s)=\Var_{Q_{n,s}}(L_n)\le v_n$, by differentiating
\eqref{eq:cgf} twice under the integral sign, which is legitimate on the
interior of the domain of finiteness. Since $\varphi_n(0)=0$ and
$\varphi_n'(0)=-(\E_{\Prob_1}[L_n\mid\Fc_{n-1}]-c_n)\le0$ by
\eqref{eq:anchor}, Taylor's formula with integral remainder gives
\[
   \varphi_n(s)=s\varphi_n'(0)+\int_0^{s}(s-v)\varphi_n''(v)\,dv
   \ \le\ \frac{v_ns^{2}}{2} .
\]
For the transform, $x\mapsto\sup_{0\le s<s_{\max}}\{sx-\frac{V_ts^{2}}{2}\}$
is attained at $s=x/V_t$ when $x\le s_{\max}V_t$, with value
$\frac{x^{2}}{2V_t}$, and is otherwise increasing on $[0,s_{\max})$, so the
supremum is the limit $s\uparrow s_{\max}$, namely
$s_{\max}x-\frac{V_ts_{\max}^{2}}{2}$. In the first case
$\frac{x^{2}}{2V_t}=\frac x2\cdot\frac{x}{V_t}$; in the second
$V_ts_{\max}\le x$ gives
$s_{\max}x-\frac{V_ts_{\max}^{2}}{2}\ge\frac{s_{\max}x}{2}$. Both are at least
$\frac x2\min\{\frac{x}{V_t},s_{\max}\}$, which is
\eqref{eq:bound-rung4p}. \qed

\subsection{Rung 5}
\label{app:rung5}

\noindent\emph{Proof of \Cref{cor:rung5}.}

\emph{The majorants.} Write $Y:=c_n-L_n$, which by \eqref{eq:h5} takes values in
$[A_n,B_n]$ with $A_n:=c_n-\overline L_n=d_n-r_n$ and $B_n:=c_n-\underline
L_n=d_n$, an interval of length $r_n$, and $\E_{\Prob_1}[Y\mid\Fc_{n-1}]\le0$.
Hoeffding's lemma (\Cref{thm:app-hoeffding}) gives
$\varphi_n(s)\le s\E_{\Prob_1}[Y\mid\Fc_{n-1}]+\frac{r_n^{2}s^{2}}{8}
\le\frac{r_n^{2}s^{2}}{8}$, which is $\psi_n^{\mathrm{quad}}$. For
$\psi_n^{\mathrm{exact}}$, convexity of $y\mapsto e^{sy}$ gives the chord bound
$e^{sy}\le\frac{B_n-y}{r_n}e^{sA_n}+\frac{y-A_n}{r_n}e^{sB_n}$ on $[A_n,B_n]$,
so
\[
   \E_{\Prob_1}\lB e^{sY}\mid\Fc_{n-1}\rB
   \ \le\ \frac{B_n-\E_{\Prob_1}[Y\mid\Fc_{n-1}]}{r_n}e^{sA_n}
        +\frac{\E_{\Prob_1}[Y\mid\Fc_{n-1}]-A_n}{r_n}e^{sB_n} .
\]
For $s\ge0$ we have $e^{sB_n}\ge e^{sA_n}$, so the right-hand side is
nondecreasing in $\E_{\Prob_1}[Y\mid\Fc_{n-1}]$ and may be evaluated at $0$,
giving $y_ne^{sA_n}+(1-y_n)e^{sB_n}=e^{sd_n}(y_ne^{-sr_n}+1-y_n)$ since
$B_n/r_n=y_n$. Taking logarithms gives $\psi_n^{\mathrm{exact}}$; that
$\psi^{\mathrm{exact}}_n\le\psi^{\mathrm{quad}}_n$ is Hoeffding's lemma applied
to the two-point law itself.

\emph{Part (a).} $\sum_{n\le t}\frac{r_n^{2}s^{2}}{8}\le\frac{R_t^{2}s^{2}}{8}$,
whose transform is $\sup_{s\ge0}\{sx-\frac{R_t^{2}s^{2}}{8}\}
=\frac{2x^{2}}{R_t^{2}}$, attained at $s^{\ast}=\frac{4x}{R_t^{2}}$.

\emph{Part (b), the aggregate.} Let $r\ge\max_{n\le t}r_n$ and write
$\tilde y_n:=d_n/r\in[0,1]$. Since $Y$ takes values in
$[d_n-r_n,d_n]\subseteq[d_n-r,d_n]$, the chord bound may be run on the larger
interval instead, which gives, by the same two lines,
$\varphi_n(s)\le sd_n+\log(\tilde y_ne^{-sr}+1-\tilde y_n)$. Since $\sum_{n\le t}d_n=tr\bar y_t$ and, by concavity of
$\log$ and Jensen's inequality applied to the uniform average over $n\le t$,
\[
   \sum_{n\le t}\log\lp\tilde y_ne^{-sr}+1-\tilde y_n\rp
   \ \le\ t\log\lp\frac1t\sum_{n\le t}
      \lp\tilde y_ne^{-sr}+1-\tilde y_n\rp\rp
   =t\log\lp\bar y_te^{-sr}+1-\bar y_t\rp ,
\]
the stated $\Psi^{(t)}$ dominates $\sum_{n\le t}\psi_n^{\mathrm{exact}}$.

\emph{Part (b), the transform.} Put $\Psi(s)=t[sr\bar y+\log(\bar
ye^{-sr}+1-\bar y)]$ with $\bar y=\bar y_t$. Then
\[
   \Psi'(s)=t\lB r\bar y-r\,\frac{\bar ye^{-sr}}
      {\bar ye^{-sr}+1-\bar y}\rB ,
\]
and setting $\bar x:=\frac{\bar ye^{-sr}}{\bar ye^{-sr}+1-\bar y}\in(0,\bar y)$ the
equation $\Psi'(s)=x$ reads $\bar x=\bar y-\frac{x}{tr}$. Solving
$\bar x=\frac{\bar ye^{-sr}}{\bar ye^{-sr}+1-\bar y}$ for $s$ gives
$e^{-sr}=\frac{\bar x(1-\bar y)}{\bar y(1-\bar x)}$, i.e.\ \eqref{eq:sstar-rung5b}
with $\bar x=x_t$ at $x=\Delta_t$. Substituting back, and using
$\bar ye^{-s^{\ast}r}+1-\bar y=\frac{1-\bar y}{1-\bar x}$,
\[
   \Psi^{\star}(x)=s^{\ast}x-\Psi(s^{\ast})
   =-s^{\ast}rt\bar x+t\log\frac{1-\bar x}{1-\bar y}
   =t\lB \bar x\log\frac{\bar x}{\bar y}
      +(1-\bar x)\log\frac{1-\bar x}{1-\bar y}\rB
   =t\,\kl(\bar x\|\bar y) ,
\]
which is the stated rate function. \Cref{thm:master}(b) gives
\eqref{eq:bound-rung5a} and \eqref{eq:bound-rung5b}, and Pinsker's inequality
(\Cref{thm:app-pinsker}) gives the last claim. \qed

\begin{proposition}[The binary rate of \Cref{cor:rung5}(b) is exact]
\label{prp:bdd-sharp}
Fix $\underline L<0<\overline L$, put $r:=\overline L-\underline L$, let
$y\in(0,1)$ with $c:=\underline L+yr>0$, and fix $\alpha\in(0,1)$. There exist
$P_0,P_1$ and an i.i.d.\ betting sequence satisfying \Cref{hyp:bounded} with
these constants and $\E_{P_1}[\log E]=c$ --- the extremal \evar{} is two-point,
supported on $\{e^{\underline L},e^{\overline L}\}$ --- for which
\[
   \lim_{t\to\infty}-\frac1t\log\gamma_t(\alpha)
   \ =\ \kl\lp-\tfrac{\underline L}{r}\,\Big\|\,y\rp
   \ =\ \lim_{t\to\infty}\kl\lp\bar y_t-\tfrac{\Delta_t}{tr}\,
        \Big\|\,\bar y_t\rp .
\]
Consequently no bound depending only on
$(\underline L,\overline L,c)$ can have a larger exponent.
\end{proposition}

\noindent\emph{Proof.}
Let $E$ take the values $e^{\overline L}$ and $e^{\underline L}$ with
$P_1$-probabilities $y$ and $1-y$, so $\E_{P_1}[\log E]=y\overline
L+(1-y)\underline L=\underline L+yr=c$. Let $P_0$ give the same two values
probabilities $y_0$ and $1-y_0$ with $y_0\in(0,1)$ chosen so that
$y_0e^{\overline L}+(1-y_0)e^{\underline L}\le1$, which is possible because
$e^{\underline L}<1$; then $E$ is an \evar{} and $P_1\ll P_0$. Under $P_1$,
$\log W_t=t\underline L+r\,S_t$ with $S_t\sim\mathrm{Bin}(t,y)$, so
\[
   \gamma_t(\alpha)=P_1\lp S_t\le\frac{\ell-t\underline L}{r}\rp ,
   \qquad
   \frac1t\cdot\frac{\ell-t\underline L}{r}\ \longrightarrow\
   -\frac{\underline L}{r}\ \in(0,y) ,
\]
the last because $0<c=\underline L+yr$ gives $-\underline L/r<y$, and
$-\underline L/r>0$. Cram\'er's theorem for the binomial
(\Cref{thm:app-cramer}) gives
$-\frac1t\log\gamma_t(\alpha)\to\kl(-\frac{\underline L}{r}\|y)$. On the other
hand $\bar y_t=y$ and $\frac{\Delta_t}{tr}=\frac{(tc-\ell)_+}{tr}\to\frac cr$,
so $\bar y_t-\frac{\Delta_t}{tr}\to y-\frac cr=-\frac{\underline L}{r}$, and the
two limits agree. \qed

\subsection{Rung 6}
\label{app:rung6}

\noindent\emph{Proof of \Cref{cor:rung6}.}
Under \Cref{hyp:iid} the increments $L_n=\log E_n$ are i.i.d.\ under $\Prob_1$
and independent of $\Fc_{n-1}$, so
$\varphi_n(s)=\log\E_{P_1}[e^{-s(\log E-c)}]=sc-\Lambda_s(E)$ is deterministic
and \eqref{eq:majorant} holds with equality; no smaller majorant exists.
\Cref{cor:iid} applies with $\psi=\varphi$ and $C_t=tc$, and
\[
   t\,\psi^{\star}\lp\frac{\Delta_t}{t}\rp
   =\sup_{s\ge0}\lC s\Delta_t-t\lp sc-\Lambda_s(E)\rp\rC
   =\sup_{s\ge0}\lC t\Lambda_s(E)-s\lp tc-\Delta_t\rp\rC
   =\sup_{s\ge0}\lC t\Lambda_s(E)-s\ell\rC
\]
whenever $tc\ge\ell$, which is \eqref{eq:bound-rung6}; when $tc<\ell$ both sides
are trivial. Differentiating, $\frac{d}{ds}\Lambda_s(E)=\E_{Q_s}[\log E]$ with
$Q_s$ as in \eqref{eq:sstar-rung6}, so the stationarity condition
$t\frac{d}{ds}\Lambda_s(E)=\ell$ is \eqref{eq:sstar-rung6}. The limit \eqref{eq:cramerlimit} is Cram\'er's theorem
(\Cref{thm:app-cramer}) applied to the i.i.d.\ sequence $(\log E_n)_{n\in\N}$
at the level $\ell/t\to0$, together with $\Lambda(E)=I(0)$; and
\eqref{eq:cramerlimit-seq} then follows from \eqref{eq:ordering}. \qed

\noindent\emph{Proof of \Cref{lem:positive}.}
Write $\phi(s):=\E_{P_1}[E^{-s}]\in(0,\infty]$, so that
$\Lambda_s(E)=-\log\phi(s)$ and $\phi(0)=1$.

\ref{it:p1}$\Rightarrow$\ref{it:p2}: if $\phi(s)<1$ for some $s>0$ then in
particular $\phi(s)<\infty$, which is \Cref{hyp:cramer}; and conditional Jensen
for the convex map $x\mapsto e^{-sx}$ gives $\phi(s)\ge e^{-sc}$, so
$e^{-sc}<1$ and hence $c>0$.

\ref{it:p2}$\Rightarrow$\ref{it:p1}: by H\"older's inequality
$\phi(s)\le\phi(s_0)^{s/s_0}$ for $s\in[0,s_0]$, so $\phi$ is finite there.
The difference quotients
$\frac{\phi(s)-1}{s}=\E_{P_1}\lB\frac{e^{-s\log E}-1}{s}\rB$ decrease as
$s\downarrow0$ and are dominated above by their value at $s_0$, which is
integrable; monotone convergence therefore gives the right derivative
$\phi'(0^{+})=-\E_{P_1}[\log E]=-c<0$, so $\phi(s)<1$ for all small $s>0$.

For the last claim, Jensen gives $\phi(s)\ge e^{-sc}$ for every $s\in\R$, so
$-\log\phi(s)\le sc<0$ when $s<0$ and $c>0$; since $-\log\phi(0)=0$, the
supremum over $\R$ is attained on $[0,\infty)$. \qed

\subsection{How the rungs are related}
\label{app:implications}

\begin{lemma}[Implications between the hypotheses]\label{lem:implications}
Let $(c_n)_{n\in\N}$ be a drift certificate. Then:
\begin{enumerate}[label=(\alph*),leftmargin=2.4em]
\item \Cref{hyp:bounded} implies \Cref{hyp:subgauss} with
      $\sigma_n^{2}=\frac{r_n^{2}}{4}$, and \Cref{hyp:freedman} with
      $b_n=r_n$ and $w_n=r_n^{2}$ (and with $b_n=d_n$ whenever $d_n>0$).
\item \Cref{hyp:freedman} implies \Cref{hyp:subgauss} with
      $\sigma_n^{2}=\max\{e\,w_n,\,2b_n^{2}\}$; it implies \Cref{hyp:crash} for
      every $a>0$, with $K_n=e^{ab_n}$, $p=2$ and $M_{2,n}=w_n$; and it implies
      \Cref{hyp:cramer} for every $s_0>0$, with $\rho_n=e^{-s_0c_n+s_0b_n}$;
      and it implies \Cref{hyp:crash2} for every $a>0$, with the same
      $K_n=e^{ab_n}$ and $w_{0,n}=w_n$. Composing with (a),
      \Cref{hyp:bounded} implies \Cref{hyp:crash2} as well.
\item \Cref{hyp:subgauss} implies \Cref{hyp:cramer} for every $s_0>0$, with
      $\rho_n=e^{-s_0c_n+\sigma_n^{2}s_0^{2}/2}$, and it implies the crash-tail
      half \eqref{eq:h2} of \Cref{hyp:crash} for every $a>0$, with
      $K_n=e^{a^{2}\sigma_n^{2}/2}$; it does \emph{not} imply the moment half
      \eqref{eq:h2b}, at any admissible exponent $p$.
\item \Cref{hyp:crash} implies \Cref{hyp:cramer} for every $s_0\in(0,a)$, with
      $\rho_n=\exp(-s_0c_n+\psi_n(s_0))$ and $\psi_n$ as in
      \eqref{eq:maj-rung2}.
\item \Cref{hyp:crash2} implies \Cref{hyp:crash} with $p=2$ and
      $M_{2,n}=w_{0,n}$, and it implies the Bernstein condition of
      \Cref{rem:rung3prime} with $\tau=\frac1a$ and
      $w_n=\max\{w_{0,n},\frac{2K_n}{a^{2}}\}$.
\item \Cref{hyp:tiltvar} with $s_{\max}=\infty$ implies \Cref{hyp:subgauss}
      with $\sigma_n^{2}=v_n$, and \Cref{hyp:bounded} implies
      \Cref{hyp:tiltvar} with $s_{\max}=\infty$ and $v_n=\frac{r_n^{2}}{4}$;
      composing the two is one proof of Hoeffding's lemma. Neither implication
      reverses.
\item No implication holds beyond those listed and their compositions. In
      particular \Cref{hyp:subgauss} implies
      neither \Cref{hyp:freedman} nor \Cref{hyp:crash}, and \Cref{hyp:crash}
      implies neither \Cref{hyp:subgauss} nor \Cref{hyp:freedman}: rungs~2
      and~4 are incomparable. Counterexamples are in \Cref{rem:notachain}.
\end{enumerate}
\end{lemma}

\noindent\emph{Proof.}
(a) The first claim is Hoeffding's lemma (\Cref{thm:app-hoeffding}), as in the
proof of \Cref{cor:rung5}. For the second, $c_n-L_n\le c_n-\underline L_n=d_n\le
r_n$ gives the floor, and $|L_n-c_n|\le\max\{d_n,\overline L_n-c_n\}\le r_n$
gives $\E_{\Prob_1}[(L_n-c_n)^{2}\mid\Fc_{n-1}]\le r_n^{2}$.

(b) For the first claim, split at $s=1/b_n$. If $sb_n\le1$ then, by
$e^{z}-z-1\le\frac{z^{2}e^{z}}{2}$ for $z\ge0$ and \eqref{eq:maj-rung3},
\[
   \varphi_n(s)\ \le\ \frac{w_n}{b_n^{2}}\lp e^{sb_n}-sb_n-1\rp
   \ \le\ \frac{w_n}{b_n^{2}}\cdot\frac{(sb_n)^{2}e}{2}
   =\frac{e\,w_n\,s^{2}}{2} .
\]
If $sb_n\ge1$ then the floor gives $e^{-s(L_n-c_n)}\le e^{sb_n}$ pointwise,
hence $\varphi_n(s)\le sb_n\le s^{2}b_n^{2}=\frac{2b_n^{2}s^{2}}{2}$. Taking the
larger of the two constants covers all $s\ge0$. The second claim: the floor
gives $\Prob_1(L_n\le c_n-x\mid\Fc_{n-1})=0$ for $x>b_n$, so the crash tail
holds with $K_n=e^{ab_n}$ for every $a>0$ and every $x\ge0$, and
$\E_{\Prob_1}[((L_n-c_n)_+)^{2}\mid\Fc_{n-1}]\le
\E_{\Prob_1}[(L_n-c_n)^{2}\mid\Fc_{n-1}]\le w_n$. The third: the floor gives
$E_n^{-s_0}=e^{-s_0L_n}\le e^{-s_0c_n+s_0b_n}$ pointwise.

(c) Evaluating \eqref{eq:h4} at $s=s_0$ gives
$\E_{\Prob_1}[E_n^{-s_0}\mid\Fc_{n-1}]=e^{-s_0c_n+\varphi_n(s_0)}\le
e^{-s_0c_n+\sigma_n^{2}s_0^{2}/2}$. For the crash tail, Markov's inequality at
tilt $a$ gives $\Prob_1(L_n\le c_n-x\mid\Fc_{n-1})\le e^{-ax+\varphi_n(a)}\le
e^{a^{2}\sigma_n^{2}/2}e^{-ax}$. The last claim is the second counterexample of
\Cref{rem:notachain}, whose winning-side tail has no moment of any order
$p>1$.

(d) By \Cref{lem:maj-crash}, $\varphi_n(s_0)\le\psi_n(s_0)<\infty$ for
$s_0<a$, and $\E_{\Prob_1}[E_n^{-s_0}\mid\Fc_{n-1}]
=e^{-s_0c_n+\varphi_n(s_0)}$.

(e) The first claim is $((L_n-c_n)_+)^{2}\le(L_n-c_n)^{2}$; the second is the
layer-cake bound $\E_{\Prob_1}[((c_n-L_n)_+)^{k}\mid\Fc_{n-1}]\le
\frac{K_nk!}{a^{k}}=\frac{k!}{2}\cdot\frac{2K_n}{a^{2}}\cdot
\lp\frac1a\rp^{k-2}$ from the proof of \Cref{cor:rung2pp}, together with
\eqref{eq:h2c}. The addition to (b) is the same crash-tail computation
together with the observation that the second half of \eqref{eq:h3} is
\eqref{eq:h2c} verbatim.

(f) The first claim is the proof of \Cref{cor:rung4p}. For the second, under
\Cref{hyp:bounded} every tilted law $Q_{n,s}$ is supported in the same interval
of length $r_n$, and a random variable supported in an interval of length $r$
has variance at most $\frac{r^{2}}{4}$.

(g) See \Cref{rem:notachain}. \qed

\subsection{The observable rung}
\label{app:observable}

\noindent\emph{Proof of \Cref{thm:empbern}.}
Put $\xi_n:=(c_n-L_n)/\bar r$. \Cref{hyp:ceiling} says exactly that $\xi_n\ge-1$, so
$\E_{\Prob_1}[\xi_n\mid\Fc_{n-1}]$ is well defined in $(-\infty,+\infty]$, and
\eqref{eq:anchor} makes it $\le0$; in particular $\xi_n<\infty$
$\Prob_1$-a.s. No upper bound on $\xi_n$ is needed. By \Cref{thm:app-fan} and
the tower property,
\[
   \E_{\Prob_1}\lB
     e^{\lambda\xi_n-\psi_{\mathrm E}(\lambda)\xi_n^{2}}\mid\Fc_{n-1}\rB
   \ \le\ 1+\lambda\,\E_{\Prob_1}\lB\xi_n\mid\Fc_{n-1}\rB\ \le\ 1 ,
\]
so $N_t:=\prod_{n\le t}\exp\{\lambda\xi_n-\psi_{\mathrm E}(\lambda)\xi_n^{2}\}$
is a nonnegative $\Prob_1$-supermartingale with $N_0=1$. Ville's inequality
(\Cref{thm:app-ville}) gives $\Prob_1(\sup_tN_t\ge\frac1\delta)\le\delta$; on
the complementary event, simultaneously for all $t$,
\[
   \lambda\sum_{n\le t}\xi_n
   \ <\ \log\tfrac1\delta+\psi_{\mathrm E}(\lambda)\sum_{n\le t}\xi_n^{2} ,
   \qquad\text{i.e.}\qquad
   \frac{\lambda}{\bar r}\lp\sum_{n\le t}c_n-\log W_t\rp
   \ <\ \log\tfrac1\delta+\frac{\psi_{\mathrm E}(\lambda)}{\bar r^{2}}\widehat V_t .
\]
Rearranging and using $\sum_{n\le t}c_n\ge C_t$ gives the first inequality of
\eqref{eq:eb-floor}; the second is $\psi_{\mathrm E}(\lambda)\le
\frac{\lambda^{2}}{2(1-\lambda)}$ from \Cref{thm:app-fan}. Finally the
right-hand side of \eqref{eq:eb-floor} is at least $\ell$ exactly when the
condition defining $\widehat T$ holds, and \eqref{eq:eb-floor} is strict, so
$\log W_{\widehat T}>\ell$ and hence $\tau_\alpha\le\widehat T$ on that
event. \qed

\noindent\emph{Proof of \eqref{eq:eb-tuned}.}
Write $D(\lambda)=\frac{\bar r\Lambda_\delta}{\lambda}
+\frac{\lambda\bar V}{2(1-\lambda)\bar r}$.
Then $D'(\lambda)=-\frac{\bar r\Lambda_\delta}{\lambda^{2}}
+\frac{\bar V}{2\bar r(1-\lambda)^{2}}$, which vanishes exactly when
$\frac{1-\lambda}{\lambda}=\sqrt{\frac{\bar V}{2\bar r^{2}\Lambda_\delta}}$,
that is at $\lambda^{\ast}$ as stated, and $D$ is strictly convex on $(0,1)$
with $D\to\infty$ at both ends, so this is the minimum. At $\lambda^{\ast}$,
$\frac1{\lambda^{\ast}}=1+\frac{\sqrt{\bar V}}{\bar r\sqrt{2\Lambda_\delta}}$
gives $\frac{\bar r\Lambda_\delta}{\lambda^{\ast}}
=\bar r\Lambda_\delta+\sqrt{\frac{\bar V\Lambda_\delta}{2}}$, and
$\frac{\lambda^{\ast}}{1-\lambda^{\ast}}
=\frac{\bar r\sqrt{2\Lambda_\delta}}{\sqrt{\bar V}}$ gives
$\frac{\lambda^{\ast}\bar V}{2(1-\lambda^{\ast})\bar r}
=\sqrt{\frac{\bar V\Lambda_\delta}{2}}$; adding, $D(\lambda^{\ast})
=\bar r\Lambda_\delta+2\sqrt{\bar V\Lambda_\delta/2}
=\bar r\Lambda_\delta+\sqrt{2\bar V\Lambda_\delta}$. \qed

\subsection{Composite alternatives and the price of the threshold}
\label{app:composite}

\noindent\emph{Proof of \Cref{prp:composite}.}
Fix $\Prob\in\Hc_1$. By hypothesis $(c_n,C_t)$ is a drift certificate and the
stated constants furnish a majorant, both under $\Prob$, so
\Cref{thm:master}(b) applies with $\Prob_1$ replaced by $\Prob$ and yields the
displayed bound, whose right-hand side does not depend on $\Prob$. Taking the
supremum over $\Prob\in\Hc_1$ on the left finishes. \qed

\noindent\emph{Proof of \Cref{prp:price}.}
Write $s^{\ast}=s^{\ast}(c)$, so that $\psi^{\star}(c)=s^{\ast}c-\psi(s^{\ast})$.
For every $x\ge0$, taking $s=s^{\ast}$ in the supremum defining
$\psi^{\star}(x)$,
\[
   \psi^{\star}(x)\ \ge\ s^{\ast}x-\psi(s^{\ast})
   =\psi^{\star}(c)+s^{\ast}(x-c) ,
\]
which is the supergradient inequality and needs neither differentiability nor
strict convexity. With $h:=\ell/t\in[0,c]$ and $x=c-h$ this reads
$\psi^{\star}(c-h)\ge\psi^{\star}(c)-h\,s^{\ast}$.
Multiplying by $t$ and using $\Delta_t/t=c-\ell/t$ gives
$t\,\psi^{\star}(\Delta_t/t)\ge t\,\psi^{\star}(c)-s^{\ast}(c)\ell$, and
\Cref{cor:iid} turns this into \eqref{eq:price}; the identity
$e^{s^{\ast}\ell}=\alpha^{-s^{\ast}}$ is the definition of $\ell$. If
$\psi^{\star}$ is twice differentiable the loss relative to \Cref{cor:iid} is
$t[\psi^{\star}(c-h)-\psi^{\star}(c)+hs^{\ast}(c)]
=\frac{\ell^{2}}{2t}(\psi^{\star})''(\zeta)$ for some
$\zeta\in(c-h,c)$, and it is $0$ whenever $\psi^{\star}$ is affine on
$[c-h,c]$, as it is for \Cref{cor:rung1}. \qed

\section{Results quoted from the literature}
\label{app:concentration}

Each result is stated in the notation of this paper. Nothing here is proved;
a reference is given for each, together with a note on where it is used.

\begin{theorem}[Ville's inequality; {\cite{Vil39}}, {\cite[Lem.~1]{How20}}]
\label{thm:app-ville}
Let $(U_t)_{t\in\N_0}$ be a nonnegative supermartingale on a filtered
probability space $(\Omega,\Fc,(\Fc_t)_{t\in\N_0},\Prob)$ with $U_0\le1$
$\Prob$-a.s. Then
\[
   \Prob\lp\sup_{t\in\N_0}U_t\ \ge\ x\rp\ \le\ \frac1x
   \qquad\text{for every }x\in(1,\infty) .
\]
\end{theorem}

\noindent
Used three times: under $\Prob_0$ for the type-I guarantee
(\Cref{thm:typeI}); under $\Prob_1$ for the time-uniform form
\Cref{thm:master}(d), where the supermartingale is the tilted process
\eqref{eq:tilted}; and again under $\Prob_1$ in \Cref{thm:empbern}.

\begin{theorem}[Hoeffding's lemma; {\cite{Hoe63}},
{\cite[Lem.~2.2]{BLM13}}]\label{thm:app-hoeffding}
Let $Y$ be a random variable with $A\le Y\le B$ $\Prob$-a.s. Then for every
$s\in\R$,
\[
   \log\E_{\Prob}\lB e^{sY}\rB\ \le\ s\,\E_{\Prob}[Y]
      +\frac{s^{2}(B-A)^{2}}{8} .
\]
The same statement holds conditionally on a sub-$\sigma$-algebra, with
$\E_{\Prob}[\,\cdot\,]$ replaced throughout by $\E_{\Prob}[\,\cdot\mid\Fc']$ and
the bounds $A,B$ allowed to be $\Fc'$-measurable.
\end{theorem}

\noindent
Used for the quadratic majorant of \Cref{cor:rung5} and for
\Cref{lem:implications}(a). It is the only place where the \emph{upper} bound
$\overline L_n$ is used at all; every other rung constrains the crash side only.

\begin{theorem}[Cram\'er's theorem; {\cite{Cra38}}, {\cite[Thm.~2.2.3]{DZ10}}]
\label{thm:app-cramer}
Let $(Y_n)_{n\in\N}$ be i.i.d.\ real random variables with
$\Lambda_Y(s):=\log\E[e^{sY_1}]<\infty$ for $s$ in a neighbourhood of $0$, and
let $I:=\Lambda_Y^{\star}$ be the Legendre transform of $\Lambda_Y$. Then for
every $x<\E[Y_1]$,
\[
   \lim_{t\to\infty}\ \frac1t\log
   \Prob\lp\frac1t\sum_{n\le t}Y_n\le x\rp\ =\ -I(x) .
\]
For $Y_1\in\{0,1\}$ with $\Prob(Y_1=1)=y$ one has $I(x)=\kl(x\|y)$.
\end{theorem}

\noindent
Used for the sharpness statements \Cref{prp:bdd-sharp} and, through
\cite[Prop.~2.23]{CEIL}, for \Cref{cor:rung6}.

\begin{theorem}[Pinsker's inequality for the binary divergence;
{\cite[Lem.~2.3]{Wai19}}]\label{thm:app-pinsker}
For all $x,y\in[0,1]$, $\ \kl(x\|y)\ \ge\ 2(x-y)^{2}$.
\end{theorem}

\noindent
Used to compare the two forms of \Cref{cor:rung5}; it shows that part (b)
implies the Azuma--Hoeffding exponent in the homogeneous case, so that part (a)
earns its place only through the inhomogeneous $R_t^{2}\le tr^{2}$.

\begin{theorem}[Donsker--Varadhan variational formula; {\cite{DV75}},
{\cite[Cor.~4.14]{BLM13}}]\label{thm:app-dv}
Let $\Prob$ be a probability measure and $\zeta$ a measurable function with
$\E_{\Prob}[e^{\zeta}]<\infty$, and let $Q^{\zeta}$ be the Gibbs measure
$dQ^{\zeta}:=e^{\zeta}\E_{\Prob}[e^{\zeta}]^{-1}d\Prob$. Then
\[
   \log\E_{\Prob}\lB e^{\zeta}\rB
   \ =\ \sup_{Q\ll\Prob}\lC\E_Q[\zeta]-\KL(Q\|\Prob)\rC ,
\]
the supremum being over probability measures $Q$ with $\KL(Q\|\Prob)<\infty$.
Moreover, for every such $Q$ with $\E_Q[\zeta]$ well defined,
\begin{equation}\label{eq:dv-gap}
   \KL(Q\|\Prob)-\E_Q[\zeta]+\log\E_{\Prob}\lB e^{\zeta}\rB
   \ =\ \KL\lp Q\,\|\,Q^{\zeta}\rp\ \ge\ 0 ,
\end{equation}
so the supremum is attained at $Q^{\zeta}$ and there only.
\end{theorem}

\noindent
Used in the proof of \Cref{prp:duality}; \eqref{eq:dv-gap} is the elementary
identity $\log\frac{dQ}{d\Prob}=\log\frac{dQ}{dQ^{\zeta}}+\zeta
-\log\E_{\Prob}[e^{\zeta}]$ integrated against $Q$, and it is what supplies the
equality case used there. Note that this variational formula is the
\emph{tool}; the statement of \Cref{prp:duality} --- that a Cram\'er rate
function equals an information projection onto a half-space --- is the
one-dimensional specialisation of Sanov's theorem combined with the contraction
principle \cite[\S6.2]{DZ10}.

\begin{theorem}[Fan's inequality; {\cite[Lem.~4.1 and (4.12)]{FGL15}}]
\label{thm:app-fan}
Let $\psi_{\mathrm E}(\lambda):=-\log(1-\lambda)-\lambda$. Then for every
$\lambda\in[0,1)$ and every $\xi\ge-1$,
\[
   \exp\lC\lambda\xi-\psi_{\mathrm E}(\lambda)\,\xi^{2}\rC
   \ \le\ 1+\lambda\xi ,
   \qquad\text{and}\qquad
   \psi_{\mathrm E}(\lambda)\ \le\ \frac{\lambda^{2}}{2(1-\lambda)} .
\]
\end{theorem}

\noindent
The engine of \Cref{thm:empbern}. Note that only the one-sided restriction
$\xi\ge-1$ is imposed, which is \Cref{hyp:ceiling}; nothing bounds $\xi$ above,
which is why the crash side may be arbitrarily heavy. Note also the
normalisation: here $\xi$ is bounded \emph{below}, whereas in
\Cref{thm:app-fgl} it is bounded above, the two being related by
$\xi\rightsquigarrow-\xi$. Being pointwise, the inequality holds for any
$[0,1)$-valued random $\lambda$; what \emph{predictability} of a sequence
$(\lambda_n)_{n\in\N}$ buys is the conditional-expectation step that turns the
stepwise product into a supermartingale, which is what \Cref{rem:eb-lil}
uses.

\begin{theorem}[Hoeffding's inequality for supermartingales; {\cite{FGL12}}]
\label{thm:app-fgl}
Let $(\xi_n)_{n\le t}$ be a supermartingale difference sequence with
$\xi_n\le1$ $\Prob$-a.s.\ and
$\sum_{n\le t}\E[\xi_n^{2}\mid\Fc_{n-1}]\le v^{2}$ $\Prob$-a.s. Then for
$0\le x\le t$,
\[
   \Prob\lp\max_{k\le t}\sum_{n\le k}\xi_n\ \ge\ x\rp
   \ \le\ \lB\lp\frac{v^{2}}{x+v^{2}}\rp^{x+v^{2}}
      \lp\frac{t}{t-x}\rp^{t-x}\rB^{\frac{t}{t+v^{2}}}
   \ \le\ \exp\lp-\frac{x^{2}}{2(v^{2}+x/3)}\rp ,
\]
with the convention $(+\infty)^{0}:=1$ at $x=t$.
\end{theorem}

\noindent
Quoted in \Cref{rem:fgl}, where it sharpens \eqref{eq:bound-rung3} under identical
hypotheses. That the \emph{supermartingale} form is available is what makes it
usable here: \Cref{def:main}(a) certifies a one-sided drift, so
$(c_n-L_n)_{n\in\N}$ is a supermartingale difference sequence and not a
martingale difference sequence.

\end{document}